\documentclass[11pt,a4paper,subeqn,reqno]{article}

\usepackage{graphicx}
\usepackage{calc}
\usepackage{amssymb}
\usepackage{epstopdf}

\usepackage{latexsym}

\usepackage{amsmath,amsthm}

\usepackage[mathscr]{eucal}
\usepackage[all]{xy}
\usepackage{mathrsfs}
\usepackage{authblk}
\usepackage{hyperref}
\usepackage{amsfonts}
\usepackage{amsmath}
\usepackage[hyperpageref]{backref}
\usepackage[dvipsnames,usenames]{color}
\usepackage{color,xcolor}
\usepackage{shuffle}
\usepackage{colortbl}
\usepackage{makeidx}
\usepackage{mathptmx}
\usepackage[all]{xy}
\usepackage{shuffle}

\usepackage{tikz-cd}
\usepackage{pifont}

\allowdisplaybreaks
\usepackage[text={169mm,245mm},centering]{geometry}

\newtheorem{theorem}{Theorem}[section] 
\newtheorem{lemma}[theorem]{Lemma}

\newtheorem{proposition}[theorem]{Proposition}
\newtheorem{claim}[theorem]{Claim}

\theoremstyle{definition}

\newtheorem{definition}[theorem]{Definition}
\newtheorem{definition-lemma}[theorem]{Definition-Lemma}
\newtheorem{definition-theorem}[theorem]{Definition-Theorem}

\newtheorem{remark}[theorem]{Remark}

\newcommand{\R}{\mathbb{R}}
\newcommand{\Z}{\mathbb{Z}}

\newcommand{\C}{\emph{C}}

\begin{document}
\title{\textbf{ A Weak $\infty$-Functor in Morse Theory}}

\author[1,2]{Shanzhong Sun\thanks{ Email: sunsz@cnu.edu.cn}}
\author[1]{Chenxi Wang\thanks{ Email: wangcxchelsea@outlook.com}}

\renewcommand\Affilfont{\small}

\affil[1]{Department of Mathematics, Capital Normal University, Beijing 100048, P. R. China}
\affil[2]{Academy for Multidisciplinary Studies, Capital Normal University, Beijing 100048, P. R. China}

\date{}

\maketitle

\begin{abstract}
In the spirit of Morse homology initiated by Witten and Floer,  we construct two $\infty$-categories $\mathcal{A}$ and  $\mathcal{B}$. The weak one $\mathcal{A}$ comes out of the Morse-Samle pairs and their higher homotopies, and the strict one $\mathcal{B}$ concerns the chain complexes of the Morse functions. Based on the boundary structures of the compactified moduli space of gradient flow lines of Morse functions with parameters, we build up a weak $\infty$-functor $\mathcal{F}: \mathcal{A}\rightarrow \mathcal{B}$. Higher algebraic structures behind Morse homology are revealed with the perspective of defects in topological quantum field theory.



\end{abstract}

\setcounter{tocdepth}{2}\tableofcontents

key words: Morse function, 
moduli space of gradient flow lines, compactifications, weak $n/\infty$-category, higher morphism, homotopy

\begin{section}
{Introduction}
\end{section}

The interplay between analysis and topology on Riemannian manifold has mutual benefits.  Ever since the classical work of Morse, the topology of manifolds is tied closely to the analytic objects the primary of which is arguably Morse function. As the modern developments in the calculus of variations and the infinite dimensional Morse theory show, the topological structures such as nontrivial homologies and homotopies provide valuable information and constraints on the critical points of action functionals which amount to the solutions to ordinary differential equations and partial differential equations. On the other hand, the non-degeneracy of the critical points has far reaching implications about the topology of the manifold. The ideas of varying the level sets already reveal the homotopy type of the manifolds which is enough to determine the manifolds homeomorphically in some cases.  With the help of Riemannian metric the gradient vector fields of Morse functions were introduced to the subject and the role of the transversality conditions among the stable and unstable submanifolds was recognized. With such insights, the cell complex structure of the manifold and the attaching maps get fully understood which resulted in the spectacular achievements such as Bott's periodicity theorem and Smale's proof of higher dimensional Poincar\'{e} conjecture.

The next revolutionary idea is due to Witten (\cite{witten82}) who interpreted the Morse theory as a supersymmetric quantum mechanics with the Morse function playing the role of the superpotential. Morse theory is also related to Hodge theory which is another model of the interaction between analysis and topology, and they are placed in the two sectors of the coupling constants of the same theory. In this picture the critical points and the gradient flow lines correspond to the perturbative ground states and the tunneling effects among them. These data gives the Morse-Smale-Witten (MSW) complex whose homology is the true ground state of the theory.

Inspired by this circle of ideas, the moduli space of gradient flow lines and its compactification play a fundamental role in defining in a mathematically rigorous way the Morse homology which is proved to be isomorphic to the singular homology of the manifold. Later it was generalized to Morse homotopy theory (\cite{fukaya93, betz-cohen}) which gives the information about the rational homotopy type of the manifold.

More precisely, given a Morse function $f$ on a smooth manifold $M$, we construct the MSW complex whose generators are the critical points graded by their Morse indices. The boundary operator is given by counting the gradient flow lines connecting the critical points with index discrepancy $1$. It is highly nontrivial that thus defined boundary operator is really a boundary operator, i.e., square zero. It depends on the meticulous analysis of the moduli space of gradient flow lines. The smoothness, compactification, boundary with corner structure and orientation of moduli space with any index difference call for a full understanding. In fact, in the current paper our point is that the moduli space should play an even fundamental role, and we will use them to get finer algebraic structures.

To understand homology and homotopy structures of topological space with extra (e.g., smooth, algebraic or group action) structures, more and more abstract algebraic tools such as differential graded algebras, operads, higher homotopy algebras and infinity categories were introduced and used. It is interesting to notice that all of these abstract algebraic structures appear in recent developments of quantum physics.

Let's take defect as an example. In a quantum field theory, if there are some lower-dimensional region in spacetime manifold $M$ where special phenomena different from its surroundings happen, such region is called a defect.  Most of time the metric nature of the region is irrelevant, so defect is considered to be topological.  It fits quite well with topological quantum field theory (TQFT). Domain wall in ferromagnets in condensed matter is such a case. It turns out many quantum field theories correspond to various categories, and topological defects correspond to morphisms in categories.  This viewpoint provides a conceptually plausible unifying theme to understand symmetries of a given quantum field theory and dualities between different quantum field theories. Please refer to \cite{kapustin, carqueville} to get a flavor.

Even though these higher algebraic structures were invented first in mathematics and later found their way in quantum field theory, geometry and topology also benefit from the interplay. Many recent mathematical developments are attempts to  clarify the underlying higher structures of physical phenomena like defects. Extended TQFT in the sense of Lurie is a prominent one in which weak $n$-category appears naturally(\cite{lurie08}).

To define the MSW complex using gradient flow lines, other than the Riemannian metric we need one Morse function the non-degenerate critical points of which are the generators of the complex. To address the independence problem of the MSW complex on the Morse function, two Morse functions are required and the moduli spaces of gradient flow lines connecting the critical points of both Morse functions are used in the cobordism like arguments to get well-defined Morse homology theory.

In Morse theory and symplectic geometry, there appears already a higher homotopical algebra structure, namely the $A_\infty$-algebra. For Morse theory it is called the Morse homotopy theory, and in symplectic geometry it is called Fukaya category which can be seen as a version of half infinite dimensional Morse theory.  In Morse homotopy theory, several Morse functions are needed and the moduli spaces of gradient flow lines connecting critical points belonging to various differences of Morse functions are used to build up the higher homotopy algebra structure. Roughly, this strong homotopy algebra is a manifestation of the tree structure behind the moduli space. For the references about Morse homotopy theory please refer to \cite{fukaya93, betz-cohen}.

In \cite{gaiotto-moore-witten}, the authors studied a Landau-Ginzburg (LG) model pertaining to a K\"{a}hler manifold with a holomorphic Morse function and developed a web-based formalism for such theories. An $A_\infty$-category of branes in this model was constructed using only data from the so-called BPS solitons and their interactions, and it corresponds to the Fukaya-Seidel category in symplectic geometry and singularity theory.  They follow and enlarge the vision of Morse theory as supersymmetric quantum mechanics (\cite{witten82}) by going beyond one dimension further.  More interestingly, when varying the theories in a continuous family, at some instance a kind of topological defect happens and it interpolates the theory left and right. Such defects are called interfaces, and an $A_\infty$ $2$-category was constructed to understand these interfaces.

To encode the information about the boundary conditions, i.e., the branes in LG models or string theory, we need $A_\infty$ category. However, to understand the interfaces between such theories, another kind of higher structures is necessary. The above $2$-category structure is a manifestation of such higher structures. That is also our main concern in the current paper to which we turn now.

In fact, this is already hinted in \cite{gaiotto-moore-witten} (Chapter 10, (10.48)).  Let us be more precise. To a Morse-Smale pair $(f, g)$ on a manifold $M$, one associates the Morse-Smale-Witten (MSW) complex $(\C_{\bullet} , \partial)$. Its homology group is called the Morse homology group $MH_\bullet(\C_{\bullet} , \partial)$. To prove that the Morse homology group is independent of the choice of the Morse-Smale pair, we choose another pair $(f', g')$ with corresponding MSW complex $(\C'_{\bullet} , \partial')$ and Morse homology group $MH'_\bullet(\C'_{\bullet} , \partial')$ and define a degree-preserving linear map
$$U: C_\bullet\rightarrow C'_\bullet$$
 which induces an isomorphism $\hat{U}: MH_\bullet\rightarrow MH'_\bullet$ whence a quasi-isomorphism.  To construct such a homomorphism $U$, one will choose an interpolation from the Morse-Smale pair $(f, g)$ to pair $(f', g')$ and $U$ is defined similarly as the differential by studying the gradient flow lines connecting critical points of the same Morse index coming from different pairs. In general, the homomorphism $U$ depends on the choice of the interpolation! If we choose two different generic interpolations, the two corresponding homomorphisms $U_0$ and $U_1$ satisfy
$$U_1-U_0=\partial' E-E\partial,$$
with $E$ a linear transformation between the two complexes which however reduces the degree by $1$, i.e., a homotopy between the homomorphisms $U_0$ and $U_1$. This is exactly (10.48) in\cite{gaiotto-moore-witten}. Once this is established, the induced maps on homology $\hat{U}_1=\hat{U}_0$ follows easily. Interestingly, $E$ can be constructuted from a generic interpolation between the two interpolations generating $U_0$ and $U_1$!

Now it is clear that the moduli spaces of gradient flow lines play a fundamental role in Morse theoretic studies. On the other hand, the complex itself, not merely an auxiliary tools to topology, becomes more and more important other than the homology group and other topological invariants. Perhaps, it is not a bad idea to go one step further in this direction. Properties of the moduli spaces will naturally lead to the algebraic structures behind the theory.

In \S $2$ we recall some basic facts and constructions in Morse theory to fix notations.

In \S $3$ we develop the analysis required for the compactified moduli space $\overline{\mathcal{M}}_{p,q}^{H[\ell]}$ with several parameters involved to arrive at the conclusion that $\overline{\mathcal{M}}_{p,q}^{H[\ell]}$ is a manifold with corners and has expected dimension. The first observation is that the iterated path space of the space of Riemannian metrics on a closed manifold is a Banach manifold. Then we introduce higher homotopies of functions and metrics interpolating given Morse-Smale pairs, and propose appropriate regular conditions on such pairs by generalizing the Morse-Smale condition{s} on a Morse-Smale pair.  Under such regular conditions for generic higher homotopies, the operator $\widehat{F}_{[\ell]} :\widetilde{\mathcal{G}}^{[\ell]}\times [0,1]^{\ell} \times \mathcal{P}_{p,q}^{1,2}\longrightarrow pr^{\ast}\mathcal{E}$ (see (\ref{Fhat})) corresponding to the gradient flow has Fredholm linearization.
The desired property of $\overline{\mathcal{M}}_{p,q}^{H[\ell]}:=F^{-1}_{[\ell]G_{[\ell]}}(0)$ (Theorem \ref{modulispaceellcase}) follows from Proposition \ref{usefulprop}, {where $G_{[\ell]}$ is a generic element in $\widetilde{\mathcal{G}}^{[\ell]}$ and $F^{-1}_{[\ell]G_{[\ell]}}:=\widehat{F}_{[\ell]}(G_{[\ell]},\cdots)$}.
To compactify the moduli space, the broken gradient flow line with respect to $(H_{[\ell]},G_{[\ell]})$ is introduced. For any sequence in the moduli space, we prove (Theorem \ref{FGellcase}) that there exists a subsequence which Floer-Gromov converges to a broken gradient flow line by combining \cite{frauenfelder-nicholls} and \cite{schwarz} {if it does not converge to a point belonging to the moduli space}. For later purpose we study in detail the boundary behavior of $1$-dimensional compactified moduli spaces with any number of parameters which provides us the desired algebraic properties introduced in the following section.


\S $4$ is the algebraic part of the paper, and here we basically follow the framework of Leinster (\cite{leinster}) for the $\infty$-categories.
After introducing the necessary notions,  we define an $n$-globular set in terms of Morse functions and their higher homotopies with natural composite and identity map. By extending $n$ to $\infty$, we get a weak $\infty$-category $\mathcal{A}$. Another strict $\infty$-category $\mathcal{B}$ is also constructed by exploring the Morse complexes and their higher homotopies.
Finally the boundary structures of the {\it one-dimensional} compactified moduli space of gradient flow lines of Morse functions with parameters allow us to get a weak $\infty$-functor $\mathcal{F}: \mathcal{A} \rightarrow \mathcal{B}$ (Theorem \ref{inftyfunctor}) to encode their relations. In the process, we give a more precise explanation about the non-generic codimension one sets for one dimensional moduli space which play a vital role to define the weak $\infty$-functor, and such sets are a kind of defects in TQFT.

For a fixed $H_{[\ell]}$ under regular condition, considering higher ($\geqslant 2$) dimensional moduli spaces would be an interesting topic.
In \S 5
we will give a description about the boundaries of two dimensional compactified moduli space, and we leave the corresponding algebraic structures for future. Other further investigations mentioned in \S 5 include the orientation of the moduli space with the further refinements of our conclusions, and the generalization to Morse homotopy and even to infinite dimensional Floer homology.

Several related attempts to higher version of Morse theory are in sequel:

In \cite{hohloch}, an almost strict $n$-category in Morse theory was constructed. The objects are the critical points of a Morse function with morphisms the moduli space of gradient flow lines. And the higher morphisms come from doing iteratively Morse theory with good Morse function on moduli space of gradient flow lines which sit in the previous morphisms. Note that here $n$ stops at $\dim M+1$. We follow partially the terminology about $n$-categories in this paper which is well suited for our purpose.

Inspired by Morse theory, Lurie and Tanaka (\cite{lurie18}) also initiated a program to construct Morse complex without using gradient flow equations. They constructed a topological stack {\bf Broken} to abstract the information about the moduli space of gradient flow lines, and the topological stack can be represented as a Lie groupoid with corners. Sheaves on such moduli stack are also studied.

The homotopy transfer theorem (HPT) transfers the differential graded algebra (DGA) structure on the singular cochain complex to an $A_\infty$-algebra structure on its homotopy retract MSW complex.  In \cite{mazuir1}, the author proves that for a given Morse function on a smooth compact manifold, the MSW complex has an $\Omega BAs$-algebra structure, a generalization of $A_\infty$ algebra structure, via moduli space of perturbed gradient trees (\cite{abouzaid11, mescher18}). The higher homotopy theory for $\Omega BAs$-morphisms is also developed. The geometric flavor of using the moduli space of perturbed gradient trees is the same as ours.  However, to achieve such higher algebra structures, the gradient vector field has to be modified around each vertex which we do not need.

It seems plausible to clarify the relations among these higher structures, and these complementary viewpoints will be helpful to get a deeper understanding of higher Morse theory.

The arrangement of the paper is as follows: after recalling the fundamental construction in Morse homology in \S 2, in \S 3 we develop the analytic part concerning the moduli space of gradient flow lines, here basically we follow \cite{schwarz} and \cite{frauenfelder-nicholls}. The latter did an excellent job based on the former which is the first complete analytical monograph on Morse homology to address and improve every aspect with all the required details about the moduli space of the gradient flow trajectories. The main difference here and theirs is that one Morse-Smale pair is enough for their purpose to construct the Morse complex and prove the independence of the Morse-Smale pair, however ALL Morse-Smale pairs are needed for us.

When we say a manifold in the paper, we always mean an $n$-dimensional closed connected smooth Riemannian manifold $(M,g)$.

\begin{section}
{Basic Morse Theory}
\label{Basic Morse Theory sec2}
\end{section}

\indent We briefly review aspects of Morse theory to fix notations for our constructions, and more details can be found in \cite{Arnold-Gusein-Zade-Varchenko}, \cite{audin-damian}, \cite{frauenfelder-nicholls}, \cite{Guillemin- Sternberg}, \cite{milnor}, \cite{nicolaescu} and \cite{schwarz}, to name a few.

Given a smooth function $f:M \rightarrow \R$,  a point $p\in M$ is called a critical point of $f$  if $df(p)=0$, which means in the local coordinates $(x_{1}, x_{2},...,x_{n})$ around $p$, we have $\frac{\partial f}{\partial x_{1}}(p)= ...= \frac{\partial f}{\partial x_{n}}(p)= 0.$ The set of all its critical points is denoted by $Cr(f)$. A smooth function $f: M\rightarrow \R$ is called a Morse function if all the critical points are non-degenerate, i.e., for any critical point $p$ of $f$, the Hessian
$Hessian(f)(p):= \big(\frac{\partial^{2} f}{\partial x_{i}\partial x_{j}}(p)\big)_{i,j}$ is non-degenerate.
{In the literature Hessian is also defined in terms of the Riemannian metric, that is, $Hessian(f)(p):= \big(g^{ij}\frac{\partial^{2} f}{\partial x_{i}\partial x_{j}}(p)\big)_{i,j}$. }

According to the classical Morse lemma, the Morse function $f$ has a very simple form in an appropriate local coordinate system near any critical point.
More specifically, let $p\in Cr(f)$, then there exists a neighborhood $U_{p}$ of $p$, where $f(x)$ can be written in the form of $f(x)= f(p)- x_{1}^{2}- ...- x_{k}^{2}+ x_{k+1}^{2}+ ...+ x_{n}^{2},\, \forall x\in U_p$. We call it the Morse chart. Moreover, $k$ is called the Morse index of $p$, denoted by $ind (p)=k$. Note that the non-degeneracy and index of the critical point is independent of the variant of the Hessian.

The compactness of $M$ will simplify the situation a lot since
a Morse function $f:M \rightarrow \R$ has only finitely many critical points if $M$ is compact.

On a Riemannian manifold $(M,g)$, the gradient $\nabla _{g}f$ of a smooth function $f:M\rightarrow \mathbb{R}$ with respect to the Riemannian metric $g$ is defined to be the unique vector field such that
$\langle \nabla_{g} f, Y\rangle_g= df(Y)$ for any vector field $Y\in \Gamma(TM)$. If there is no potential confusion, we will abbreviate $\nabla _{g}f$ to $\nabla f$.
Note that in the definition of the critical points, we don't need the metric. However the gradient vector field does depend on the Riemannian metric $g$ on $M$.

Now consider the negative gradient flow equation for a fixed Morse function $f$ and a Riemannian metric $g$ on $M$:
$ \dot{u}(t)= -\nabla _{g} f(u(t))$.
A solution $u:\R\rightarrow M \in C^{\infty}(\R, M)$ of the system of equations is called a parametrized negative gradient flow line or a trajectory, and {sometimes we call it simply gradient flow line if there is no ambiguity}.

We are interested in the moduli space of all smooth solutions to the negative gradient flow equation above under appropriate boundary conditions. However, this space is lack of completeness and not a Banach space, and the Sard's theorem cannot be used. To resolve this issue, as usually done in functional analysis, we need to relax the requirements to get a Banach manifold $\mathcal{P}^{1,2}_{p,q}$ with local model a Banach space.  In fact, we will introduce $\mathcal{P}^{1,2}_{p,q}$ in the next chapter. According to the regularity theorem of elliptic differential operators, it is the moduli space with the desired properties.

We also have the fact that any negative gradient flow line $u(t)$ is asymptotic to critical points of the Morse function $f$, i.e.,
$\lim\limits_{t\rightarrow -\infty} u(t)=p\in Crf,\ \lim\limits_{t\rightarrow \infty} u(t)=q\in Crf.$


The rest point of the gradient flow $u(t)$ corresponds to the critical point of $f$. Fix a critical point $p\in Crf$, we have the stable manifold $W^s({p})$ and unstable manifold $W^{u}(p)$ of the flow $u(t)$ at $p$.
They are diffeomorphic to $ind p$- and $(dim M- ind p)$-dimensional balls respectively. Once again, we would like to emphasize that these definitions involve the Riemannian metric $g$ as $\nabla f=\nabla_{g} f$.

Let $f: M \longrightarrow \R$ be a Morse function on $M$, and $g$ be a Riemannian metric. We call the pair $(f,g)$ a \textbf{Morse-Smale (MS) pair} if for any two critical points $p,q\in Crf$ of $f$, $W^{u}(p)$ intersects $W^{s}(q) $ transversally, i.e. for any $x\in W^{u}(p)\cap W^{s}(q)$,
we have $T_{x}M= T_{x}W^{u}(p)+T_{x}W^{s}(q).$

For a MS pair $(f,g)$ and a pair of critical points $p,q \in Crf$, consider the negative gradient flow equation $\frac{\partial u}{\partial t}=-\nabla_{g}f(u(t))$ with the boundary conditions $\lim_{t\rightarrow -\infty}u(t)=p,\,\,\,\, \lim_{t\rightarrow +\infty}u(t)=q$. In local coordinates, it has the form: $\frac{du^{i}}{dt}= - \sum_{j=1}^n g^{ij}\frac{\partial f}{\partial u^{j}}$,
where $(g^{ij})=(g_{ij})^{-1}$, $i,j=1,\cdots,n$.

We define the solution space with respect to these boundary conditions to be
$$\widetilde{\mathcal{M}}_{p,q}^{f}:=\{u:\R\longrightarrow M| \frac{\partial u}{\partial t}=-\nabla_{g}f(u(t)),\, \lim_{t\rightarrow -\infty} u(t)=p \,\, \textrm{and} \,\, \lim_{t\rightarrow +\infty} u(t)=q \}.$$
And we call it the moduli space of the negative gradient flow lines from $p$ to $q$. It is of dimension $ind(p)-ind(q)$.

In fact, there is a natural free $\R$-action on $\widetilde{\mathcal{M}}_{p,q}^{f}$, namely the time translation
\begin{equation}
\begin{split}
\R\times \widetilde{\mathcal{M}}_{p,q}^{f}&\longrightarrow \widetilde{\mathcal{M}}_{p,q}^{f}\\
(r,u(\cdot))&\longmapsto u(\cdot+r).\nonumber
\end{split}
\end{equation}
We define the reduced moduli space  $\mathcal{M}_{p,q}^{f}$ as the quotient
$\mathcal{M}_{p,q}^{f}:= \widetilde{\mathcal{M}}_{p,q}^{f}/\R$ with respect to the $\R$-action.

The reduced moduli space $\mathcal{M}_{p,q}^{f}$ can be compactified to a stratified space $\overline{\mathcal{M}}_{p,q}^{f}$ by adding `broken lines'. The codimension-$(k-1)$ stratum $\mathcal{M}_{k}(p,q)$ is described as follows. Assume $ind (p)- ind (q)= \ell (\geqslant 1)$, we consider $k-1$ critical points $c_{1},...,c_{k-1}$ such that $ind(p)> c_{1}> ...> c_{k-1}> ind(q)$. Thus the maximum of $k$ is $\ell$, i.e. $1\leqslant k\leqslant \ell$. Consider the free $\R^{k}$-action on $\widetilde{\mathcal{M}}_{p,c_{1}}^{f}\times\widetilde{\mathcal{M}}_{c_{1},c_{2}}^{f} \times \cdots \times \widetilde{\mathcal{M}}_{c_{k-1},q}^{f}$
\begin{equation}
\begin{split}
\R^{k}\times \widetilde{\mathcal{M}}_{p,c_{1}}^{f}\times\widetilde{\mathcal{M}}_{c_{1},c_{2}}^{f} \times \cdots \times \widetilde{\mathcal{M}}_{c_{k-1},q}^{f}&\longrightarrow \widetilde{\mathcal{M}}_{p,c_{1}}^{f}\times\widetilde{\mathcal{M}}_{c_{1},c_{2}}^{f} \times \cdots \times \widetilde{\mathcal{M}}_{c_{k-1},q}^{f}\\
((r_{1},...,r_{k}),(u_{1},...,u_{k}))&\longmapsto (u_{1}(r_{1}+\cdot),...,u_{k}(r_{k}+\cdot)).
\nonumber
\end{split}
\end{equation}
Then we define
$\mathcal{M}_{k}(p,q):= \widetilde{\mathcal{M}}_{p,c_{1}}^{f}\times\widetilde{\mathcal{M}}_{c_{1},c_{2}}^{f} \times \cdots \times \widetilde{\mathcal{M}}_{c_{k-1},q}^{f}/\R^{k}$.
Specially, if $k=1$, we only have the codimension-$0$ stratum, i.e., the top stratum $\mathcal{M}_{1}(p,q)= \mathcal{M}_{p,q}^{f}$ as expected. By adding these codimension-$k$ strata, we obtain the compactification $\overline{\mathcal{M}}_{p,q}^{f}$ of $\mathcal{M}_{p,q}^{f}$, i.e.,
$\overline{\mathcal{M}}_{p,q}^{f}= \bigsqcup\limits_{k=1}\limits^{\ell}\mathcal{M}_{k}(p,q).$

Recall that a smooth manifold with corner is a second countable Hausdorff space $M$ of dimension $n$ such that $\forall p\in M,$ it has a neighbourhood homeomorphic to an open set of $\R^{k}\times [0,\infty)^{n-k}$ with $k$ integer such that $1\leqslant k\leqslant n$, and the transition maps are smooth.

Now we can state the key fact about the moduli space.
\begin{theorem}(\cite{hutchings}\label{modulionefunction}, p. 10)
Let $(f,g)$ be a Morse-Smale pair on the closed manifold $M$ with two given critical points $p, \, q\in Cr f$. Then the natural compactification $\overline{\mathcal{M}}^f_{p,q}$ of the reduced moduli space $\mathcal{M}^f_{p,q}$ is an oriented smooth manifold of dimension $ind(p)-ind(q)-1$ with boundary and corners.
\end{theorem}
Here we are content with stating only the result, and we will explore the moduli spaces of various types further systematically in the next section to build up our algebraic structures.

{Now we recall Morse-Smale-Witten complex in Morse theory.}
 Let $f: M\rightarrow \R$ be a Morse function on $M$ and
$Cr_{k}(f):=\{ p\in Cr(f)|ind (p)= k\}.$
Define a vector space $\C_{k}$ generated by $Cr_{k}(f)$ over $\Z_{2}$ as
$\C_{k}:= \Z_{2}\langle Cr_{k}(f)\rangle,$
and we introduce the graded vector space over $\Z_2$ as
$\C_\bullet:=\oplus_{k=1}^n \C_k.$
For any $1\leqslant k\leqslant n$, define a linear map
 \begin{eqnarray*}
 \partial_{k}: \C_{k} &\rightarrow& \C_{k-1},\\
  p &\mapsto& \sum\limits_{q\in Cr_{k-1}(f)}\sharp_{2}(\mathcal{M}_{p,q}^{f})\cdot q
  \end{eqnarray*}
where $\sharp_{2}(\mathcal{M}_{p,q}^{f})= \sharp (\mathcal{M}_{p,q}^{f})\ \  mod\ 2$, and $\sharp(\mathcal{M}_{p,q}^{f})$ means the number of elements in $\mathcal{M}_{p,q}^{f}$.
The linear map is well defined since the moduli space $\mathcal{M}_{p,q}^f$ consists of finitely many isolated points when $ind(p)-ind(q)=1$ following Theorem \ref{modulionefunction}. Furthermore,
the linear maps satisfy $\partial^{2}= 0$. Thus, $(\C_{\bullet} , \partial)$ is equipped with the structure of a chain complex, which is called the {\it Morse-Smale-Witten (MSW) complex}.

Let $(f,g)$ be a Morse-Smale pair on the closed manifold $M$ with $(\C_{\bullet} , \partial)$ its MSW complex. Define
the $\ell$-th {\it Morse homology group} of MSW chain complex to be
$MH_{\ell}:=\frac{Ker \partial_{\ell}}{Im \partial_{\ell+1}}, \ \ \ell=0,\cdots,n.$ What's more,
Morse homology group $MH_{\bullet}$ is independent of the choice of $(f,g)$.
However, $( \C_{\bullet} , \partial)$ does depend on the choice of $(f,g)$. In fact, to understand the dependence on $(f,g)$ is the starting point and main goal of the current paper. Another important fact concerning the Morse homology group is that
Morse homology group $MH_{\bullet}$ is isomorphic to the singular homology group of $M$, i.e. $MH_{\bullet}\cong H_{\bullet}$.

\indent
The proof of the theorem that Morse homology group $MH_{\bullet}$ is independent of the choice of $(f,g)$ relies on the careful analysis of the moduli space of the parameterized negative gradient flow lines. In this paper, we focus on the chain complex rather than the homology and want to figure out algebraic structures via the parameterized negative gradient flow line and its corresponding moduli space to which we turn now.

\begin{section}
{Moduli Space and its Compactification}
\label{Moduli Space and its Compactification sect3}
\end{section}

In the last section, we have already met the moduli space of negative gradient flow lines of a Morse-Smale pair, in which case there is only one Morse function involved. As noted in the introduction, to prove the independence of the Morse homology group on the Morse-Smale pair, a ``cobordism'' between two Morse-Smale pairs was used to derive the homomorphism between the corresponding Morse-Smale-Witten complexes by exploring the moduli space of gradient flow lines connecting the critical points belonging to different MS pairs. We call such interpolation between the MS pairs a homotopy between them. In general two such homomorphisms are NOT equal exactly, and only equal up to homotopy! This suggest us to go further to higher homotopy of the two homotopies of MS pairs to unravel the algebraic structures among the MS complexes. As a first step we need to understand better the resulting moduli spaces and their compactifications which is the subject of the section.

More precisely, to encode the relations among the homotopies, we will consider the space of maps from parameter space $[0,1]^k$ to the space of Riemannian metrics and smooth functions on $M$.  We adapt the analysis from \cite{schwarz} and \cite{frauenfelder-nicholls}. For the sake of conciseness, we only emphasize the differences.

At the starting point we prove that the iterated path space of the space of Riemannian metrics on a closed manifold is  a Banach manifold which will be a basic fact needed to generalize Proposition \ref{usefulprop} to our settings.

\textbf{Space $(\mathcal{G}_{g_{0}},\parallel \cdot \parallel _{a})$  of Riemannian Metrics}

Following Schwarz(\cite{schwarz}, p. 45-46), let $g_{0}$ be a fixed Riemnnian metric on $M$ with induced norm $\|\cdot\|$. The space $\mathcal{G}_{g_{0}}$ consists of all Riemannian metrics on $M$, and it can be viewed as a subset of a vector space formed by finite norm sections. We first introduce the vector space of sections together with a norm which turns out to be a Banach space.

\begin{definition}
Let $E$ be a vector bundle on $M$. Assume $E$ is equipped with a bundle norm $\mid \cdot\mid$ and a covariant derivation $\nabla$. Fix a sequence $(a_{j})_{j\in\mathbb{N}}$ such that $a_{j}\in \mathbb{R}_{>0}$, for any $j\in\mathbb{N}$.
For any smooth section $s\in C^{\infty}(E):=C^{\infty}(M, E)$, we define
$$\parallel s \parallel _{a}:=\sum\limits_{j\geqslant0} a_{j} \sup\limits_{p\in M}\mid \nabla^{j}s(p)\mid,$$
where$\mid \nabla^{j}s(p)\mid =
\max \{\mid \nabla_{X_{i_{1}}}\cdots \nabla_{X_{i_{j}}}s(p) \mid \big|
\parallel X_{i_{1}} \parallel=\cdots =\parallel X_{i_{j}} \parallel=1,\
\ i_{1},\cdots,i_{j}\in \{1,\cdots, \dim M = n\}, \ \, X_{i_{k}}\in T_{p}M, \forall k=1,\cdots,j\}$ and $\mid \nabla^{0}s(p)\mid:= \mid s(p)\mid$.
\end{definition}

\begin{remark}
$1.$ Actually, for an arbitrary element $s\in C^{\infty}(E)$, $\parallel s \parallel _{a}$ could be infinite. As usual, only when $\parallel s \parallel _{a}$ is finite for any $s$ in an appropriate  subspace, we call $\parallel \cdot \parallel _{a}$ a norm for that space.

$2.$ For the $C^k$ smooth case, we can define the norm similarly, i.e. if $s\in C^{k}(E)$, then for given positive real sequence $(a_{j})_{j=0}^{k}$, we take $\parallel s \parallel _{a}:=\sum\limits_{j=0}^{k} a_{j} \sup\limits_{p\in M}\mid \nabla^{j}s(p)\mid$.
\end{remark}

As indicated in the last remark, we gather together all finite sections under $\parallel \cdot \parallel _{a}$ which constitute a vector space,
$$ C^{a}(E):= \{s\in C^{\infty}(E)\big| \parallel s \parallel _{a}< \infty \}.  $$
So $\parallel \cdot \parallel _{a}$  is a norm on space $ C^{a}(E)$.

\begin{remark}
$1.$ The sequence $(a_{j})_{j\geqslant0}$, $\forall a_{j}\in \R_{>0}$, can be taken arbitrarily.  Taking advantage of the arbitrariness, we can choose good sections $s$ such that $\parallel s \parallel _{a}< \infty$, which roughly means that when $j$ is sufficiently large, $\mid\nabla^{j}s\mid$ tends to zero.

$2.$ Such sequence of positive real numbers $(a_{j})_{j\geqslant0}$ always exists so that $(C^{a}(E),\parallel \cdot \parallel _{a})$ is not empty. The idea is that any chosen section $s\in C^{\infty}(E)$ can be viewed as a standard vector bundle section (in our case, standard Riemannian metric), then there exists a sequence of positive real numbers $(a_{j})_{j\geqslant0}$ such that $s$ belongs to $(C^{a}(E),\parallel \cdot \parallel _{a})$. The point is that for such fixed $s$ and for any $j$, $\sup\limits_{p\in M}\mid \nabla^{j}s(p)\mid$ is fixed too, so we can choose $a_j$ that $(a_{j})$ goes to zero fast enough to ensure $ \parallel s \parallel _{a}< \infty$.
\end{remark}

\begin{proposition} \label{Banach space}
$(C^{a}(E),\parallel \cdot \parallel _{a})$ is a Banach space if it is not empty.
\end{proposition}

To prove the statement, we only need to verify that any Cauchy sequence is convergent in $C^{a}(E)$. It can be proved in two steps. The first one is to show that any Cauchy sequence uniformly converges to a section $s\in C^{\infty}$. The second step aims to check that $\parallel s \parallel _{a}< \infty$. The details are left to the reader.

Now, as in Schwarz (\cite{schwarz}, p. 45-46), we define the space of Riemannian metrics.

\begin{definition}
For a Riemannian metric $g_{0}$ on $M$, we define
$$\mathcal{T}_{g_{0}}:= \{(p,T_{p})\big| p\in M, T_{p}  \textrm{\,\, is positive definite and self-adjoint linear transformation  with respect to} \, \langle\cdot,\cdot\rangle_{p}:= g_{0}(p)\}$$ and
 $\mathcal{G}_{g_{0}}:= C^{a}(\mathcal{T}_{g_{0}}).$
\end{definition}
As positive definiteness is an open condition,  $\mathcal{T}_{g_{0}}$ is an open neighbourhood of section $1_{M}$ in a subbundle of $E=End(TM)$ over $M$.
Proposition \ref{Banach space} tells us that $\mathcal{G}_{g_{0}}$ is a Banach manifold. Recall that Banach manifold is a manifold which is locally modeled on Banach space.

\textbf{Path Space $(\widetilde{\mathcal{G}_{g_{0}}},\parallel \cdot \parallel _{\sim_{b}})$ of the Space of Riemannian Metrics}

Given two Riemannian metrics $g^{\alpha}, \ g^{\beta}\in \mathcal{G}_{g_{0}}$, we consider the paths in $\mathcal{G}_{g_{0}}$  with fixed ends $g^{\alpha}$ and $ g^{\beta}$ , and they form the space $\widetilde{\mathcal{G}_{g_{0}}}$. To simplify the notation here we omit the dependence on the ends, and this is the main case that we will used later.
Then $\widetilde{\mathcal{G}_{g_{0}}}$ is also a Banach manifold in a Banach space under a norm $\parallel \cdot \parallel _{\sim_{b}}$ to be defined later.

\begin{definition}
Let $E=End(TM)$ be the vector bundle on $M$. Define
$$ C^{\infty}([0,1],(C^{a}(E),\parallel \cdot \parallel_{a}))= \big\{\widetilde{s}:[0,1]\longrightarrow C^{a}(E)\, \textrm{smooth} \big\},$$
$$C^{\infty}([0,1],\mathcal{G}_{g_{0}}):=\big\{\widetilde{s}:[0,1]\longrightarrow \mathcal{G}_{g_{0}}\,\, \textrm{smooth} \,\,\big|\,\,\ \widetilde{s}(0)=g^{\alpha},\ \widetilde{s}(1)=g^{\beta} \big\}.$$
Fix as before another sequence $(b_{n})_{n\in\mathbb{N}}$ such that $b_{n}\in \mathbb{R}_{>0}$, for any $n\in\mathbb{N}$. 
For any smooth section $\widetilde{s}\in C^{\infty}([0,1],(C^{a}(E),\parallel \cdot \parallel_{a}))$, we define
$$\parallel \widetilde{s} \parallel _{\sim_{b}}:=\sum\limits_{n\geqslant0} b_{n} \sup\limits_{t\in [0,1]}\parallel \widetilde{s}^{(n)}(t)\parallel_{a},$$
here the superscript $(n)$ means $n$-th derivative with respect to time t.

Define
$$\widetilde{E^{a,b}}:=\big\{\widetilde{s}:[0,1]\longrightarrow C^{a}(E)\, \textrm{smooth}\,\big|\, \parallel \widetilde{s} \parallel _{\sim_{b}}< \infty \big\}\subset C^{\infty}([0,1],(C^{a}(E),\parallel \cdot \parallel_{a})),$$
$$\widetilde{\mathcal{G}_{g_{0}}}:=\big\{\widetilde{s}:[0,1]\longrightarrow \mathcal{G}_{g_{0}}\subset C^{a}(E) \,\,\big|\,\,\widetilde{s}(0)=g^{\alpha},\ \widetilde{s}(1)=g^{\beta},\  \parallel \widetilde{s} \parallel _{\sim_{b}} < \infty \big\}\subset \widetilde{E^{a,b}}.$$
\end{definition}

We have similar remarks as before,

\begin{remark}
$1.$  We can also define $\parallel \cdot \parallel _{\sim_{b}}$ in the $C^{k}$ situation. For $(k+1)$ given positive real numbers $(b_{n} )_{n=0}^{k}$, we define $\parallel \widetilde{s} \parallel _{\sim_{b}} := \sum\limits_{n=0}^{k} b_{n} \sup\limits_{t\in [0,1]}\parallel \widetilde{s}^{(n)}(t)\parallel_{a}$.\\
$2.$ There exists a sequence of positive real numbers $(b_{n})_{n\geqslant0}$ such that $\widetilde{E^{a,b}}$ is not empty. In fact, we have already known that $C^a(E)$ is nonempty by choosing suitable $(a_{n})_{n\geqslant0}$.
Since $C^{a}(E)$ is a sub-vector space of $C^{\infty}(E)$, there exists some $\widetilde{s}\in C^{\infty}([0,1],C^{a}(E))$ such that $\widetilde{s}(t_{0})\in C^a(E)$ for some $t_0$ which implies that $\widetilde{s}(t)\in C^{a}(E)$ for any $t\in [0,1]$. Once this is done, the following is similar. Fixing such an $\widetilde{s}$, $\parallel \widetilde{s}^{(\ell)}(t)\parallel_{a} (\forall \ell\geqslant 0)$ are also fixed, then there exists a sequence of positive real numbers $(b_{n})_{n\geqslant0}$ tending to zero quickly enough, so that $\widetilde{s}$ belongs to $\widetilde{E^{a,b}}=\{\widetilde{r}\in C^{\infty}([0,1], C^{a}(E))\big|\parallel\widetilde{r} \parallel_{\sim_{b}}<\infty \}$. So there exists a sequence of positive real numbers $(b_{n})_{n\geqslant0}$ such that $\widetilde{E^{a,b}}$ is not empty.

$3.$ Furthermore, we can generalize the construction iteratively.
\end{remark}

\begin{proposition}
$(\widetilde{E^{a,b}},\parallel \cdot \parallel _{\sim_{b}})$ is a Banach space if it is not empty.
\end{proposition}

\begin{proposition}
$(\widetilde{\mathcal{G}_{g_{0}}},\parallel \cdot \parallel_{\sim_{b}})$ is a Banach manifold in a Banach space.
\end{proposition}


One can consider the path space of the path space $\widetilde{\mathcal{G}_{g_0}}$ of the space $\mathcal{G}_{g_{0}}$ of the Riemannian metric on $M$ recursively by introducing more parameters and get the Banach manifold structure.  To lighten the notations, we will use $\widetilde{\mathcal{G}}^{[\ell]}$ to denote such ``higher'' path space with fixed ends of the space of Riemannian metrics, which we now turn to.

\textbf{Higher homotopies of functions and metrics}

Let us first fix two MS pairs $(f^{\alpha},g^{\alpha})$ and $(f^{\beta},g^{\beta})$ once and for all. {
We define a homotopy $H^{\alpha \beta}_{[\ell]}$ between the two Morse functions$f^{\alpha}$ and $f^{\beta}$ to be a smooth map
\begin{eqnarray*}
H^{\alpha \beta}_{[\ell]}: [0,1]^{\ell}\times \mathbb{R}\times M &\longrightarrow& \mathbb{R} \\(s^{\underline{\ell}},s^{\underline{\ell-1}},\cdots,s^{\underline{1}},t,u(t)) &\mapsto& H^{\alpha \beta}_{[\ell]}(s^{\underline{\ell}},s^{\underline{\ell-1}},\cdots,s^{\underline{1}},t,u(t)),
\end{eqnarray*}
 and there exists $T>0$ which we will fix once for all such that
$$H^{\alpha \beta}_{[\ell]}(s^{\underline{\ell}},s^{\underline{\ell-1}},\cdots,s^{\underline{1}},t,\cdot) = \left\{
 \begin{split}
  &f^{\alpha}&\indent&\text{if t}\in (-\infty, -T)\\
  &f^{\beta} &\indent&\text{if t}\in (T,+\infty).
 \end{split}\right.$$
We call such homotopy $H^{\alpha \beta}_{[\ell]}$ between the two Morse functions$f^{\alpha}$ and $f^{\beta}$ \textbf{finite} since it is nontrivial in time variable $t$ only in a finite segment$[-T,T]$. We also write $H_{[\ell]}$ instead of $H^{\alpha \beta}_{[\ell]}$ for brevity if there is no confusion, and sometimes we write $H$ replacing $H_{[0]}$. Denote $\mathbf{s}=(s^{\underline{\ell}},s^{\underline{\ell-1}},\cdots,s^{\underline{1}})$.
The followings are the concrete definitions of such homotopies we utilize.}

Define
\begin{align}
\mathcal{H}^{[0]}:=\mathcal{H}:= \Big\{ & H=H_{[0]}:\R\times M\longrightarrow \R \ smooth   \,\, \Big |\,\,
                                        \, H(t,\cdot)=f^{\alpha},\, t\in (-\infty, -T)\, ; \, H(t,\cdot)=f^{\beta},\, t\in (T,+\infty) \Big\}. \nonumber
\end{align}
Let $\ell \in \Z_{>0}$ be the number of the parameters, we define
\begin{align}
\mathcal{H}^{[\ell]}:= \Big\{ & H_{[\ell]}:[0,1]^{\ell}\times \R\times M \longrightarrow \R \ smooth ,\, (s^{\underline{\ell}},s^{\underline{\ell-1}},\cdots,s^{\underline{1}},t,\cdot)\mapsto  H_{[\ell]}(s^{\underline{\ell}},s^{\underline{\ell-1}},\cdots,s^{\underline{1}},t,\cdot)   \,\, \Big| \,\,   \nonumber \\
                              &H_{[\ell]}(0,\underbrace{\cdot,\cdots,\cdot}_{\ell-1},\cdot,\cdot)\in \mathcal{H}^{[\ell-1]}, \ H_{[\ell]}(1,\underbrace{\cdot,\cdots,\cdot}_{\ell-1},\cdot,\cdot)\in \mathcal{H}^{[\ell-1]}; \nonumber \\
                              &H_{[\ell]}(\cdot,0,\underbrace{\cdot,\cdots,\cdot}_{\ell-2},\cdot,\cdot)=H_{[\ell]}(0,0,\underbrace{\cdot,\cdots,\cdot}_{\ell-2},\cdot,\cdot)\in \mathcal{H}^{[\ell-2]}, \ H_{[\ell]}(\cdot,1,\underbrace{\cdot,\cdots,\cdot}_{\ell-2},\cdot,\cdot)=H_{[\ell]}(0,1,\underbrace{\cdot,\cdots,\cdot}_{\ell-2},\cdot,\cdot)\in \mathcal{H}^{[\ell-2]}; \nonumber \\
                              &\cdots; \nonumber \\
                              &H_{[\ell]}(\underbrace{\cdot,\cdots,\cdot}_{\ell-1},0,\cdot,\cdot)=H_{[\ell]}(\underbrace{0,\cdots,0}_{\ell-1},0,\cdot,\cdot)\in \mathcal{H}^{[0]}, \ H_{[\ell]}(\underbrace{\cdot,\cdots,\cdot}_{\ell-1},1,\cdot,\cdot)=H_{[\ell]}(\underbrace{0,\cdots,0}_{\ell-1},1,\cdot,\cdot)\in \mathcal{H}^{[0]}; \nonumber \\
                              &H_{[\ell]}(\underbrace{\cdot,\cdots,\cdot}_{\ell},t,\cdot)=H_{[\ell]}(\underbrace{0,\cdots,0}_{\ell},t,\cdot)=f^{\alpha},\, t\in (-\infty, -T) ;\  \nonumber \\ &H_{[\ell]}(\underbrace{\cdot,\cdots,\cdot}_{\ell},t,\cdot)=H_{[\ell]}(\underbrace{0,\cdots,0}_{\ell},t,\cdot)=f^{\beta},\,\, t\in (T,+\infty)   \Big\}. \nonumber
\end{align}
Such an element will be called a higher homotopy between Morse functions $f^\alpha$ and $f^\beta$.


Similarly, we define $\widetilde{\mathcal{G}}^{[0]}:=\widetilde{\mathcal{G}_{g_0}}$, and as above we sometimes write $\widetilde{\mathcal{G}}^{[0]}$ as $\widetilde{\mathcal{G}}$ for brevity. In general, for $\ell\in \mathbb{Z}_{\geqslant 1}$, let
\begin{align}
\widetilde{\mathcal{G}}^{[\ell]}:=\Big\{ & G_{[\ell]}:[0,1]^{\ell}\times \R\longrightarrow \mathcal{G}_{g_{0}},\  (s^{\underline{\ell}},s^{\underline{\ell-1}},\cdots,s^{\underline{1}},t)\mapsto  G_{[\ell]}(s^{\underline{\ell}},s^{\underline{\ell-1}},\cdots,s^{\underline{1}},t)    \Big|  \nonumber \\
                              &G_{[\ell]}\ depend\ on\ parameters\ and\ variable\ smoothly;  \nonumber \\
                              &G_{[\ell]}(0,\underbrace{\cdot,\cdots,\cdot}_{\ell-1},\cdot)\in \widetilde{\mathcal{G}}^{[\ell-1]}, \  G_{[\ell]}(1,\underbrace{\cdot,\cdots,\cdot}_{\ell-1},\cdot)\in \widetilde{\mathcal{G}}^{[\ell-1]}; \nonumber \\
                              &G_{[\ell]}(\cdot,0,\underbrace{\cdot,\cdots,\cdot}_{\ell-2},\cdot)=G_{[\ell]}(0,0,\underbrace{\cdot,\cdots,\cdot}_{\ell-2},\cdot)\in \widetilde{\mathcal{G}}^{[\ell-2]},\  G_{[\ell]}(\cdot,1,\underbrace{\cdot,\cdots,\cdot}_{\ell-2},\cdot)=G_{[\ell]}(0,1,\underbrace{\cdot,\cdots,\cdot}_{\ell-2},\cdot)\in \widetilde{\mathcal{G}}^{[\ell-2]}; \nonumber \\
                              &\cdots; \nonumber \\
                              &G_{[\ell]}(\underbrace{\cdot,\cdots,\cdot}_{\ell-1},0,\cdot)=G_{[\ell]}(\underbrace{0,\cdots,0}_{\ell-1},0,\cdot)\in \widetilde{\mathcal{G}}^{[0]}, \  G_{[\ell]}(\underbrace{\cdot,\cdots,\cdot}_{\ell-1},1,\cdot)=G_{[\ell]}(\underbrace{0,\cdots,0}_{\ell-1},1,\cdot)\in \widetilde{\mathcal{G}}^{[0]}; \nonumber \\
                              &G_{[\ell]}(\underbrace{\cdot,\cdots,\cdot}_{\ell},t)=G_{[\ell]}(\underbrace{0,\cdots,0}_{\ell},t)=g^{\alpha},\,\,  t\in (-\infty, -T)  ;\nonumber \\
                              &G_{[\ell]}(\underbrace{\cdot,\cdots,\cdot}_{\ell},t)=G_{[\ell]}(\underbrace{0,\cdots,0}_{\ell},t)=g^{\beta},\,\,t\in (T,+\infty)  \Big\}. \nonumber
\end{align}

Note that $\widetilde{\mathcal{G}}^{[\ell]}$ depends on the choice of the Riemannian metric $g_{0}$.
Such an element in the above set will be called a higher homotopy between Riemannian metrics $g^\alpha$ and $g^\beta$.
Figure \ref{parameter space} describes the family of pairs (function, Riemannian metric) with parameter degeneration.

We take $\mathcal{H}^{[\ell]}$ as an example. \\
(a) shows a family $H$ with one-variable $t$ from $f^{\alpha}$ to $f^{\beta}$. \\
(b) is a family with one-parameter $s$ and one variable $t$ such that
when $t\in (-\infty, -T)$, $H_{[1]}(\cdot,t,\cdot)$ (left segment) degenerate to $f^{\alpha}$; $t\in (T, +\infty)$, $H(\cdot,t,\cdot)$ (right segment) degenerate to $f^{\beta}$. \\
Without lose the generality, we suppose $t\in [-T-\varepsilon, T+\varepsilon]$, when $t=-T-\varepsilon,T+\varepsilon$, $H_{[1]}(\cdot,-T-\varepsilon,\cdot)$ (left segment) and $H(\cdot,T+\varepsilon,\cdot)$ (right segment) degenerate to $f^{\alpha}$ and $f^{\beta}$ respectively.\\
(c) reveals a family $H_{[2]}$ with two-parameter $s^{\underline{2}},\,  s^{\underline{1}}$ and one variable $t$ such that $H_{[2]}(\cdot,\cdot,-T-\varepsilon,\cdot)$ (left face), $H_{[2]}(\cdot,\cdot,T+\varepsilon,\cdot)$ (right face) degenerate to two points;  $H_{[2]}(\cdot,0,\cdot,\cdot)$ (front face), $H_{[2]}(\cdot,1,\cdot,\cdot)$ (back face) degenerate to two segments. As the number of parameters increases, $H_{[\ell]}$ has the analogous degeneration.  $\widetilde{\mathcal{G}}^{[\ell]}$ is also similar.

We will see this is a geometric reminiscence of higher morphisms in higher categories to be discussed in \S \ref{strictcat}. 

Given two MS pairs $(f^{\alpha},g^{\alpha})$ and $(f^{\beta},g^{\beta})$, a higher homotopy $(H_{[\ell]},G_{[\ell]})$ between $(f^{\alpha},g^{\alpha})$ and $(f^{\beta},g^{\beta})$ means that $H_{[\ell]}$ is a higher homotopy between $f^{\alpha}$ and $f^{\beta}$ and $G_{[\ell]}$ is a higher homotopy between $g^{\alpha}$ and $g^{\beta}$.

\begin{figure}
  \centering
  \includegraphics[scale=0.6]{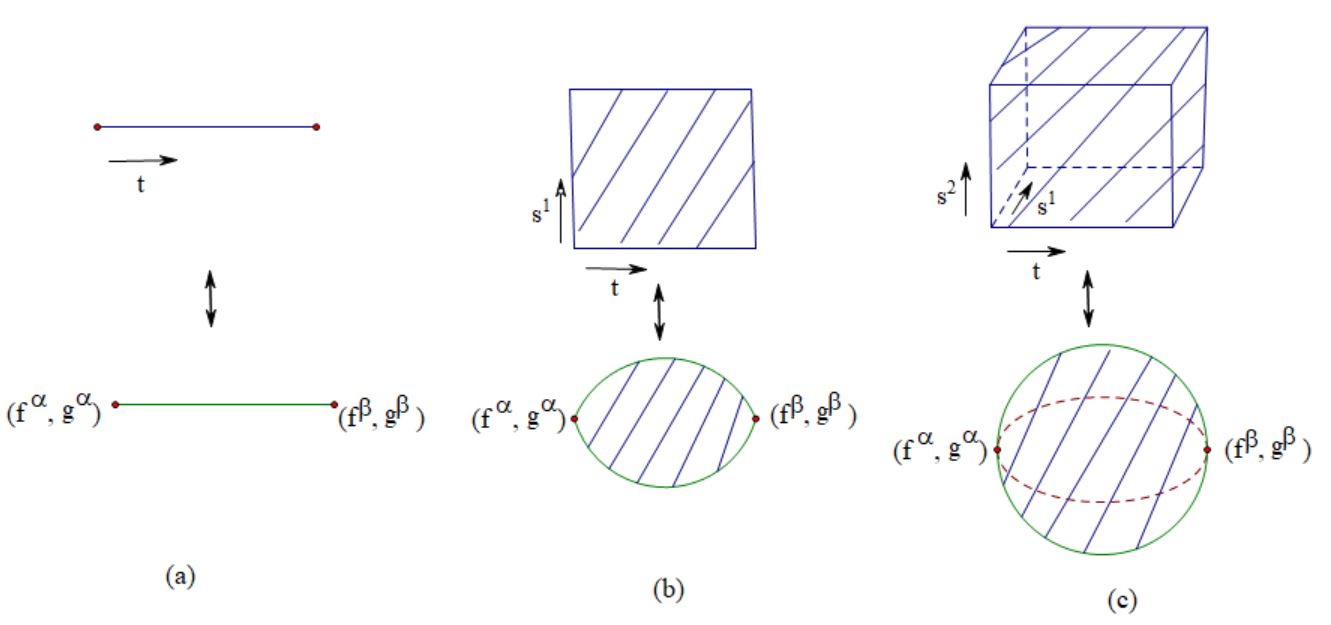}
  \caption{parameters $s^i$ v.s. variable $t$}
  \label{parameter space}
\end{figure}

Now we recall the Banach manifold $\mathcal{P}^{1,2}_{p,q}$ connecting two given points $p$ and $q$ and the fact that the linearization of the differential operator coming from the negative gradient flow equation along a solution path turns out to be a Fredholm operator.

\textbf{Banach manifold $\mathcal{P}^{1,2}_{p,q}$}

To introduce $\mathcal{P}^{1,2}_{p,q}$ briefly, here we follow \cite{eliasson} and \cite{schwarz}, and detailed properties and proofs can be found there.

Let $\overline{\R}:=\R\cup\{\pm \infty\}$ and $p,q\in M$. By $h\in C^{\infty}(\overline{\mathbb{R}}, M)$, we mean a smooth map from $\mathbb{R}$ to $M$ such that $\lim\limits_{t\rightarrow-\infty}h(t)=p$,
$\lim\limits_{t\rightarrow +\infty}h(t)=q$.  Let $\mathcal{D}\subset TM$ be a neighbourhood of zero section in $TM$. Because $\overline{\R}$ is contractible, $h^{\ast}\mathcal{D}$ is a trivial bundle. Then we let $\varphi:h^{\ast}\mathcal{D}\rightarrow \overline{\R}\times \R^{n}$ be a global smooth trivialization of $ h^{\ast}\mathcal{D}$.
For $\xi \in W^{1,2}(\overline{\mathbb{R}}, \R^{n})$ and $\xi(-\infty)=\xi(+\infty)=0$, $\exp_{h}\xi:\overline{\R}\rightarrow M$ describes a $W^{1,2}$-path connecting $p$ and $q$ near $h$. Set
$$\mathcal{P}_{p,q}^{1,2}=\Big\{exp_{h}\xi \Big| \xi\in W^{1,2}(\overline{\mathbb{R}}, \R^{n}), \xi(-\infty)=\xi(+\infty)=0 ,\forall h\in C^{\infty} \Big\}$$
Clearly it defines $\mathcal{P}_{p,q}^{1,2}$ with a covering which is responsible for its Banach manifold structure. The details can be found in \cite{schwarz} p. 212.

\begin{theorem}(\cite{schwarz}, p. 25)
The set $\mathcal{P}_{p,q}^{1,2}$ is an infinitely dimensional $C^{\infty}$ Banach manifold. Furthermore, it is independent of the choice of the Riemannian metric $g$.
\end{theorem}

There are two natural Banach bundles $T\mathcal{P}_{p,q}^{1,2}$ and $\mathcal{E}$ over $\mathcal{P}_{p,q}^{1,2}$ whose fibers are $T_{u}\mathcal{P}_{p,q}^{1,2}=W^{1,2}(u^{\ast}TM)$ and $\mathcal{E}_{u}:=L^{2}(u^{\ast}TM)$ respectively for any $u \in \mathcal{P}_{p,q}^{1,2}$.

\textbf{Linearization of the gradient flow equation}

For $H_{[\ell]}:[0,1]^{\ell}\times \R \times M \rightarrow \R$ and $G_{[\ell]}$ as defined above, we consider the map
\begin{equation}
\begin{split}
F_{[\ell]} :[0,1]^{\ell} \times \mathcal{P}_{p,q}^{1,2}&\longrightarrow \mathcal{E}\\
               (s_{\ell},\cdots,s_{1},u)&\longmapsto \partial_{t}u+\nabla_{G_{[\ell]}(s_{\ell},\cdots,s_{1},t)}H_{[\ell]}(s_{\ell},\cdots,s_{1},t,u).
               \nonumber
\end{split}
\end{equation}
The equation $F_{[\ell]}(s_{\ell},\cdots,s_{1},u)=\partial_{t}u+\nabla_{G_{[\ell]}(s_{\ell},\cdots,s_{1},t)}H_{[\ell]}(s_{\ell},\cdots,s_{1},t,u)=0$ is exactly the negative gradient flow equation which plays a fundamental role in our later constructions.
In the local coordinate $u_{1},\cdots,u_{n}$ on $M$ and letting $G_{[\ell]}(s_{\ell},\cdots,s_{1},t)(u)= g_{[\ell]}(s_{\ell},\cdots,s_{1},t,u)$, it has the form:
$$\partial_{t}u_{i}(t)+\sum\limits_{j=1}^{n}g^{ij}_{[\ell]}(s_{\ell},\cdots,s_{1},t,u)
\frac{\partial H_{[\ell]}(s_{\ell},\cdots,s_{1},t,u)}{\partial u_{j}}=0,\ \ i=1,\cdots,n.$$
\indent Taking the linearization along a solution $u$ , we obtain:\\
$$\partial_{t}\xi_{i}(t)
+\sum\limits_{j}\sum\limits_{k}\frac{\partial g^{ij}_{[\ell]}(s_{\ell},\cdots,s_{1},t,u)}{\partial u_{k}}\xi_{k}\frac{\partial H_{[\ell]}(s_{\ell},\cdots,s_{1},t,u)}{\partial u_{j}}
+\sum\limits_{j}\sum\limits_{k}g^{ij}_{[\ell]}(s_{\ell},\cdots,s_{1},t,u)
\frac{\partial^{2}H_{[\ell]}(s_{\ell},\cdots,s_{1},t,u)}{\partial u_{k}\partial u_{j}}\xi_{k}=0$$
Let $$A_{ik}(s_{\ell},\cdots,s_{1},t):=\sum\limits_{j}\left(\frac{\partial g^{ij}_{[\ell]}(s_{\ell},\cdots,s_{1},t,u)}{\partial u_{k}}\frac{\partial H_{[\ell]}(s_{\ell},\cdots,s_{1},t,u)}{\partial u_{j}}
+g^{ij}_{[\ell]}(s_{\ell},\cdots,s_{1},t,u)
\frac{\partial^{2}H_{[\ell]}(s_{\ell},\cdots,s_{1},t,u)}{\partial u_{k}\partial u_{j}}\right),$$
then we have
$$\partial_{t}\xi_{i}(t)+\sum\limits_{k=1}^{n}A_{ik}(s_{\ell},\cdots,s_{1},t)\xi_{k}=0,\ \ i=1,\cdots,n.$$
Let $A(-\infty):=\lim\limits_{t\rightarrow -\infty}A(s_{\ell},\cdots,s_{1},t)=\lim\limits_{t\rightarrow -\infty}\big( A_{ik}(s_{\ell},\cdots,s_{1},t)\big)_{i,k}$ and
$A(+\infty):=\lim\limits_{t\rightarrow +\infty}A(s_{\ell},\cdots,s_{1},t)=\lim\limits_{t\rightarrow +\infty}\big( A_{ik}(s_{\ell},\cdots,s_{1},t)\big)_{i,k}.$
Then $A(-\infty)$ is the Hessian of the Morse function $f^{\alpha}$ at $p$, $A(+\infty)$ is the Hessian of the Morse function $f^{\beta}$ at $q$.

\textbf{Adjoint}

Let $B(\mathbf{s},\cdot)\in C^{0}(\R, End(\R^{n}))$ with $\mathbf{s}=(s^{\underline{\ell}},\cdots,s^{\underline{1}})$ and $\lim\limits_{t\rightarrow \pm\infty}B(\mathbf{s},t)=B^{\pm}$ be non-degenerate and self-adjoint.
Let $D_{B}: W^{1,2}(\R,\R^{n})\rightarrow L^{2}(\R,\R^{n})$, $D_{B}\xi(t)=\partial_{t}\xi(t)+B(\mathbf{s},t)\cdot\xi(t)$.
If for any $w(t)\in (Im(D_{B}))^{\bot}\subset L^{2}(\R, \R^{n})$,
$$ 0 =\int_{-\infty}^{\infty}G_{[\ell]}(\mathbf{s},t,u)( D_{B}\xi(t),w(t)) dt
     =\int_{-\infty}^{\infty}G_{[\ell]}(\mathbf{s},t,u)( \partial_{t}\xi(t),w(t)) dt
   +\int_{-\infty}^{\infty}G_{[\ell]}(\mathbf{s},t,u)( B(\mathbf{s},t)\cdot\xi(t),w(t)) dt,$$
we have $$\int_{-\infty}^{\infty}G_{[\ell]}(\mathbf{s},t,u)( \partial_{t}\xi(t),w(t)) dt=-\int_{-\infty}^{\infty}G_{[\ell]}(\mathbf{s},t,u)( B(\mathbf{s},t)\cdot\xi(t),w(t)) dt =-\int_{-\infty}^{\infty}G_{[\ell]}(\mathbf{s},t,u)( \xi(t),(B(\mathbf{s},t))^{T}w(t)) dt.$$
This means that $w(t)$ has the weak derivative $(B(\mathbf{s},t))^{T}w(t)$.
Therefore, $w(t)\in W^{1,2}$ and satisfies the following ODE $$\partial_{t}w(t)=(B(\mathbf{s},t))^{T}w(t),$$ then $w(t)\in KerD_{-B^{T}}$ with $-B^{T}=-(B(\mathbf{s},t))^{T}$. It
will be used to study the regularity of the moduli space in later section.

\textbf{Fredholm properties}

{More details about Fredholm theory can be found in, say \cite{mcduff-salamon} (Appendix A).} In our case, let $F_{A}\xi(t)=\dot{\xi}(t)+A(t)\cdot\xi(t)$,
where $\xi\in W^{1,2}(\mathbb{R},\mathbb{R}^{n})$, 
and $A(t)\in GL(n,\mathbb{R}),\forall t\in \overline{\mathbb{R}}$ which depends on $t$ continuously. Define $ind A(t)$ to be the negative index of inertia of matrix $A(t)$. And $A(\pm\infty)$ is the Hessian of the corresponding Morse function at its associated critical points which is well known to be self-adjoint, and its index is just the Morse index of the critical point.

\begin{proposition}
$F_{A}$ is a Fredholm operator with index $ind F_{A}=ind A(-\infty)-ind A(+\infty)$. Moreover, $F_{[\ell]A}\xi(t)=\dot{\xi}(t)+A(\mathbf{s},t)\cdot\xi(t)$ is also Fredholm with index $indF_{[\ell]A}=ind A(-\infty)-ind A(+\infty)+\ell$ for any $\ell\geq 1$.
\end{proposition}
The proof can be adapted from \cite{schwarz} (p. 38). In the proof, a key role is played by the exponential decay of the solution when $t\rightarrow \pm \infty$.  One may notice that for $indF_{[\ell]A}$, the index increases by $\ell$ which is basically coming from the number of parameters. 

\textbf{Regularity and smoothness}

Next, we want to prove that the moduli space of gradient flow trajectories is smooth with expected dimension under regular conditions. We fix two MS pairs $(f^{\alpha},g^{\alpha})$ and $(f^{\beta},g^{\beta})$ once and for all, $p\in Cr f^{\alpha}$, $q\in Cr f^{\beta}$.
Let's start with a useful proposition.

\begin{proposition}\label{usefulprop} (\cite{schwarz}, p. 42)
Let $\mathcal{G}$ and $\mathcal{P}$ be two Banach manifolds, and $E$ be a Banach bundle over $\mathcal{P}$ with fiber $\mathbb{E}$.
Let $\Phi:\mathcal{G}\times \mathcal{P}\longrightarrow E$ be a smooth map  with $\Phi_{g}(\cdot):=\Phi(g,\cdot)$ ($\forall g\in\mathcal{G}$) smooth section of the bundle ${E}$ satisfying the following:
there is a countable trivialization $\{(U,\psi)|\, U\subset \mathcal{P},\, \psi :E|_{U}\cong U\times \mathbb{E}\}$ such that for any $(U,\psi)$,

(i) $0$ is a regular value of $pr_{2}\circ \psi \circ \Phi:\mathcal{G}\times U\rightarrow \mathbb{E}$,

(ii) $pr_{2}\circ \psi \circ \Phi_{g}: U\rightarrow \mathbb{E}$ is a Fredholm map of index r.

Then there exists a Baire category set $\Sigma\subset\mathcal{G}$ such that for any $g\in\Sigma$, $\Phi_{g}^{-1}(0)$ is a smooth manifold and its dimension is the Fredholm index r. 
An element in $\Sigma$ will be called a generic element in $\mathcal{G}$.
\end{proposition}
Here a Baire category set means a countable intersection of open dense sets which turns out to be dense as well.
In our case the Baire set appears as follows (compare \cite{frauenfelder-nicholls}):
Let $\mathcal{G}$ be $\widetilde{\mathcal{G}}^{[\ell]}$ and $\mathcal{P}$ be $[0,1]^{\ell}\times\mathcal{P}_{p,q}^{1,2}$.
We define the following map:
\begin{equation}\label{Fhat}
\begin{split}
\widehat{F}_{[\ell]}:\widetilde{\mathcal{G}}^{[\ell]}\times [0,1]^{\ell}\times \mathcal{P}^{1,2}_{p,q}&\longrightarrow {pr^{\ast}}\mathcal{E}\\
(G_{[\ell]},s^{\underline{\ell}},\cdots,s^{\underline{1}},u)&\longmapsto
\dot{u}(t)+\nabla_{G_{[\ell]}(s^{\underline{\ell}},\cdots,s^{\underline{1}})}H_{[\ell]}(s^{\underline{\ell}},\cdots,s^{\underline{1}},t,u(t)),
\end{split}
\end{equation}
where $pr:[0,1]^{\ell}\times \mathcal{P}^{1,2}_{p,q}\longrightarrow \mathcal{P}^{1,2}_{p,q}$ is the projection map.
In fact, $\widehat{F}_{[\ell]}$ corresponds to $\Phi$ in Proposition \ref{usefulprop}. 
{For consistency, we rewrite $\mathcal{G}_{g_{0}}$ as $\widetilde{\mathcal{G}}^{[-1]}$, and
\begin{equation}
\begin{split}
\widehat{F}:\widetilde{\mathcal{G}}^{[-1]}\times \mathcal{P}^{1,2}_{p,q}&\longrightarrow \mathcal{E}\\
(g,u)&\longmapsto
\dot{u}(t)+\nabla_{g}f(u(t)), \nonumber
\end{split}
\end{equation}
with $f$ the fixed Morse function on $M$.  }\\
\indent Let $\Pi: \widehat{F}_{[\ell]}^{-1}(0)\rightarrow \widetilde{\mathcal{G}}^{[\ell]}$ be the projection to $\widetilde{\mathcal{G}}^{[\ell]}$ component, and $\widetilde{\Sigma}_{[\ell]}:=\big\{ G_{[\ell]}\in \widetilde{\mathcal{G}}^{[\ell]}  \big| G_{[\ell]}$ is a regular value of $\Pi \big\}$. By Sard's theorem(\cite{mcduff-salamon}, Appendix, p. 547), we know $\widetilde{\Sigma}_{[\ell]}$ is a Baire category set and it is called a {\it generic set}. In this sense, we call $G_{[\ell]}$ is {\it generic} with respect to $H_{[\ell]}$.
Sometimes we write $\widetilde{\Sigma}_{[0]}=\widetilde{\Sigma}$ for brevity. And note that  $\widetilde{\Sigma}_{[\ell]}$ depends on the critical points $p,q$ and $H_{[\ell]}$.

\begin{theorem}\label{regularthm}
Let $f$ be a Morse function on $M$. For a generic Riemannian metric $g\in\Sigma$ ,
the moduli space $\mathcal{\widetilde{M}}^{f}_{p,q}:=F^{-1}_{g}(0)$ is a smooth submanifold of $\mathcal{P}^{1,2}_{p,q}$ of dimension ($indp-indq$) {with $F_{g}=\widehat{F}(g,\cdot)$, and the {\bf reduced} moduli space $\mathcal{M}^{f}_{p,q}=\mathcal{\widetilde{M}}^{f}_{p,q}/\R$ is a smooth submanifold of $\mathcal{P}^{1,2}_{p,q}$ of dimension ($indp-indq-1$) due to the time translation invariance.}
\end{theorem}

 Now this is a standard fact in Morse theory, and \cite{frauenfelder-nicholls} (p. 76-77) has a nice proof (See also \S 2).

{\bf A reminder about the notations}: later on the homotopies between  Morse functions will depend on the time variable $t$ and extra parameters $\mathbf{s}$, to lighten the notations we will signify the function involved at the superscript of the corresponding moduli spaces of gradient flow lines without using "widetilde" as in the above theorem. Hope that it will not gear confusion.

\begin{definition}\textbf{(regular)}
 A homotopy $H:=H_{[0]}(\cdot,\cdot):=H^{\alpha \beta}_{[0]}(\cdot,\cdot)$ is called {\it regular} if for any $\mathbb{R}$-critical point $u_{0}$ of $H_{[0]}$ (i.e., $\nabla H_{[0]}(t,u_{0})=0$, for all $t\in \mathbb{R}$), $\frac{\partial}{\partial t}+Hessian(H_{[0]}(t))(u_{0}): W^{1,2}(u_{0}^{\ast}TM)\longrightarrow L^{2}(u_{0}^{\ast}TM)$ is onto.\\ 
 \indent Note that $Hessian(H_{[0]}(t))(u_{0})$ is the Hessian of $H_{[0]}$ at $u_{0}$ which is metric $G_{[0]}$-dependent. Hence it is reasonable to say that $H_{[0]}$ is regular under $G_{[0]}$. Additionally, $\mathbb{R}$-critical point $u_{0}$ is actually a constant trajectory.
\end{definition}

\begin{remark}
$1.$ There always exists regular homotopy $H^{\alpha \beta}_{[0]}(\cdot,\cdot)$ between the given Morse functions $f^{\alpha}$ and $f^{\beta}$. The idea (\cite{schwarz}, p. 46)  is that if there exists a non-regular $\R$-critical point $u_{1}$, one can always construct a function $k_{1,t} \in C^{\infty}(\R\times M, \R)$ supporting on $u_{1}$ and $dk_{1,t}(u_{1})\neq 0$. We can modify the original non-regular homotopy at finitely many such non-regular $\R$-critical points $(u_{j})$ by adding corresponding $k_{j,t}$s to make it regular.

$2.$ The regularity is an open condition.
\end{remark}

\begin{theorem}\label{regularthm0}
Assume that $(f^{\alpha},g^{\alpha})$ and $(f^{\beta},g^{\beta})$ are two MS pairs.
Then there exists a generic  interpolation homotopy $(H_{[0]},G_{[0]})$ between pairs $(f^{\alpha},g^{\alpha})$ and $(f^{\beta},g^{\beta})$ such that the moduli space $\mathcal{M}^{H_{[0]}}_{p,q}:=F_{G_{[0]}}^{-1}(0) $ is a smooth submanifold of $\mathcal{P}^{1,2}_{p,q}$ of dimension $ind p- ind q$, where $F_{G_{[0]}}=F_{[0]G_{[0]}}=\widehat{F}_{[0]}(G_{[0]},\cdot)=\widehat{F}(G_{[0]},\cdot)$.
In other word, for generic $G_{[0]}\in\widetilde{\Sigma}$, the moduli space $\mathcal{M}^{H_{[0]}}_{p,q}$ is a smooth submanifold of $\mathcal{P}^{1,2}_{p,q}$ of dimension $indp-indq$.
\end{theorem}
\begin{remark}
Here "generic" means that for such pairs, $H_{[0]}$ is regular with respect to $G_{[0]}$, and at the same time $G_{[0]}$ is generic with respect to $H_{[0]}$ in the sense as above.
\end{remark}
The proof of Theorem \ref{regularthm0} combines the methods in (\cite{schwarz}, p. 47-49) and (\cite{frauenfelder-nicholls}, p. 68-73). The idea of the proof is as follows: We choose any finite homotopy $\overline{G}_{[0]}$ from $g^{\alpha}$ to $g^{\beta}$,  then there exists a finite homotopy $H_{[0]}$ which is regular under $\overline{G}_{[0]}$. Such $H_{[0]}$ satisfies the conditions of the Proposition \ref{usefulprop}. Then a generic $G_{[0]}$ with respect to $H_{[0]}$  can be found to assure that  $F_{[0]G_{[0]}}^{-1}(0)$ is a smooth submanifold with expected dimension by using the proposition. { The details are omitted.}


For general case, we can define the regularity of $H_{[\ell]}$ similarly.

\begin{definition}\label{regular}
A higher homotopy $H_{[\ell]}:=H^{\alpha \beta}_{[\ell]}$ is called {\it regular} with respect to the Riemannian metric $G_{[\ell]}$ if for any $\mathbf{s}=(s^{\underline{\ell}},\cdots,s^{\underline{1}})\in [0,1]^{\ell}$, $H_{[\ell]}(s^{\underline{\ell}},\cdots,s^{\underline{1}},t,u)$ is regular, i.e., for any $\R$- or $(\mathbf{s},t)-$critical point $u_{0}$ of $H_{[\ell]}(s^{\underline{\ell}},\cdots,s^{\underline{1}},t,u)$:
$\nabla_{G_{[\ell]}} H_{[\ell]}(s^{\underline{\ell}},\cdots,s^{\underline{1}},t,u_{0})=0$, for all $t\in \mathbb{R}$, we have $\frac{\partial}{\partial t}+Hessian(H_{[\ell]}(s^{\underline{\ell}},\cdots,s^{\underline{1}},t))(u_{0})$ is onto.\\
\indent Note that $Hessian(H_{[\ell]}(s^{\underline{\ell}},\cdots,s^{\underline{1}},t))(u_{0})$ is the Hessian of $H_{[\ell]}$ at $u_{0}$ which is $G_{[\ell]}$-dependent.
\end{definition}



\begin{theorem}\label{modulispaceellcase}
Let $(f^{\alpha},g^{\alpha})$ and $(f^{\beta},g^{\beta})$ be two given MS pairs, and $p\in Cr f^{\alpha}$, $q\in Cr f^{\beta}$.
Then there exists a generic $(H_{[\ell]},G_{[\ell]})$ such that $\mathcal{M}_{p,q}^{H_{[\ell]}}:=F_{[\ell]G_{[\ell]}}^{-1}(0)$ is a ($ind p- ind q+\ell$)-dimensional submanifold of $[0,1]^{\ell}\times\mathcal{P}^{1,2}_{p,q}$, where $F_{[\ell]G_{[\ell]}}(\cdots)=\widehat{F}_{[\ell]}(G_{[\ell]},\cdots)$.
\end{theorem}

{It is a generalization of Theorem $2$ in \cite{schwarz} (p. 50).} The idea of the proof 
is as follows: From Theorem \ref{regularthm0}, there exist generic finite homotopies $(H_{[0]0},G_{[0]0})$ and $(H_{[0]1},G_{[0]1})$ between $(f^{\alpha},g^{\alpha})$ and $(f^{\beta},g^{\beta})$. Here "generic" means that $H_{[0]0}$ is regular under $G_{[0]0}$, $H_{[0]1}$ is regular under $G_{[0]1}$, at the same time $G_{[0]0}$ is generic w.r.t. $H_{[0]0}$ and $G_{[0]1}$ is generic w.r.t. $H_{[0]1}$.
Based on the following Lemma \ref{goingtohigher}, there exists a finite higher homotopy $H_{[1]}$ which is regular under $\overline{G}_{[1]}$, where $\overline{G}_{[1]}$ is any finite homotopy from $G_{[0]0}$ to $G_{[0]1}$. Then we can check that $H_{[1]}$ satisfies the conditions of Proposition \ref{usefulprop}. {We should mention that the regularity property for $H_{1}$ assures that the second condition in the Proposition \ref{usefulprop} holds.} A generic $G_{[1]}$ can be found to guarantee  that $F_{[1]G_{[1]}}^{-1}(0)$ is a submanifold with expected dimension by Proposition \ref{usefulprop}. {By iterating the arguments, there exist $(H_{[\ell]},G_{[\ell]})$ such that $F_{[\ell]G_{[\ell]}}^{-1}(0)$ is a ($ind p- ind q+\ell$)-dimensional submanifold of $[0,1]^{\ell}\times\mathcal{P}^{1,2}_{p,q}$.}
In the process, the following two lemmas play significant roles.

Recall that we have already chosen some finite homotopy $\overline{G}_{[\ell+1]}$ and homotopy $H_{[\ell+1]}$ regular under $\overline{G}_{[\ell+1]}$, the goal of the following lemmas is to find $G_{[\ell +1]}$ in the neighborhood of $\overline{G}_{[\ell+1]}$ such that the pair $(H_{[\ell+1]}, G_{[\ell +1]})$ is generic.

\begin{lemma}
Let $(G_{[\ell]}, H_{[\ell]})$ be generic, then the homotopy $H_{[\ell+1]}$ satisfying that there exists an $r_{0}\in [0,1]$ such that $H_{[\ell+1]}(r_{0},\cdots)=H_{[\ell]}$ is regular with respect to $G_{[\ell+1]}$.
\end{lemma}
\begin{proof}
Let us start from the following commutative diagram at any $r=r_0\in [0,1]$ {which is the first coordinate of $[0,1]^{\ell+1}$}.

\begin{tikzcd}
\widetilde{\mathcal{G}}^{[\ell+1]}\times [0,1]^{\ell}\times \mathcal{P}_{p,q}^{1,2}  \arrow[r, "id\times em_{r=r_{0}}\times id"] \arrow[dr, "f_{r=r_{0}}:=ev_{r_{0}}\times id \times id"] & \widetilde{\mathcal{G}}^{[\ell+1]}\times [0,1]^{\ell+1}\times \mathcal{P}_{p,q}^{1,2}  \arrow[r, "\widehat{F}_{[\ell+1]}"]  \arrow[d, "ev_{r_{0}}\times pr \times id"] & pr^{\ast}_{[\ell+1]}\mathcal{E}=[0,1]^{\ell+1}\times \mathcal{E} \arrow[d, "pr"]\\
        & \widetilde{\mathcal{G}}^{[\ell]}\times [0,1]^{\ell}\times \mathcal{P}_{p,q}^{1,2}   \arrow[r, "\widehat{F}_{[\ell]_{r_{0}}}"] & pr^{\ast}_{[\ell]}\mathcal{E}=[0,1]^{\ell}\times \mathcal{E}
\end{tikzcd}\\
{where $ev_{r_{0}}:\widetilde{\mathcal{G}}^{[\ell+1]}\longrightarrow \widetilde{\mathcal{G}}^{[\ell]}$, $ev_{r_{0}}G_{[\ell+1]}(r,\cdots)=G_{[\ell+1]}(r_{0},\cdots)$ is the evaluation map and $em$ is the embedding map.}

Differentiating the maps we get the following communicative diagram at $r=r_{0}$.

\begin{tikzcd} \label{tang comm diag}
T_{G_{[\ell+1]}}\widetilde{\mathcal{G}}^{[\ell+1]}\times T_{\mathbf{s}}[0,1]^{\ell}\times T_{u}\mathcal{P}_{p,q}^{1,2}  \arrow[r, "d(id\times em_{r=r_{0}}\times id)"] \arrow[dr, "d(f_{r=r_{0}}):=d(ev_{r_{0}}\times id \times id)"] & T_{G_{[\ell+1]}}\widetilde{\mathcal{G}}^{[\ell+1]}\times T_{(r=r_{0},\mathbf{s})}[0,1]^{\ell+1}\times T_{u}\mathcal{P}_{p,q}^{1,2}  \arrow[r, "d\widehat{F}_{[\ell+1]}"]  \arrow[d, "d(ev_{r_{0}}\times pr \times id)"] & T_{(r_{0},\mathbf{s},u,\widetilde{v})}pr^{\ast}_{[\ell+1]}\mathcal{E} \arrow[d, "d pr"]\\
        &T_{G_{[\ell]}}\widetilde{\mathcal{G}}^{[\ell]}\times T_{\mathbf{s}}[0,1]^{\ell}\times T_{u}\mathcal{P}_{p,q}^{1,2}   \arrow[r, "d\widehat{F}_{[\ell]_{r_{0}}}"] & T_{(\mathbf{s},u,v)}pr^{\ast}_{[\ell]}\mathcal{E}
\end{tikzcd}\\
where $\widetilde{v}=\partial_{t}u+\nabla_{G_{[\ell+1]}}H_{[\ell+1]}$ and $v=\partial_{t}u+\nabla_{G_{[\ell]}}H_{[\ell]}$,  $T_{(\mathbf{s},u,v)}pr^{\ast}_{[\ell]}\mathcal{E}=T_{(\mathbf{s},u,v)}([0,1]^{\ell}\times \mathcal{E})$ and $T_{(r,\mathbf{s},u,v)}pr^{\ast}_{[\ell+1]}\mathcal{E}=T_{(r,\mathbf{s},u,v)}([0,1]^{\ell+1}\times \mathcal{E})$.

We introduce three maps $DF_{2\textrm{I}}$, $DF_{2\textrm{II}}$ and $DF_{2\textrm{III}}$. Here $\mathbb{E}$ means the fiber of $\mathcal{E}$, $pr^{\ast}_{[\ell]}\mathcal{E}$ and $pr^{\ast}_{[\ell+1]}\mathcal{E}$.

$DF_{2\textrm{I}}:T_{\mathbf{s}}[0,1]^{\ell}\times T_{u}\mathcal{P}_{p,q}^{1,2} \longrightarrow  T_{(\mathbf{s},u,v)}pr^{\ast}_{[\ell]}\mathcal{E}\stackrel{pr}{\longrightarrow} \mathbb{E}$,

$DF_{2\textrm{II}}:T_{\mathbf{s}}[0,1]^{\ell}\times T_{u}\mathcal{P}_{p,q}^{1,2} \stackrel{incl_{r=r0}\times id}{\longrightarrow} T_{(r_{0},\mathbf{s})}[0,1]^{\ell+1}\times T_{u}\mathcal{P}_{p,q}^{1,2}  \longrightarrow T_{(r_{0},\mathbf{s},u,\widetilde{v})}pr^{\ast}_{[\ell+1]}\mathcal{E}\stackrel{pr}{\longrightarrow} \mathbb{E}$,

$DF_{2\textrm{III}}:T_{(r,\mathbf{s})}[0,1]^{\ell+1}\times T_{u}\mathcal{P}_{p,q}^{1,2}  \longrightarrow T_{(r,\mathbf{s},u,\widetilde{v})}pr^{\ast}_{[\ell+1]}\mathcal{E}\stackrel{pr}{\longrightarrow} \mathbb{E}$.

Following from the second diagram \ref{tang comm diag}, we have $DF_{2\textrm{I}}=DF_{2\textrm{II}}$. Due to the regularity of $H_{[\ell]}$, the map $DF_{2\textrm{I}}$ is onto. Consequently, $DF_{2\textrm{II}}$ is onto as well which in turn  implies $DF_{2\textrm{III}}$ is so, because $[0,1]^{\ell}\subset [0,1]^{\ell+1}$.
\end{proof}
\begin{lemma}\label{goingtohigher}
Let $(G_{[\ell]0}, H_{[\ell]0})$ and $(G_{[\ell]1}, H_{[\ell]1})$ ($\ell\in\mathbb{Z}_{\geqslant 0}$) be generic. Then there exists a homotopy $G_{[\ell+1]}$ from $G_{[\ell]0}$ to $G_{[\ell]1}$ which is generic with respect to $H_{[\ell+1]}$, i.e. $G_{[\ell+1]}$ is a regular value of  $\Pi: \widehat{F}_{[\ell+1]}^{-1}(0)\longrightarrow \widetilde{\mathcal{G}}^{[\ell+1]}$ with $\widehat{F}_{[\ell+1]}$ defined in terms of $H_{[\ell+1]}$: $\widehat{F}_{[\ell+1]}(G_{[\ell+1]},r,\mathbf{s},u)=(r,\mathbf{s},\partial_{t}u+\nabla_{G_{[\ell+1]}}H_{[\ell+1]}(r,\mathbf{s},u))$ where $(r,\mathbf{s})\in [0,1]^{\ell+1}$.
\end{lemma}
\begin{proof}
Based on the two commutative diagrams from the last lemma,
we introduce 

\begin{tikzcd}
\widehat{fiber}_{r=r_{0}} \arrow[r, " "] & \widetilde{\mathcal{G}}^{[\ell+1]}\arrow[d, "ev_{r_{0}}"]\\
 & \widetilde{\mathcal{G}}^{[\ell]}
\end{tikzcd}      \quad \quad
\begin{tikzcd}
fiber_{r=r_{0}} \arrow[r, " "] & T_{G_{[\ell+1]}}\widetilde{\mathcal{G}}^{[\ell+1]}\arrow[d, "d ev_{r_{0}}"]\\
 & T_{G_{[\ell]}}\widetilde{\mathcal{G}}^{[\ell]}
\end{tikzcd}\\
where $\widehat{fiber}_{r=r_{0}}$ means the inverse image of $ev_{r_0}$ at $G_{[\ell]}$ which equals to $\big\{ G_{[\ell+1]}\big| G_{[\ell+1]}(r_{0},\cdots)= G_{[\ell]}\big\}$ with corresponding $fiber_{r=r_{0}}$ for tangent spaces. They induce the following

\begin{tikzcd}
\widehat{fiber'}_{r=r_{0}} \arrow[r, " "] & f_{r=r_{0}}^{-1}\widehat{F}_{[\ell]}^{-1}(0)\arrow[d, "ev_{r_{0}}"]\\
 & \widehat{F}_{[\ell]}^{-1}(0)
\end{tikzcd}\\
\begin{tikzcd}
fiber'_{r=r_{0}} \arrow[r, " "] & T_{(G_{[\ell+1]},\mathbf{s},u)}(f_{r=r_{0}}^{-1}\widehat{F}_{[\ell]}^{-1}(0))\arrow[d, "d (ev_{r_{0}}\times id\times id)"] &\subset T_{(G_{[\ell+1]},\mathbf{s},u)}(\widetilde{\mathcal{G}}^{[\ell+1]}\times [0,1]^{\ell} \times \mathcal{P}_{p,q}^{1,2})\\
 & T_{(G_{[\ell]},\mathbf{s},u)}(\widehat{F}_{[\ell]}^{-1}(0)) &\subset T_{(G_{[\ell]},\mathbf{s},u)}(\widetilde{\mathcal{G}}^{[\ell]}\times [0,1]^{\ell} \times \mathcal{P}_{p,q}^{1,2})
\end{tikzcd}\\
The equality $\widehat{fiber}_{r=r_{0}}= \widehat{fiber'}_{r=r_{0}}$ implies that $fiber_{r=r_{0}}\cong fiber'_{r=r_{0}}$.
Let
\begin{equation}
\begin{split}
f_{1} &:= d\Pi_{1}: T_{(G_{[\ell]},\mathbf{s},u)}(\widehat{F}_{[\ell]}^{-1}(0))\subset T_{G_{[\ell]}}\widetilde{\mathcal{G}}^{[\ell]}\times T_{\mathbf{s}}[0,1]^{\ell}\times T_{u}\mathcal{P}_{p,q}^{1,2} \longrightarrow T_{G_{[\ell]}}\widetilde{\mathcal{G}}^{[\ell]}; \nonumber\\
f_{2} &:= d\Pi_{2}:T_{(G_{[\ell+1]},\mathbf{s},u)}(f_{r=r_{0}}^{-1}\widehat{F}_{[\ell]}^{-1}(0))\subset T_{G_{[\ell+1]}}\widetilde{\mathcal{G}}^{[\ell+1]}\times T_{\mathbf{s}}[0,1]^{\ell}\times T_{u}\mathcal{P}_{p,q}^{1,2} \longrightarrow T_{G_{[\ell+1]}}\widetilde{\mathcal{G}}^{[\ell+1]};\nonumber\\
f_{3} &:= d\Pi_{3}:T_{(G_{[\ell+1]},r_{0},\mathbf{s},u)}(\widehat{F}_{[\ell+1]}^{-1}(0))=T_{(G_{[\ell+1]},r,\mathbf{s},u)}(\widehat{F}_{[\ell+1]}^{-1}(0))\big|_{r_{0}}\subset T_{G_{[\ell+1]}}\widetilde{\mathcal{G}}^{[\ell+1]}\times T_{(r_{0},\mathbf{s})}[0,1]^{\ell+1}\times T_{u}\mathcal{P}_{p,q}^{1,2} \longrightarrow T_{G_{[\ell+1]}}\widetilde{\mathcal{G}}^{[\ell+1]};\nonumber\\
f_{4} &:= d\Pi_{4}:T_{(G_{[\ell+1]},r,\mathbf{s},u)}(\widehat{F}_{[\ell+1]}^{-1}(0))\subset T_{G_{[\ell+1]}}\widetilde{\mathcal{G}}^{[\ell+1]}\times T_{(r,\mathbf{s})}[0,1]^{\ell+1}\times T_{u}\mathcal{P}_{p,q}^{1,2} \longrightarrow T_{G_{[\ell+1]}}\widetilde{\mathcal{G}}^{[\ell+1]}.\nonumber
\end{split}
\end{equation}
If the following condition
 $$(\star)\ \  T_{G_{[\ell+1]}}\widetilde{\mathcal{G}}^{[\ell+1]}=\ span\big\{fiber_{r=r_{0}},\ fiber_{r=r_{1}},\, r_0\neq r_1\in [0,1]\big\}$$
holds, then the map $f_{4}$ is onto which implies that $G_{[\ell+1]}$ is generic.
In fact, we have $$fiber_{r=r_{0}}\subset d\Pi_{2}(T_{(G_{[\ell+1]},\mathbf{s},u)}(f_{r=r_{0}}^{-1}\widehat{F}_{[\ell]}^{-1}(0)))=: Im_{1} \subset Imf_{4} \subset T_{G_{[\ell+1]}}\widetilde{\mathcal{G}}^{[\ell+1]},$$ and $$fiber_{r=r_{1}}\subset d\Pi_{2}(T_{(G_{[\ell+1]},\mathbf{s},u)}(f_{r=r_{1}}^{-1}\widehat{F}_{[\ell]}^{-1}(0)))=: Im_{2} \subset Imf_{4} \subset T_{G_{[\ell+1]}}\widetilde{\mathcal{G}}^{[\ell+1]}.$$ The condition $(\star)$ implies that $Span\big\{Im_{1},\ Im_{2}\big\}=T_{G_{[\ell+1]}}\widetilde{\mathcal{G}}^{[\ell+1]}.$ Hence, $Imf_{4}=T_{G_{[\ell+1]}}\widetilde{\mathcal{G}}^{[\ell+1]}$.\\
\indent Now we prove $(\star)$. Notice that any tangent vector $X\in T_{G_{[\ell+1]}}\widetilde{\mathcal{G}}^{[\ell+1]}$ is a variational vector field along $G_{[\ell+1]}$ as a path in $\widetilde{\mathcal{G}}^{[\ell]}$.  Without loss of generality, we may assume $r_{0}=0$ and $r_{1}=1$ which is our case. Let $G_{[\ell+1]}(0,\cdots)=G_{[\ell]0}$ and $G_{[\ell+1]}(1,\cdots)=G_{[\ell]1}$.
We have $fiber_{r=0}=\big\{X:[0,1]\longrightarrow G_{[\ell+1]}^{\ast}(T\widetilde{\mathcal{G}}^{[\ell]}) \big|X(0)=0 \big\}$ and $fiber_{r=1}=\big\{X:[0,1]\longrightarrow G_{[\ell+1]}^{\ast}(T\widetilde{\mathcal{G}}^{[\ell]}) \big|X(1)=0 \big\}$. {Due to compactness of $[0,1]$, one can choose a homotopy  $G_{[\ell+1]}$ connecting $G_{[\ell]0}$ and $G_{[\ell]1}$} such that for any $X\in T_{G_{[\ell+1]}}\widetilde{\mathcal{G}}^{[\ell+1]}$, there exists $X_{1}\in fiber_{r=0}$ such that $X_{1}(1)=X(1)$. Let $X_{2}=X-X_{1}$, then $X_{2}\in fiber_{r=1}$, i.e., $X\in span\big\{fiber_{r=r_{0}},\ fiber_{r=r_{1}} \big\}$. The other direction is obvious.
\end{proof}

Let $H_{[\ell]}\in \mathcal{H}^{[\ell]}$ and $G_{[\ell]}\in \mathcal{G}^{[\ell]}$ such that $H_{[\ell]}$ is regular under $G_{[\ell]}$ and at the same time $G_{[\ell]}$ is generic w.r.t $H_{[\ell]}$ so that $F_{[\ell]G_{[\ell]}}^{-1}(0)$ is a submanifold of $\mathcal{P}^{1,2}_{p,q}$. For readers' convenience we collect here the various moduli spaces that we will use later.
$$\mathcal{M}_{p,q}^{f}:=\Big\{u\in \mathcal{P}^{1,2}_{p,q} \Big|\
     \dot{u}(t)=-\nabla_{g}f(u),\ \
    \lim\limits_{t\rightarrow -\infty}u(t)=p,\ \ \lim\limits_{t\rightarrow \infty}u(t)=q\Big\}/\R ,$$
where ``$/\R$'' means that modulo the time translation.
\begin{align}
  \mathcal{M}_{p,q}^{H_{[\ell]}^{\alpha,\beta}}:=\Big\{(s^{\underline{\ell}},\cdots,s^{\underline{1}},u)\Big|  & s^{\underline{\ell}},\cdots,s^{\underline{1}}\in[0,1],\ u(t)\in \mathcal{P}^{1,2}_{p,q}, \nonumber \\
    & \dot{u}(t)=-\nabla_{G_{[\ell]}}H_{[\ell]}(s^{\underline{\ell}},\cdots,s^{\underline{1}},t,u):= -\nabla_{G_{[\ell]}(s^{\underline{\ell}},\cdots,s^{\underline{1}},t)}H_{[\ell]}(s^{\underline{\ell}},\cdots,s^{\underline{1}},t,u),\nonumber \\
    &\lim\limits_{t\rightarrow -\infty}u(t)=p,\ \lim\limits_{t\rightarrow \infty}u(t)=q\Big\}=F_{[\ell]G_{[\ell]}}^{-1}(0), \nonumber
\end{align}
where $\ell\in\mathbb{Z}_{\geqslant 0}$ and ``$\ \ \dot\ \ $'' is the derivative with respect to time variable $t$.

Especially, when $\ell=0$,
\begin{align}
  \mathcal{M}_{p,q}^{H_{[0]}^{\alpha,\beta}}=\mathcal{M}_{p,q}^{H^{\alpha,\beta}}:=\Big\{u\in \mathcal{P}^{1,2}_{p,q}\Big|
    & \dot{u}(t)=-\nabla_{G}H(t,u):= -\nabla_{G(t)}H(t,u),\
    \lim\limits_{t\rightarrow -\infty}u(t)=p,\ \lim\limits_{t\rightarrow \infty}u(t)=q\Big\}=F_{G_{[0]}}^{-1}(0). \nonumber
\end{align}

\textbf{Compactness}

Having already established the smooth structure on the moduli space $\mathcal{M}_{p,q}^{H_{[\ell]}^{\alpha,\beta}}$ ($\ell\geq 0$) with expected dimension, in this section we will show that it can be compactified to a smooth manifold with corners. We give a detailed analysis about its boundaries and corners. The compactified moduli space of dimension one plays a key role and is responsible for the higher algebraic structures to be discussed in the next part.

We start with the case that $H:=H^{\alpha,\beta}$ is a smooth finite map connecting two Morse functions $f^{\alpha}$ and $f^{\beta}$. The main difference between $H^{\alpha,\beta}$ and $f$ is that $H^{\alpha,\beta}$ depends on $t$ whereas $f$ is not which results in the nonautonomous ODE.

The following convergence property can be established by standard techniques from the theory of elliptic differential operators.


\begin{lemma}\label{H C loc inf conv}(\cite{frauenfelder-nicholls}, p. 21-22)\label{partiallimit}
For any sequence $(u_{\nu})_{\nu}$ in $\mathcal{M}_{p,q}^{H^{\alpha,\beta}}$, there exists a convergent subsequence $u_{\nu_{j}}$ under $C_{loc}^{\infty}$-topology:
$$u_{\nu_{j}}\stackrel{C_{loc}^{\infty}}{\longrightarrow}w\in C^{\infty}(\mathbb{R},M),$$
{where $C_{loc}^{\infty}$ means that $\forall R\in \R_{>0}$,  $u_{\nu_{j}}|_{[-R,R]}\stackrel{C^{\infty}}{\longrightarrow}w|_{[-R,R]}\in C^{\infty}(\mathbb{R},M)$ and according to \cite{audin-damian} (p. 544), the $C^{\infty}$(resp. $C ^{k}$) topology is that of the uniform convergence of the function and all of its derivatives (resp. those of order at most k) on compact subsets.}
\end{lemma}

For the proof, one can use the Arzela-Ascoli theorem to get the $C_{loc}^{0}$ convergence, then improve the smoothness via bootstrap for the elliptic equations.
See (\cite{frauenfelder-nicholls}, p. 21-22) for details.

Now we are ready to describe the boundaries to be added to the compactified moduli space. We first recall the notion of the broken gradient flow line of $n$-fold for fixed $f$ (autonomous case) which we met implicitly in \S $2$.

\begin{definition}
Given two critical points $p,q\in Crf$, an $n (\geqslant 2)$-tuple  $y=\{u^{k}\}_{1\leqslant k\leqslant n}$ is called a broken gradient flow line (broken line) of $n$-fold 
from $p$ to $q$ for MS pair $(f,g)$ if there exist $c_{0}=p, c_{1}, \cdots, c_{n}=q\in Cr f$ and $\forall 1\leqslant k\leqslant n$, $u^{k}\in \mathcal{M}^{f}_{c_{k-1},c_{k}}$ and $u^{k}$ is not a constant gradient flow line.
Moreover, we say $(u_{\nu})_{\nu}$ Floer-Gromov (F-G) converges to the broken line $y=\{u^{k}\}_{1\leqslant k\leqslant n}$ if for any $1\leqslant k\leqslant n$, there exits a sequence of real numbers $(r_{\nu}^{k})_{\nu}$ such that $(r_{\nu}^{k})_{\star}u_{\nu} \stackrel{C_{loc}^{\infty}}{\longrightarrow}u^{k}$, where $(r_{\nu}^{k})_{\star}u_{\nu}(t):=u_{\nu}(t+r_{\nu}^{k})$.
\end{definition}

It can be generalized from the autonomous gradient flow equations to nonautonomous ones.
\begin{definition}
\indent An $n(\geqslant 2)$-tuple $y=\{u^{k}\}_{1\leqslant k\leqslant n}$ is called a broken gradient flow line (broken line) of $n$-fold 
from critical point $p\in Cr f^{\alpha}$ to critical point $q\in Cr f^{\beta}$ for pair $(H,G)$ if the $n$-tuple satisfies:
(1) $\forall k, u^{k}$ is not constant;
(2) There exist $k_{0}$ and $c_{0}=p, c_{1}, \cdots, c_{n}=q$ such that $k<k_{0}$, $c_{k}\in Cr f^{\alpha}$; $k\geqslant k_{0}$, $c_{k}\in Cr f^{\beta}$; And $u^{k}\in \mathcal{M}^{f^{\alpha}}_{c_{k-1},c_{k}}$ ($k<k_{0}$); $u^{k_0}\in \mathcal{M}^{H^{\alpha,\beta}}_{c_{k_{0}-1},c_{k_{0}}}$;  $u^{k}\in \mathcal{M}^{f^{\beta}}_{c_{k-1},c_{k}}$ ($k>k_{0}$).

We say $(u_{\nu})_{\nu}$ Floer-Gromov (F-G) converge to the broken line $y=\{u^{k}\}_{1\leqslant k\leqslant n}$ if for any $1\leqslant k\leqslant n$, there exits a $k_{0}\in \{k\in\mathbb{N}\big| 1\leqslant k\leqslant n\}$ and a sequence of $(r_{\nu}^{k})_{\nu}$ with $r_{\nu}^{k_{0}}=0$ such that $(r_{\nu}^{k})_{\star}u_{\nu}\stackrel{C_{loc}^{\infty}}{\longrightarrow}u^{k}$, where $(r_{\nu}^{k})_{\star}u_{\nu}(t):=u_{\nu}(t+r_{\nu}^{k})$.
\end{definition}

Based on these notions and "partial'' limit behavior of Lemma \ref{partiallimit}, there is the following complete picture about the limits.

{\begin{theorem}(\cite{frauenfelder-nicholls}, p. 21-22, 25-34)
Let $f$ be a Morse function on $M$, and two critical points $p, q\in Cr f$. For a sequence of negative gradient flow lines $(u_{\nu})_{\nu}$ in $\mathcal{M}^{f}_{p,q}$, there is a subsequence $(u_{\nu_{j}})_{\nu_{j}}$ and a broken gradient flow line $y=\{u^{k}\}_{1\leqslant k\leqslant n}$ from $p$ to $q$ such that $u_{\nu_{j}}\stackrel{F-G}{\longrightarrow}y$ {if $u_{\nu_{j}}$ does not converge in $\mathcal{M}^{f}_{p,q}$}.
\end{theorem}

\begin{theorem}(\cite{schwarz}, p. 64-67 and \cite{frauenfelder-nicholls}, p. 25) \label{H case FG convergence thm}
For a sequence of negative gradient flow lines $(u_{\nu})_{\nu}$ from critical point $p\in Cr f^\alpha$ to critical point $q\in Cr f^\beta$ in $\mathcal{M}^{H^{\alpha, \beta}}_{p,q}$, there is a subsequence $(u_{\nu_{j}})_{\nu_{j}}$ and a broken gradient flow line $y=\{u^{k}\}_{1\leqslant k\leqslant n}$ from $p$ to $q$ such that $u_{\nu_{j}}\stackrel{F-G}{\longrightarrow}y$ {if $u_{\nu_{j}}$ does not converge in $\mathcal{M}^{H^{\alpha, \beta}}_{p,q}$}.
\end{theorem}

Now we turn to our $H_{[\ell]}$ with parameters $s^{\underline{\ell}},\cdots, s^{\underline{1}}$.
%
%
%
Similarly, we have
\begin{lemma} \label{H[1] C loc inf conv}
{For any sequence $\{(\mathbf{s},u)_{\nu}\}_{\nu}=\{(\mathbf{s_{\nu}},u_{\nu})\}_{\nu}$ in $\mathcal{M}_{p,q}^{H_{[\ell]}^{\alpha,\beta}}$, there exists a convergent subsequence $\{(\mathbf{s},u)_{\nu_{j}}\}$ under $C_{loc}^{\infty}$:
$$(\mathbf{s},u)_{\nu_{j}}\stackrel{C_{loc}^{\infty}}{\longrightarrow}(\mathbf{s_{0}},w)\in [0,1]^{\ell}\times C^{\infty}(\mathbb{R},M)$$}
\end{lemma}
\begin{proof}
\indent We use the Arzela-Ascoli theorem to prove the $C_{loc}^{0}$ convergence, and improve the smoothness via bootstrap for the elliptic equation. The main difference from the Lemma \ref{H C loc inf conv} is the smooth dependence of $H_{[\ell]}^{\alpha,\beta}$ and $G_{[\ell]}^{\alpha,\beta}$ on the parameters which cause no further difficulty for the convergence due to the compactness of the parameter space according to the fundamental theorem for the system of ordinary differential equations. Since the proof is similar to the Lemma \ref{H C loc inf conv}, we omit it.
\end{proof}

Now we define broken gradient flow line and F-G convergence in terms of $H_{[\ell]}$ case.
\begin{definition}
{ An $n (\geqslant {2})$-tuple $\{u^{1},\cdots,(\mathbf{s_{0}},u^{k_{0}}),\cdots,u^{n}\}{_{1\leqslant k\leqslant n}}$ is called a broken gradient flow line (abbr. broken line) of $n$-fold 
from the critical point $p$ to the critical point $q$ for the pair $(H_{[\ell]},G_{[\ell]})$ if the followings hold:\\
(1) $\forall k$, $u^{k}$ is not a constant gradient flow line;\\
(2) There exist a $1\leqslant k_{0}\leqslant n$, $\mathbf{s_{0}}\in[0,1]^{\ell}$ and critical points $c_{0}=p, c_{1}, \cdots, c_{n}=q$ such that  $c_{k}\in Cr f^{\alpha}$($k<k_{0}$); $c_{k}\in Cr f^{\beta}$($k\geqslant k_{0}$), and $u^{k}\in \mathcal{M}^{f^{\alpha}}_{c_{k-1},c_{k}}$($k<k_{0}$); $(\mathbf{s_{0}},u^{k_{0}})\in \mathcal{M}^{H_{[\ell]}^{\alpha,\beta}}_{c_{k_{0}-1},c_{k_{0}}}$; $u^{k}\in \mathcal{M}^{f^{\beta}}_{c_{k-1},c_{k}}$($k>k_{0}$).}

{ We say the sequence $(\mathbf{s_{\nu}},u_{\nu})_{\nu}$ Floer-Gromov (F-G) converges to the broken line
$y=\{u^{1},\cdots,(\mathbf{s_{0}},u^{k_{0}}),\cdots,u^{n}\}$
 if
there exits a $k_{0}\in \{k\in\mathbb{N}\big| 1\leqslant k\leqslant n\}$ and a sequence of $(r_{\nu}^{k})_{\nu}$ and $r_{\nu}^{k_{0}}=0$ such that $(r_{\nu}^{k})_{\star}u_{\nu}\stackrel{C_{loc}^{\infty}}{\longrightarrow}u^{k}$, for any $1\leqslant k\leqslant n$.}
\end{definition}

\begin{theorem}\label{FGellcase}
For a sequence of negative gradient flow lines $(\mathbf{s}_{\nu},u_{\nu})_{\nu}$ in $\mathcal{M}^{H_{[\ell]}^{\alpha, \beta}}_{p,q}$,  there is a subsequence $(\mathbf{s}_{\nu_{j}},u_{\nu_{j}})_{\nu_{j}}$ and a broken gradient flow line $y=\{u^{1},\cdots,u^{k_{0}-1},(\mathbf{s_{0}},u^{k_{0}}),u^{k_{0}+1},\cdots,u^{n}\}$ from critical point $p$ to critical point $q$ such that $(\mathbf{s}_{\nu_{j}},u_{\nu_{j}})\stackrel{F-G}{\longrightarrow}y$ {if $(\mathbf{s}_{\nu_{j}},u_{\nu_{j}})$ does not converge in $\mathcal{M}^{H_{[\ell]}}_{p,q}$}.
\end{theorem}
\begin{proof}
{We follow \cite{schwarz} (p. 64-67) and \cite{frauenfelder-nicholls} (p. 25-34),  and the proof is similar. 
By Lemma \ref{H[1] C loc inf conv}, there exists a subsequence $\{(\mathbf{s}_{\nu_{j}},u_{\nu_{j}})\}_{j}$ and $(\mathbf{s_{0}},w)\in [0,1]^{\ell}\times C^{\infty}(\mathbb{R},M)$ such that
$$(\mathbf{s}_{\nu_{j}},u_{\nu_{j}})\stackrel{C_{loc}^{\infty}}{\longrightarrow}(\mathbf{s_{0}},w),$$
which means that $\dot{u}_{\nu_{j}}(t)=-\nabla_{G_{[\ell](\mathbf{s}_{\nu_{j}},t)}}H_{[\ell]}(\mathbf{s}_{\nu_{j}},t,u_{\nu_{j}}(t))$ holds and for $\forall R>0$, we have $(\mathbf{s}_{\nu_{j}},u_{\nu_{j}}(t))\big|_{[-R,R]}\stackrel{C^{\infty}}{\longrightarrow}(\mathbf{s_{0}},w(t))\big|_{[-R,R]}$.
Besides, $(\mathbf{s_{0}},w)$ satisfies $\dot{w}(t)=-\nabla_{G_{[\ell]}(\mathbf{s_{0}},t)}H_{[\ell]}(\mathbf{s_{0}},t,w(t))$.
However, $(\mathbf{s_{0}},w(t))$ does not belong to $\mathcal{M}^{H_{[\ell]}^{\alpha, \beta}}_{p,q}$ but $\mathcal{M}^{H_{[\ell]}^{\alpha, \beta}}_{c_{1},c_{2}}$, where $c_{1}\in Crf^{\alpha}$, $c_{2}\in Crf^{\beta}$ since $H_{[\ell]}\equiv f^{\alpha}(t<-T) $ and $H_{[\ell]}\equiv f^{\beta} (t>T)$.
Additionally, $indc_{1}\leqslant indp$ and $indc_{2}\geqslant indq$. Let {$(\mathbf{s_{0}},u^{k_{0}}(t))=(\mathbf{s_{0}},w(t))$}. If $c_{2}\neq q$ and $c_{2}\in Crf^{\beta}$, we can choose any regular value $a\in [f^{\beta}(q),f^{\beta}(c_{2})]$ of $f^{\beta}$.
There exists a sequence $\big\{r_{j}\big\}_{j}$ and  $\widetilde{u_{j}}:=r_{j\star}u_{\nu_{j}}$ such that $f^{\beta}(\widetilde{u_{j}}(0))=a$ holds for any $j$.
By Lemma \ref{H[1] C loc inf conv}, there exists $\widetilde{w}$ such that $\widetilde{u_{j_{\ell}}}\stackrel{C_{loc}^{\infty}}{\longrightarrow}\widetilde{w}$, and
{$ind c_{2}\geqslant ind \widetilde{w}(-\infty)$, $ind\widetilde{w}(+\infty)\geqslant indq$}. {The case that $c_{1}\neq p \in Crf^{\alpha}$ can be considered in the same way.}

Now it is clear that if the index inequality holds strictly, we can always choose a regular value inbetween until they coincide, and the whole process stops in finite steps since $\sharp (Crf^{\alpha}\cup Crf^{\beta})<+\infty$.

To conclude, there exist $u^{1},\cdots,u^{k_{0}-1},(s_{0},u^{k_{0}}),u^{k_{0}+1},\cdots,u^{n}$ and $\lim\limits_{t\rightarrow +\infty}u^{i}(t)=\lim\limits_{t\rightarrow -\infty}u^{i+1}(t),\, (i=1,\cdots,n-1$) such that $$u_{\nu_{j}}\stackrel{F-G}{\longrightarrow}y.$$
}
\end{proof}

{Define the compactified moduli space $\overline{\mathcal{M}}_{p,q}^{H_{[\ell]}}$ to be the union $\mathcal{M}_{p,q}^{H_{[\ell]}}\  \bigcup$ $\big\{y \big| u_{\nu_{j}}\stackrel{F-G}{\longrightarrow}y, \ \forall u_{\nu_{j}}\in \mathcal{M}_{p,q}^{H_{[\ell]}} \big\}$.}

Now one can describe the boundary structures of the compactified moduli space 
which turns out to be of the following two types:

{
$(i)$ the natural boundary}\begin{equation}
\begin{aligned}
 generic:\, \, &(\mathbf{s},u^{k_{0}}),\,\,\, \\
 \, &where\, \mathbf{s}\in \partial[0,1]^{\ell}\, {s.t.\, G_{[\ell]}(\mathbf{s},\cdot)\, is \, generic}\ and\ (\mathbf{s},u^{k_{0}})\in \mathcal{M}_{p,q}^{H_{[\ell]}}; \\
 non-genric:\, \, &u^{1},\cdots,u^{k_{0}-1},(\mathbf{s},u^{k_{0}}),u^{k_{0}+1},\cdots,u^{n},\,\,\, \\
 \, &where\, \mathbf{s}\in \partial[0,1]^{\ell}\, s.t.\, G_{[\ell]}(\mathbf{s},\cdot)\,is \, not\ generic,\,\, {\exists\ some\ }n\in\mathbb{N}_{>1},\ \\
 \, &{\ c_{0},\cdots,c_{k_{0}-1}\in Crf^{\alpha},\  c_{k_{0}},\cdots,c_{n}\in Crf^{\beta},\  u^{i}\in \mathcal{M}^{f^{\alpha}}_{c_{i-1},c_{i}}(i=1,\cdots,k_{0}-1),} \\
 \, &{(\mathbf{s},u^{k_{0}})\in \mathcal{M}_{p,q}^{H_{[\ell]}},\ u^{i}\in \mathcal{M}^{f^{\beta}}_{c_{i-1},c_{i}}(i=\cdots,k_{0}+1,\cdots,n);} \nonumber
\end{aligned}
\end{equation}

{
$(ii)$ the boundary coming from the {interior} non-generic $G_{[\ell]}(\mathbf{s_{0}},\cdot)$ with $\mathbf{s_{0}}\in (0,1)^{\ell}$}
$$u^{1},\cdots,u^{k_{0}-1},(\mathbf{s_{0}},u^{k_{0}}),u^{k_{0}+1},\cdots,u^{n}, \,\,\, {\exists\ some\ }n\in\mathbb{N}_{>1},$$
where {$c_{0},\cdots,c_{k_{0}-1}\in Crf^{\alpha},\  c_{k_{0}},\cdots,c_{n}\in Crf^{\beta},\  u^{i}\in \mathcal{M}^{f^{\alpha}}_{c_{i-1},c_{i}}(i=1,\cdots,k_{0}-1),\ (\mathbf{s},u^{k_{0}})\in \mathcal{M}_{p,q}^{H_{[\ell]}},\ u^{i}\in \mathcal{M}^{f^{\beta}}_{c_{i-1},c_{i}}(i=k_{0}+1,\cdots,n).$}

Here are two examples to get a flavor about the boundaries of the compactified moduli space.

$1.$ the boundary of the $1$-dimensional compactified moduli space.
Here Figure \ref{sketchboundary 1 dimensionalmodulispace} is a schematic. In the following part, the boundary of the $1$-dimensional compactified moduli space will be discussed in details.
\begin{figure}
  \centering
  \def\svgscale{0.35}
\begingroup%
  \makeatletter%
  \providecommand\color[2][]{%
    \errmessage{(Inkscape) Color is used for the text in Inkscape, but the package 'color.sty' is not loaded}%
    \renewcommand\color[2][]{}%
  }%
  \providecommand\transparent[1]{%
    \errmessage{(Inkscape) Transparency is used (non-zero) for the text in Inkscape, but the package 'transparent.sty' is not loaded}%
    \renewcommand\transparent[1]{}%
  }%
  \providecommand\rotatebox[2]{#2}%
  \newcommand*\fsize{\dimexpr\f@size pt\relax}%
  \newcommand*\lineheight[1]{\fontsize{\fsize}{#1\fsize}\selectfont}%
  \ifx\svgwidth\undefined%
    \setlength{\unitlength}{265.95788281bp}%
    \ifx\svgscale\undefined%
      \relax%
    \else%
      \setlength{\unitlength}{\unitlength * \real{\svgscale}}%
    \fi%
  \else%
    \setlength{\unitlength}{\svgwidth}%
  \fi%
  \global\let\svgwidth\undefined%
  \global\let\svgscale\undefined%
  \makeatother%
  \begin{picture}(1,0.94873674)%
    \lineheight{1}%
    \setlength\tabcolsep{0pt}%
    \put(0,0){\includegraphics[width=\unitlength,page=1]{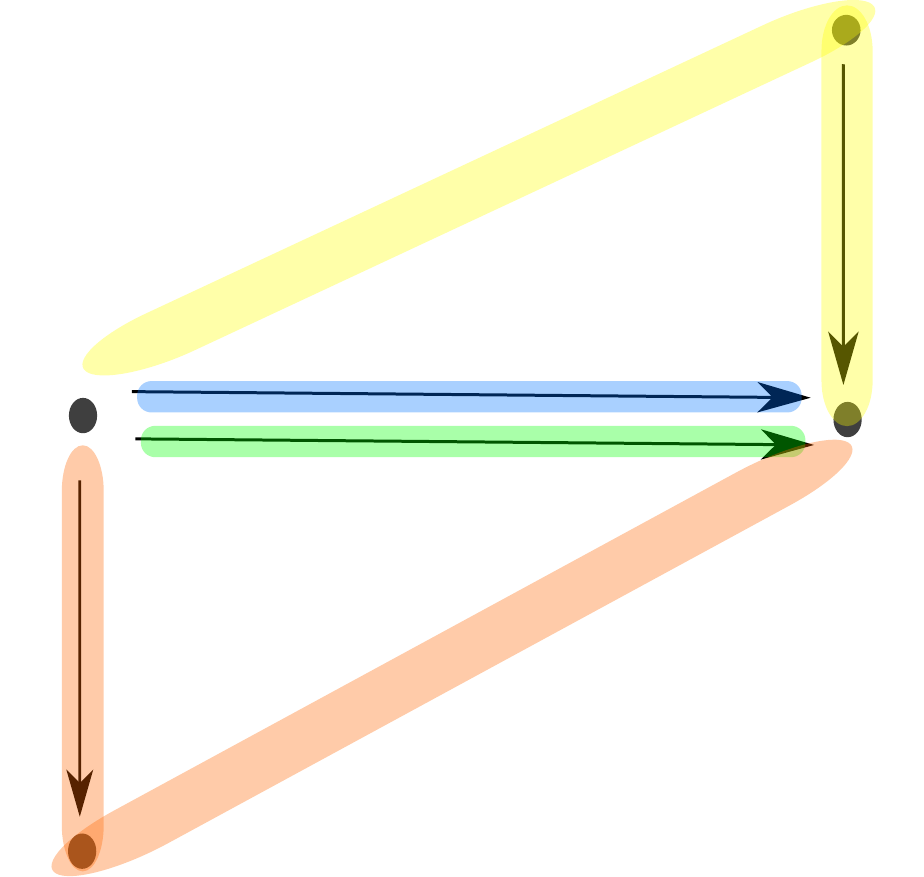}}%
    \put(-0.00691224,0.48505604){\color[rgb]{0,0,0}\makebox(0,0)[lt]{\lineheight{1.25}\smash{\begin{tabular}[t]{l}p\end{tabular}}}}%
    \put(0.95860089,0.46433156){\color[rgb]{0,0,0}\makebox(0,0)[lt]{\lineheight{1.25}\smash{\begin{tabular}[t]{l}q\end{tabular}}}}%
    \put(0,0){\includegraphics[width=\unitlength,page=2]{dim1modulispaceboundary.pdf}}%
  \end{picture}%
\endgroup%

  \caption{different types of the boundary of $1$-dimensional compactified moduli space with boundary $\partial \overline{\mathcal{M}}_{p,q}^{H_{[1]}^{\alpha,\beta}}$ of dimension zero: broken lines with respect to critical points of $f^\alpha$ and $f^\beta$ respectively}
  \label{sketchboundary 1 dimensionalmodulispace}
\end{figure}
Since the compactified moduli space is of dimension one, each component is a segment (Figure
\ref{sketchboundary 1 dimensionalmodulispace}) or a circle.
\begin{figure}
  \centering
  \includegraphics[scale=0.15]{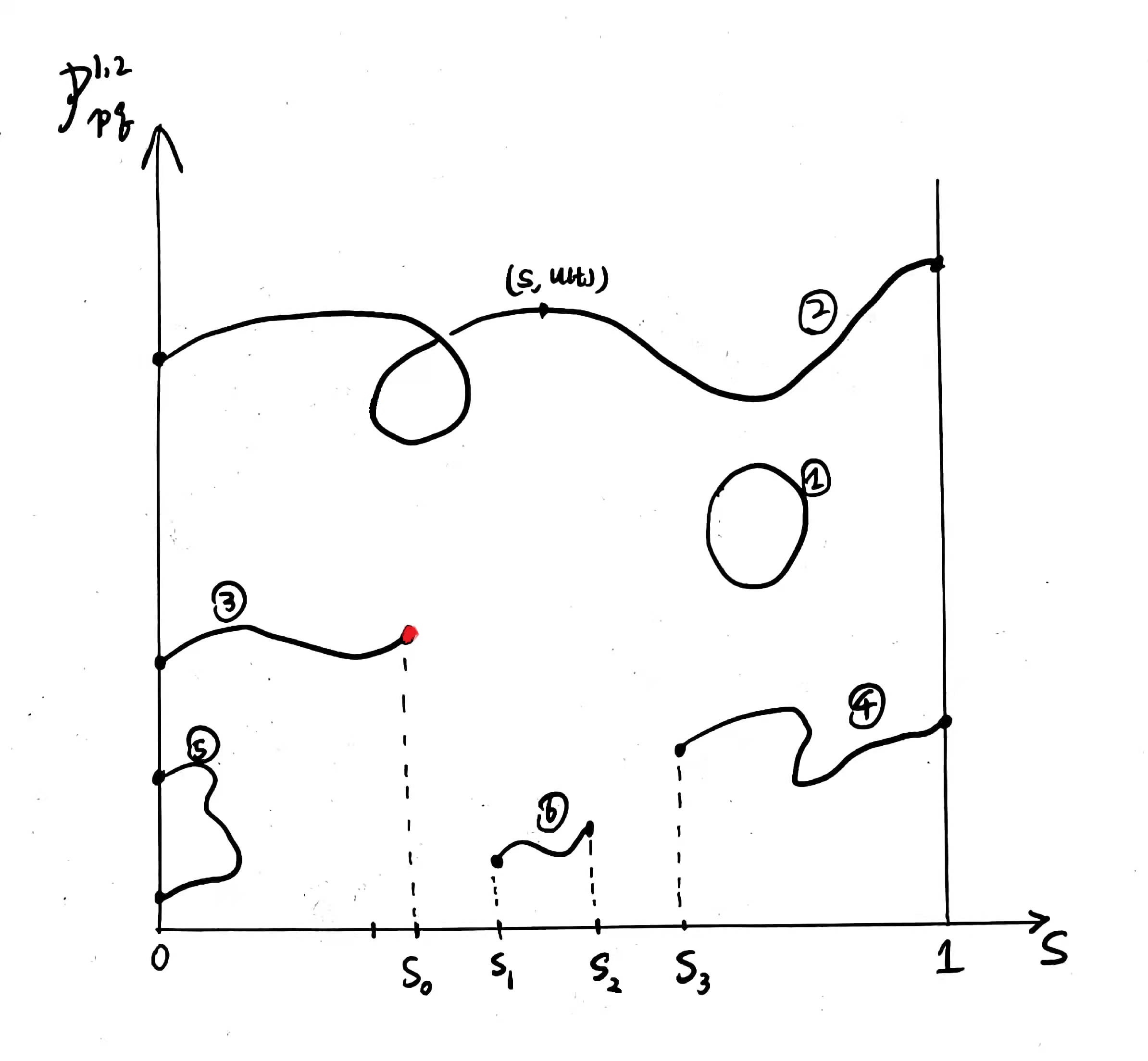}
  \caption{$1$-dimensional compactified moduli space with boundary: case with one parameter $s$ (\cite{schwarz})}
  \label{onedimmfldwithbdryofspara}
\end{figure}

Figure  \ref{onedimmfldwithbdryofspara}  is a picture we borrow from (\cite{schwarz}, p. 146). This is an example for the $1$-dimensional compactified moduli space {\it with one parameter} which turns out to be a one dimensional manifold with boundary, i.e. some segments and some circles in $[0,1]\times \mathcal{P}^{1,2}_{p,q}$. The horizontal axis represents the parameter $s\in [0,1]$, while the vertical direction represents the ambient space $\mathcal{P}^{1,2}_{p,q}$. For the circle component$\textcircled{\small{1}}$ in the figure, its boundary is empty since it is already compact; For $\textcircled{\small{2}}$, the two boundaries lie where  $s=0,1$, which are the natural boundaries; so does $\textcircled{\small{5}}$, the difference is that the boundaries for $\textcircled{\small{5}}$ are both at $s=0$. For some exceptional moments  $s=s_{0},s_{1},s_{2}$ and $s_{3}$ as shown in the figure, non-generic situations occur. This is unavoidable when we have one parameter.



\indent $2.$ the boundary of the $2$-dimensional compactified moduli space with {\it two parameters} from which one will get an intuition about the corners.
Fig \ref{a2b2inch3} and Fig \ref{c2inch3} {show the possible types of the boundary of the two-dimensional compactified moduli space with two parameters.} {More detailed analysis will be given in \S \ref{further direction sec5}. 
\begin{figure}
  \centering
  \includegraphics[scale=0.5]{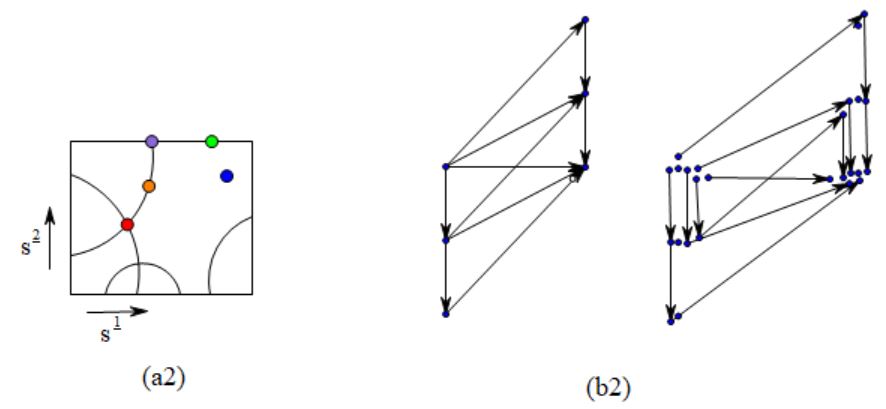}
  \caption{{Two-dimensional parameter space and different boundary types}}
  \label{a2b2inch3}
\end{figure}
\begin{figure}
  \centering
  \includegraphics[scale=0.6]{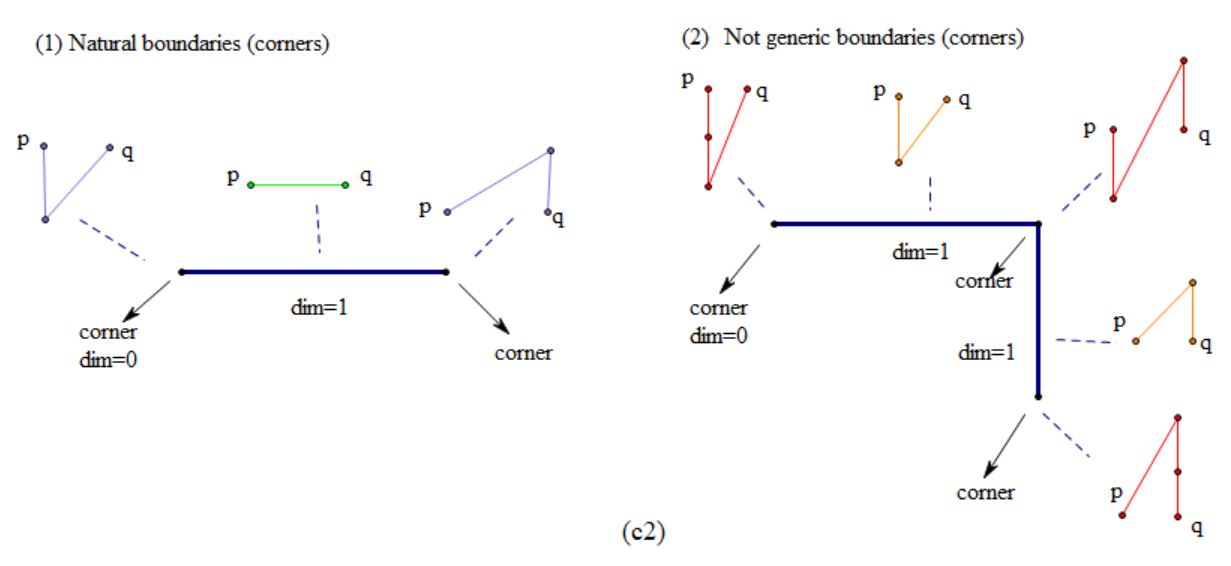}
  \caption{(c$2$) {Possible types of the boundary with $indp-indq=0$ and the dimension of $\partial \overline{\mathcal{M}}_{p,q}^{H_{[2]}^{\alpha,\beta}}$ is one}}
  \label{c2inch3}
\end{figure}

\textbf{The boundaries of the $1$-dimensional compactified moduli space}

We explore in details the boundaries of one dimensional compactified moduli space geometrically which will be responsible for the algebraic structures to be described in the following section.
As usual, $M$ is a closed manifold with $f^{\alpha}$ and $f^{\beta}$ two Morse functions on it.
The corresponding results for $H$ and $H_{[1]}$ have already appeared in, e.g., \cite{frauenfelder-nicholls} and \cite{hutchings}. 
{Now, we discuss the more general case with $H_{[\ell]}$ ($\ell\geqslant 0$).}

Let $(f^{\alpha},g^{\alpha})$ and $(f^{\beta},g^{\beta})$ be two MS pairs. By {regularity and smoothness} mentioned in \S \ref{Moduli Space and its Compactification sect3}, there exists $(H_{[\ell]},G_{[\ell]})$ satisfying the following conditions. The above $H_{[\ell]}$ is finite with respect to $f^{\alpha}$ and $f^{\beta}$ and regular under $G_{[\ell]}$. At the same time, $G_{[\ell]}$ is finite with respect to $g^{\alpha}$ and $g^{\beta}$ and generic under $H_{[\ell]}$.
For two critical points $p\in Cr f^\alpha$ and $q\in Cr f^\beta$ such that $indp-indq=-\ell+1$, we have
$$dim(\mathcal{M}_{p,q}^{H_{[\ell]}})=indp-indq+\ell=1.$$
 Let critical points $c'\in Crf^{\beta}$ and $c\in Crf^{\alpha}$ satisfy $indp-indc'=indc-indq=-\ell$.
To figure out the boundaries of compactified $\overline{\mathcal{M}}_{p,q}^{H_{[\ell]}}$, we introduce the following moduli space
\begin{align}
  \mathcal{M}_{p,c'}^{H_{[\ell]}(S_{0})}=\Big\{(\mathbf{s_{0}},u(t)) \Big|  & \dot{u}(t)=-\nabla_{G_{[\ell]}(\mathbf{s_{0}},\cdot)}H_{[\ell]}(\mathbf{s_{0}},t,u(t)), \ \mathbf{s_{0}}\in S_{0}\subset [0,1]^{\ell},\ ind p-ind c'=-\ell;\nonumber \\
   & \lim\limits_{t\rightarrow -\infty}u(t)=p,\  \lim\limits_{t\rightarrow +\infty}u(t)=c' \Big\},\nonumber
\end{align}
where $S_{0}$ is a codimension one set
of the space $[0,1]^{\ell}$ of parameters $s^{\underline{\ell}},\cdots,s^{\underline{1}}$ and $\mathbf{s}=(s^{\underline{\ell}},\cdots,s^{\underline{1}})$.
As before, we define $\widehat{S}(p,c')$ to be the set of all $\widehat{\mathbf{s}}$ such that $G_{[\ell]}(\widehat{\mathbf{s}},\cdot)$ is not generic (in the sense of Proposition \ref{usefulprop}) with respect to (fixed) $p\in Crf^{\alpha}$, $c'\in Crf^{\beta}$ and $H_{[\ell]}(\widehat{\mathbf{s}},\cdot,\cdot)$. Similarly, we can define $\widehat{S}(c,q)$. Then $S_{0}$ is either $\widehat{S}(p,c')$ or $\widehat{S}(c,q)$. Moreover, we have
$$dim \mathcal{M}_{p,c'}^{H_{[\ell]}(\widehat{S}(p,c'),\cdot,\cdot)}= ind p- ind c' +1+(\ell-1){=-\ell+1+(\ell-1)=0},$$
where $ind p- ind c'=-\ell$. Here `$+1$' results from non genericity, and `$(\ell-1)$' comes from the freedom of the parameters, 
that is, $\widehat{\mathbf{s_{0}}}$ can be taken in $S_{0}$ which is a codimension one set in $[0,1]^{\ell}$.
Note that $\widehat{S}(p,c')$ could be the union of codimension one manifolds which provide part of the boundaries.
The specific requirements will be introduced in \S \ref{Morse weak n pre-category}.

The boundaries of $\overline{\mathcal{M}}_{p,q}^{H_{[\ell]}}$ contain two parts, one of which is natural boundaries coming from the boundaries of parameter $s^{\underline{\ell}}\in[0,1]$ (Be aware that only $s^{\underline{\ell}}$ plays a role since other parameters' direction are degenerate when $s^{\underline{\ell}}=0,1$), and another is coming from the non-generic set.
Then the broken line has the following four possible types as shown schematically in Figure \ref{Ptwo}:

(i) $\{w(t)\}$,
where $w(t)\in \mathcal{M}_{p,q}^{H_{[\ell]}(0,\cdots,\cdot,\cdot)}$;

(ii) $\{w'(t)\}$,
where $w'(t)\in \mathcal{M}_{p,q}^{H_{[\ell]}(1,\cdots,\cdot,\cdot)}$;

{(iii) $\{(\mathbf{\widehat{s}_{0}},u^{1=k_{0}}(t)),u^{2}(t)\}$,
where $(\mathbf{\widehat{s}_{0}},u^{1=k_{0}}(t))\in \mathcal{M}_{p,c'}^{H_{[\ell]}},\ u^{2}(t)\in \mathcal{M}_{c',q}^{f^{\beta}}$; $\mathbf{\widehat{s}_{0}}\in \widehat{S}(p,c')$, $c'\in Crf^{\beta}$ and $ind c'= ind p+\ell$.}

{(iv) $\{u^{1}(t),(\mathbf{\widehat{s}_{1}},u^{2=k_{0}}(t))\}$,
where $u^{1}(t)\in \mathcal{M}_{p,c}^{f^{\alpha}},\ (\mathbf{\widehat{s}_{1}},u^{2=k_{0}}(t))\in \mathcal{M}_{c,q}^{H_{[\ell]}}$; $\mathbf{\widehat{s}_{1}}\in \widehat{S}(c,q)$ $c\in Crf^{\alpha}$ and $ind c= ind q-\ell$.}
So the boundary is
\begin{equation}\label{E ell}
  \partial \overline{\mathcal{M}}_{p,q}^{H_{[\ell]}}=
 \mathcal{M}_{p,q}^{H_{[\ell]}(0,\cdots,\cdot,\cdot)}\cup \mathcal{M}_{p,q}^{H_{[\ell]}(1,\cdots,\cdot,\cdot)}
\cup\bigcup\limits_{c'\in Crf^{\beta} \atop ind c'= ind p+\ell}\mathcal{M}_{p,c'}^{H_{[\ell]}(\widehat{S}(p,c'),\cdot,\cdot)}\times \mathcal{M}_{c',q}^{f^{\beta}}
\cup\bigcup\limits_{c\in Crf^{\alpha} \atop  ind c= ind q-\ell }\mathcal{M}_{p,c}^{f^{\alpha}}\times \mathcal{M}_{c,q}^{H_{[\ell]}(\widehat{S}(c,q),\cdot,\cdot)}
\end{equation}
where $\widehat{S}(p,c'),\,\widehat{S}(c,q) \subset[0,1]^{\ell}$, and both are the finite union of some one dimensional manifolds.
\begin{figure}
\centering
\def\svgscale{0.4}
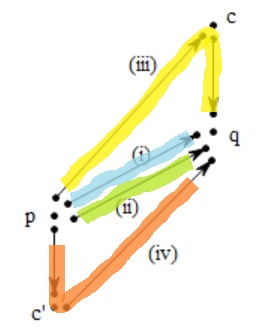
\caption{Four types of the boundary in $H_{[\ell]}$ case}
\label{Ptwo}
\end{figure}

Then we can define a linear map $\phi_{[\ell]}$ which finally turns out to be a higher homotopy. In \S \ref{weak infty functor},
we will discuss these higher homopoties $\phi_{[\ell]}$ and codimension sets from the non-generic part of the boundaries explicitly.

\begin{section}
  {Weak $\infty$-category}
  \label{Weak infty category sect4}
\end{section}

As usual, let $M$ be an $n$-dimensional smooth oriented closed manifold.
In this section, we firstly define an $n$-globular set in terms of Morse-Samle pairs and {their} higher homotopies with two operations: composite and identity map satisfying certain conditions which finally turns out to be a weak $n$-category. {With $n$ going to $\infty$, we can further define an $\infty$-globular set and turn it into a weak $\infty$-category $\mathcal{A}$.}
Here "weak''  means that the associativity holds up to homotopy. And we require that the composition of morphisms remain smooth. Thus, only the morphisms that can be connected smoothly are taken into consideration.
While we have the strict $n$-category {$\mathbf{MSWCx}$} to be introduced in \S \ref{strict n cate CHCx}, we extend it trivially to a strict $\infty$-category $\mathcal{B}$. Then we can construct a weak $\infty$-functor via the boundary of the compactification of the moduli space with varies parameters of dimension one. For the notions about  $\infty$-categories, our basic references are Hohloch (\cite{hohloch}) and Leinster (\cite{leinster}), and the construction of the weak functor is inspired by Frauenfelder-Nicholls (\cite{frauenfelder-nicholls}).

\begin{subsection}
  {Strict m-category}
  \label{strictcat}
\end{subsection}

\begin{definition}(\cite{hohloch}, p. 13) 
Given $m\in \mathbb{N}$, an $m$-globular set $Y$ contains a collection of sets $\{Y(\ell)\}_{\ell=0}^{m}$ together with source and target functions $\mathfrak{s},\mathfrak{t}: Y(\ell)\longrightarrow Y(\ell-1)$ satisfying $\mathfrak{s}\circ \mathfrak{s}= \mathfrak{s}\circ \mathfrak{t}$ and $\mathfrak{t}\circ \mathfrak{s}=\mathfrak{t}\circ \mathfrak{t}$ for every $0< \ell\leqslant m$. Elements in $Y(\ell)$ are called $\ell$-cells.
\end{definition}
\begin{figure}
\centering
\def\svgscale{0.5}
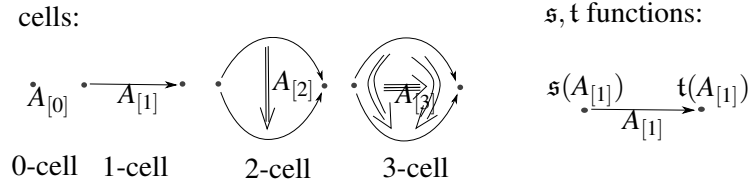
\caption{Cells, Sources and Targets}
\end{figure}

{\begin{definition}
A map or a morphism $f$ between two $m$-globular sets $(Y,\mathfrak{s},\mathfrak{t})$ and $(Y',\mathfrak{s}',\mathfrak{t}')$  {consists of} $f(A)\in Y'(\ell),\, \forall A\in Y(\ell),\, \ell=1,2,\cdots,m$ {such that $f\mathfrak{s}=\mathfrak{s}'f,\, f\mathfrak{t}=\mathfrak{t}'f$}.
\end{definition} }

For $\ell\in \Z_{>0}, \, 0\leqslant p<\ell$, we define $$Y(\ell)\times_{Y(p)} Y(\ell):=\Big\{ (A_{[\ell]2}, A_{[\ell]1})\in Y(\ell)\times Y(\ell)\Big|\mathfrak{s}^{\ell-p}(A_{[\ell]2})=\mathfrak{t}^{\ell-p}(A_{[\ell]1}) \Big\}$$ and we say that two $\ell$-cells $A_{[\ell]2}$ and $A_{[\ell]1}$ can be glued together along a common $p$-cell $A_{[p]}$:
$\mathfrak{s}^{\ell-p}(A_{[\ell]2})=\mathfrak{t}^{\ell-p}(A_{[\ell]1})=A_{[p]}.$ 

\begin{definition}\label{strict m cate def}(\cite{hohloch}, p. 15) 
Let $m\in \mathbb{N}$. A strict $m$-category $\mathcal{Y}$ is an $m$-globular set $Y$ together with

$\bullet$ A function $\circ_{p}: Y(\ell)\times_{Y(p)} Y(\ell)\longrightarrow Y(\ell)$ for each $0\leqslant p < \ell \leqslant m$, and $\circ_{p}(A',A):= A'\circ_{p}A$ is called a composite of $A'$ and $A$ along some $p$-cell;

$\bullet$ A function $\mathbf{1}:Y(\ell)\longrightarrow Y(\ell+1)$ for each $0\leqslant \ell < m$, and $\mathbf{1}(A):=\mathbf{1}_{A}$ is called identity on $A$;

satisfying the following axioms:

$(a)\mathbf{(sources\ and\ targets\ of\ composites)}$ If $0\leqslant p < \ell \leqslant m$ and $A',A \in Y(\ell)\times_{Y(p)} Y(\ell)$, then
\begin{equation}
\begin{array}{ccc}
  for\,\, p= \ell-1 & \mathfrak{s}( A'\circ_{p}A)=\mathfrak{s}(A)               & \mathfrak{t}( A'\circ_{p}A)=\mathfrak{t}(A') \\
  for\,\, p< \ell-1 & \mathfrak{s}( A'\circ_{p}A)=\mathfrak{s}(A')\circ_{p}\mathfrak{s}(A) & \mathfrak{t}( A'\circ_{p}A)=\mathfrak{t}(A')\circ_{p}\mathfrak{t}(A).\nonumber
\end{array}
\end{equation}

$(b)\mathbf{(sources\ and\ targets\ of\ identities)}$ For $0\leqslant \ell < m$ and $A\in Y(\ell)$, we have
$$\mathfrak{s}(\mathbf{1}_{A})= A =\mathfrak{t}(\mathbf{1}_{A}).$$

$(c)\mathbf{(associativity)}$ If $0\leqslant p < \ell \leqslant m$, $A,B,C\in Y(\ell)$ and $(C,B),(B,A)\in Y(\ell)\times_{Y(p)} Y(\ell)$, then
$$(C\circ_{p}B)\circ_{p}A=C\circ_{p}(B\circ_{p}A).$$

$(d)\mathbf{(identities)}$ For $0\leqslant p < \ell \leqslant m$ and $A\in Y(\ell)$, we have
 $$\mathbf{1}^{\ell-p}(\mathfrak{t}^{\ell-p}(A))\circ_{p}A= A =A\circ_{p}\mathbf{1}^{\ell-p}(\mathfrak{s}^{\ell-p}(A)).$$

$(e)\mathbf{(binary\ interchange)}$ For $0\leqslant q< p < \ell \leqslant m$, $A,B,C,D\in Y(\ell)$ and $(D,C),(B,A)\in Y(\ell)\times_{Y(p)} Y(\ell)$,$(D,B),(C,A)\in Y(\ell)\times_{Y(q)} Y(\ell)$, we have
$$(D\circ_{p}C)\circ_{q}(B\circ_{p}A)=(D\circ_{q}B)\circ_{p}(C\circ_{q}A).$$

$(f)\mathbf{(nullary\ interchange)}$ For $0\leqslant p < \ell \leqslant m$, and $(B,A)\in Y(\ell)\times_{Y(p)} Y(\ell)$, we have
$$\mathbf{1}_{B}\circ_{p}\mathbf{1}_{A}=\mathbf{1}_{B\circ_{p}A}.$$
\begin{figure}
\centering
\def\svgscale{0.9}
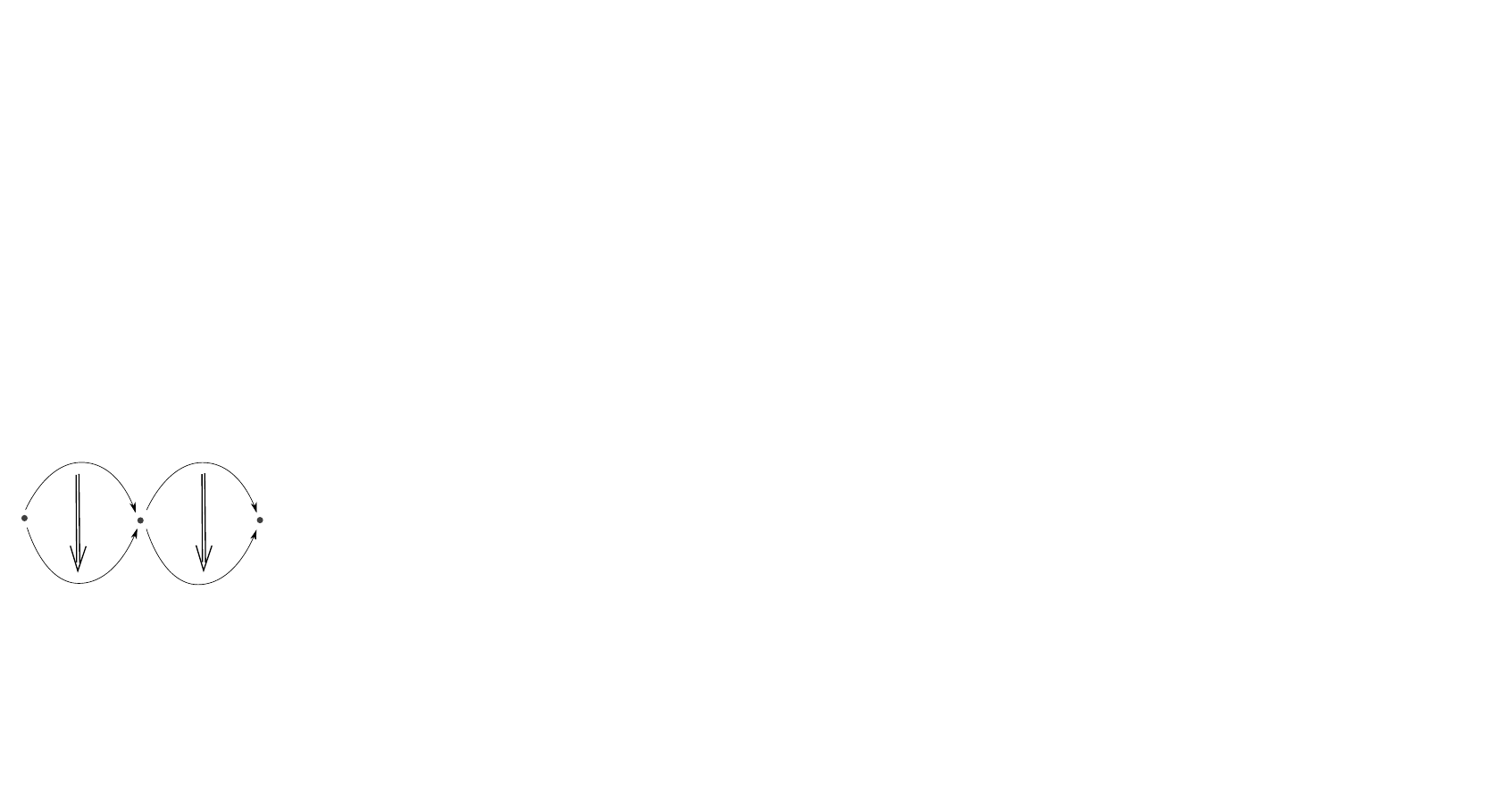
\caption{Axioms for an $m$-category}
\end{figure}
\end{definition}

\begin{remark}
Sometimes to emphasize the $p$-cell $A_{[p]}$ along which we glue two cells, we write $A'_{[\ell]2}\circ_{A_{[p]}}A_{[\ell]1}$ instead of $A'_{[\ell]2}\circ_{p}A_{[\ell]1}$.
\end{remark}

\begin{definition}
A strict $m$-functor between two strict $m$-categories $\mathcal{A}$ and $\mathcal{B}$ is a map $f:\mathcal{A}\longrightarrow \mathcal{B}$ of the underlying $m$-globular sets compatible with composites and identities.
\end{definition}


\begin{definition}
Given a collection of sets $\{Y(\ell)\}_{\ell=0}^{+\infty}$, if for any
$m\in \mathbb{N}$, $\{Y(\ell)\}_{\ell=0}^{m}$ forms an $m$-globular set $Y$, we call the collection a \textbf{$\infty$-globular set}. Similarly, we can define a \textbf{strict $\infty$-category}.

For an $\infty$-globular set, for $\ell\in \Z_{\geqslant 0}$, two $\ell$-cells $A,B\in Y(\ell)$ is said to be \textbf{homotopic} if there exists a $(\ell+1)$-cell $C$ such that $\mathfrak{s}(C)=A$ and $\mathfrak{t}(C)=B$.
{A \textbf{weak $\infty$-category} is an $\infty$-globular set together with functions $\circ_{p}$ and $\mathbf{1}$ as in the Definition \ref{strict m cate def} such that (c),(d),(e) holds up to homotopy in \ref{strict m cate def}.}
\end{definition}


\begin{subsection}
  {Morse weak $\infty$-category}
  \label{Morse weak n pre-category}
\end{subsection}

We build up the Morse weak $n$-category in three steps and claim that we obtain a weak $\infty$-category by extending the $n=dim M$ to an $\infty$-structure.
We firstly construct an $\infty$-globular set, then define the composite and identity maps, and finally the six properties are verified.

{For simplicity, in this section about the algebraic structures behind our previous geometric setting, we use the following convention about the reparametrization of $t\in [0,1]=:I$. This is because $t\in \R$ and the finiteness of $H_{[\ell]}$ and $G_{[\ell]}$, $\R$ can be extended to $\overline{\R}$ which is diffeomorphic to $I$.
Note that we view $t\in I:=[0,1]$ as the time variable, and $s^{\underline{\ell}},s^{\underline{\ell-1}},\cdots,s^{\underline{1}}$ as parameters.

\textbf{Convention:} Let
$H^{\alpha \beta}_{[\ell]}: [0,1]^{\ell}\times I\times M \longrightarrow \mathbb{R},\,\,(s^{\underline{\ell}},s^{\underline{\ell-1}},\cdots,s^{\underline{1}},t,x)\mapsto H^{\alpha \beta}_{[\ell]}(s^{\underline{\ell}},s^{\underline{\ell-1}},\cdots,s^{\underline{1}},t,x)$  be smooth. There exists $\varepsilon>0$ such that
$$H^{\alpha \beta}_{[\ell]}(s^{\underline{\ell}},s^{\underline{\ell-1}},\cdots,s^{\underline{1}},t,\cdot) = \left\{
 \begin{split}
  &f^{\alpha}&\indent&\text{if t}\in [0, \varepsilon)\\
  &f^{\beta} &\indent&\text{if t}\in (1-\varepsilon,1].
 \end{split}\right .$$
The metric $G_{[\ell]}$ is similar,
$$G_{[\ell]}(s^{\underline{\ell}},s^{\underline{\ell-1}},\cdots,s^{\underline{1}},t) = \left\{
 \begin{split}
  &g^{\alpha}&\indent&\text{if t}\in [0, \varepsilon)\\
  &g^{\beta} &\indent&\text{if t}\in (1-\varepsilon,1].
 \end{split}\right .$$
 }

To obtain a proper $\infty$-globular set, we need to define a collection of sets $Y$, source maps $\mathfrak{s}$ and target maps $\mathfrak{t}$.

We first fix two Morse-Smale pairs: $(f^{\alpha},g^{\alpha})$ and $(f^{\beta},g^{\beta})$ based on which the $\infty$-globular set is constructed. Recall that $\mathcal{G}:=\mathcal{G}_{g_{0}}$ is the space of all Riemannian metrics on $M$ obtained through a fixed Riemannian metric $g_{0}$.
To get the corresponding algebraic structures from the desired moduli spaces, we need further requirements on generic $G_{[\ell]}$.

\textbf{Requirements} on $G_{[\ell]}$ $(\star)$\label{requirementofG}:

(1) {For fixed $H_{[\ell]}$ and critical points $p\in Crf^{\alpha}$ and $q\in Crf^{\beta}$ such that $indp-indq=1-\ell$, we require $G_{[\ell]}$ to be generic with respect to $(H_{[\ell]},p,q)$;

(2) Let critical points $c\in Crf^{\alpha}$ and   $c'\in Crf^{\beta}$ such that $indp-indc'=indc-indq=-\ell$ with
$$\widehat{S}(p,c'):=\Big\{\widehat{\mathbf{s}}\in [0,1]^{\ell}\Big| G_{[\ell]}(\widehat{\mathbf{s}},\cdot)\in \Sigma^{c}(p,c',H_{[\ell]}(\widehat{\mathbf{s}},\cdot,\cdot))\Big\}$$
depending on $(H_{[\ell]},p,c')$. Here $\Sigma^{c}$ is the non-generic set defined {in Regularity and smoothness in \S 3}. 
The set $\widehat{S}(c,q)$ is defined similarly. We require $G_{[\ell]}$ to assure that $\widehat{S}(p,c')$ ($\widehat{S}(c,q)$ respectively) is of codimension one, and intersects the boundary of $[0,1]^\ell$ transversally. {Furthermore we require such codimension one set to be the union of finitely many codimension one smooth manifolds with transversal intersections.}

\begin{remark}
(1) 
{Taking $\widehat{S}(p,c')$ as an example, we can denote it as $\widehat{S}(p,c')=\bigcup\limits_{i}\limits^{k}\widehat{S}_{i}(p,c')$ ($k\in \Z_{>0}$) with each $\widehat{S}_{i}(p,c')$ a codimension one smooth manifold.}

(2) When we fix $p,q$ and $H_{[\ell]}$, we have codimension one sets $\widehat{S}(p,c')$ and $\widehat{S}(c,q)$ for all $ c'\in Cr f^{\beta},\  c\in Cr f^{\alpha}$ such that $indp-indc'=indc-indq=indp-indq-1=-\ell$.

(3) {The requirement that the number of components is finite is reasonable since the parameter space $[0,1]^{\ell}$ is compact.}
\end{remark}

Now we are ready to construct the Morse weak $\infty$-category.

$\mathbf{Step\,1:\, source\, and\, target}$

We define the following sets

$Y(0)=\{(H,G):I \longrightarrow C^{\infty}(M)\times \mathcal{G}_{g_{0}},\ t\mapsto (H,G)(t)$ s.t. the following conditions $\textcircled{1}-\textcircled{3}$ hold$\}$ with
$\textcircled{1}$ $(H,G)$ depend on $t$ smoothly;
$\textcircled{2}$ There exists $\varepsilon>0$ such that $t\in [0,\varepsilon)$, $H(t)= f^{\alpha},\,G(t)= g^{\alpha}$; $t\in (1-\varepsilon,1]$, $H(t)=f^{\beta},\,G(t)=g^{\beta}$; and
$\textcircled{3}$ $H$ is regular w.r.t. $G$.

$Y(1)=\{(H_{[1]},G_{[1]}):[0,1]\times I \longrightarrow C^{\infty}(M)\times \mathcal{G}_{g_{0}},\  (s^{\underline{1}},t)\mapsto (H_{[1]},G_{[1]})(s^{\underline{1}},t) $  s.t. the following conditions  $\textcircled{1}-\textcircled{4}$  hold$\}$ with
$\textcircled{1}$ $(H_{[1]},G_{[1]})$ depend on $s^{\underline{1}}$ and $t$ smoothly;
$\textcircled{2}$ $(H_{[1]},G_{[1]})(0,\cdot), (H_{[1]},G_{[1]})(1,\cdot)\in Y(0)$;
$\textcircled{3}$ $G_{[1]}$ is generic and satisfies $(\star)$; and
$\textcircled{4}$ $H_{[1]}$ is regular w.r.t. $G_{[1]}$.

In general for $\ell\geqslant 2$, $Y(\ell)=\{(H_{[\ell]},G_{[\ell]}):\underbrace{[0,1]\times\cdots\times[0,1]}\limits_{\ell}\times I\longrightarrow C^{\infty}(M)\times \mathcal{G}_{g_{0}}, (s^{\underline{\ell}},\cdots,s^{\underline{1}},t)\mapsto (H_{[\ell]},G_{[\ell]})(s^{\underline{\ell}},\cdots,s^{\underline{1}},t)$ s.t. the following conditions $\textcircled{1}-\textcircled{4}$ hold$\}$ with
$\textcircled{1}$ $(H_{[\ell]},G_{[\ell]})$ depend on $s^{\underline{\ell}},\cdots,s^{\underline{1}}$ and $t$ smoothly;
$\textcircled{2}$ $(H_{[\ell]},G_{[\ell]})(\cdots,\underbrace{0}\limits_{i-th},\cdots)=(H_{[\ell]},G_{[\ell]})(\underbrace{0,\cdots,0}\limits_{\ell+2-i},\underbrace{\cdot,\cdots,\cdot}_{i-1}),\ (H_{[\ell]},G_{[\ell]})(\cdots,\underbrace{1}\limits_{i-th},\cdots)=(H_{[\ell]},G_{[\ell]})(\underbrace{0,\cdots,0,1}\limits_{\ell+2-i},\underbrace{\cdot,\cdots,\cdot}_{i-1})\in Y(i-1), i=1,2,\cdots,l$;
$\textcircled{3}$ $G_{[\ell]}$ is generic and satisfies $(\star)$; and
$\textcircled{4}$ $H_{[\ell]}$ is regular w.r.t. $G_{[\ell]}$.

The condition $\textcircled{2}$ in each $Y(i),i=2,3,\cdots$ tells us the degenerations and implies that $\textcircled{2}'$: $(H_{[\ell]},G_{[\ell]})(0,0,\cdots)=(H_{[\ell]},G_{[\ell]})(\cdot\ ,0,\cdots)=(H_{[\ell]},G_{[\ell]})(1,0,\cdots)$; $(H_{[\ell]},G_{[\ell]})(0,1,\cdots)=(H_{[\ell]},G_{[\ell]})( \cdot\ ,1,\cdots)=(H_{[\ell]},G_{[\ell]})(1,1,\cdots)$.
Such atavistic property make the definitions of the source and target maps possible. The conditions $\textcircled{3}$ and $\textcircled{4}$ ensure that we are able to obtain the desired moduli spaces.

Define the source and target maps:
$$\mathfrak{s}:Y(\ell)\longrightarrow Y(\ell-1);\  (H_{[\ell]},G_{[\ell]})(\cdots)\mapsto (H_{[\ell]},G_{[\ell]})(0,\cdots);$$
$$\mathfrak{t}:Y(\ell)\longrightarrow Y(\ell-1);\  (H_{[\ell]},G_{[\ell]})(\cdots)\mapsto (H_{[\ell]},G_{[\ell]})(1,\cdots);$$

\begin{claim}
$\mathfrak{s}\circ \mathfrak{s}= \mathfrak{s}\circ \mathfrak{t}$; $\mathfrak{t}\circ \mathfrak{s}= \mathfrak{t}\circ \mathfrak{t}$
\end{claim}

\begin{proof}
Take any element $(H_{[\ell]},G_{[\ell]})$ in $Y(\ell)$, by condition $\textcircled{2}'$ for $Y(\ell)$,\\
$1.$ $\mathfrak{s}\circ \mathfrak{s}(H_{[\ell]},G_{[\ell]}) =(H_{[\ell]},G_{[\ell]})(0,0,\cdots)=(H_{[\ell]},G_{[\ell]})(\cdot\ ,0,\cdots)=(H_{[\ell]},G_{[\ell]})(1,0,\cdots)=\mathfrak{s}\circ \mathfrak{t}(H_{[\ell]},G_{[\ell]})$. \\
$2.$ $\mathfrak{t}\circ \mathfrak{s}(H_{[\ell]},G_{[\ell]})=(H_{[\ell]},G_{[\ell]})(0,1,\cdots)=(H_{[\ell]},G_{[\ell]})( \cdot\ ,1,\cdots)=(H_{[\ell]},G_{[\ell]})(1,1,\cdots)= \mathfrak{t}\circ \mathfrak{t}(H_{[\ell]},G_{[\ell]})$.
\end{proof}

The above collection of sets together with $\mathfrak{s},\mathfrak{t}$ forms an $\infty$-globular set.

$\mathbf{Step\,2: composite\,\, and\,\, indentity}$

$1.\,\mathbf{(composite)}$
We define the composite $\circ _{p}:Y(\ell)\times Y(\ell)\longrightarrow Y(\ell)$ of two $\ell$-cells along a $p$-cell
for $0\leqslant p < \ell$ to be $$\circ _{p}((H_{[\ell]},G_{[\ell]})_{2},(H_{[\ell]},G_{[\ell]})_{1}):=(H_{[\ell]},G_{[\ell]})_{2}\circ_{(H_{[p]},G_{[p]})}(H_{[\ell]},G_{[\ell]})_{1}$$ with
 $$(H_{[\ell]},G_{[\ell]})_{2}\circ_{(H_{[p]},G_{[p]})}(H_{[\ell]},G_{[\ell]})_{1}:= \left\{
 \begin{split}
  &(H_{[\ell]},G_{[\ell]})_{1}(s^{\underline{\ell}},\cdots,2\cdot s^{\underline{p+1}},\cdots,s^{\underline{1}},t) &\indent& if\, s^{\underline{p+1}}\in [0, \frac{1}{2}]\\
  &(H_{[\ell]},G_{[\ell]})_{2}(s^{\underline{\ell}},\cdots,2\cdot s^{\underline{p+1}}-1,\cdots,s^{\underline{1}},t) &\indent& if\, s^{\underline{p+1}}\in [\frac{1}{2}, 1]
 \end{split}\right.$$
 such that $(H_{[\ell]},G_{[\ell]})_{1}$ and $(H_{[\ell]},G_{[\ell]})_{2}$ connect smoothly at $s^{\underline{p+1}}= \frac{1}{2}$.

$2. \,\mathbf{(identity)}$
{For $\ell\geqslant 0$, the identity $\mathbf{1}:Y(\ell)\longrightarrow Y(\ell+1)$ is a map such that
$$\mathbf{1}_{(H_{[\ell]},G_{[\ell]})}(s^{\underline{\ell+1}},s^{\underline{\ell}},\cdots,t)=\mathbf{1}((H_{[\ell]},G_{[\ell]})):=(H_{[\ell]},G_{[\ell]})(s^{\underline{\ell}},\cdots), \forall s^{\underline{\ell+1}}\in [0,1].$$}

$\mathbf{Step\,3:\, compatible\, conditions}$

We check that the six properties in the definition of the weak $\infty$-category are fulfilled.

$(a)\mathbf{(sources\ and\ targets\ of\ composites)}$

For $p=\ell-1$ and $((H_{[\ell]},G_{[\ell]})_{1},(H_{[\ell]},G_{[\ell]})_{2})\in Y(\ell)\times_{\ell-1} Y(\ell)$,
 $$(H_{[\ell]},G_{[\ell]})_{2}\circ_{(\ell-1)}(H_{[\ell]},G_{[\ell]})_{1}:= \left\{
 \begin{split}
  &(H_{[\ell]},G_{[\ell]})_{1}(2\cdot s^{\underline{\ell}},\cdots,s^{\underline{1}},t) &\indent& if\, s^{\underline{\ell}}\in [0, \frac{1}{2}]\\
  &(H_{[\ell]},G_{[\ell]})_{2}(2\cdot s^{\underline{\ell}}-1,\cdots,s^{\underline{1}},t) &\indent& if\, s^{\underline{\ell}}\in [\frac{1}{2}, 1].
 \end{split}\right.$$
So,
\begin{flalign*}
       & \mathfrak{s}((H_{[\ell]},G_{[\ell]})_{2}\circ_{(\ell-1)}(H_{[\ell]},G_{[\ell]})_{1}(s^{\underline{\ell}},\cdots,s^{\underline{1}},t)) 
        = (H_{[\ell]},G_{[\ell]})_{2}\circ_{(\ell-1)}(H_{[\ell]},G_{[\ell]})_{1}(0,s^{\underline{\ell-1}},\cdots,s^{\underline{1}},t)\\
       & = (H_{[\ell]},G_{[\ell]})_{1}(0,s^{\underline{\ell-1}},\cdots,s^{\underline{1}},t) 
        = \mathfrak{s}((H_{[\ell]},G_{[\ell]})_{1})(s^{\underline{\ell-1}},\cdots,s^{\underline{1}},t)
\end{flalign*}
The second equality holds due to the definition.

For $p< \ell-1$ and $((H_{[\ell]},G_{[\ell]})_{1},(H_{[\ell]},G_{[\ell]})_{2})\in Y(\ell)\times_{p} Y(\ell)$,
 $$\mathfrak{s}((H_{[\ell]},G_{[\ell]})_{2}\circ_{p}(H_{[\ell]},G_{[\ell]})_{1})\\= \left\{
 \begin{split}
  &(H_{[\ell]},G_{[\ell]})_{1}(0,s^{\underline{\ell-1}},\cdots,2\cdot s^{\underline{p+1}},\cdots,s^{\underline{1}},t) &\indent& if\, s^{\underline{p+1}}\in [0, \frac{1}{2}]\\
  &(H_{[\ell]},G_{[\ell]})_{2}(0,s^{\underline{\ell-1}},\cdots,2\cdot s^{\underline{p+1}}-1,\cdots,s^{\underline{1}},t) &\indent& if\, s^{\underline{p+1}}\in [\frac{1}{2}, 1]
 \end{split}\right.$$
which is exactly equals to $\mathfrak{s}((H_{[\ell]},G_{[\ell]})_{2})\circ_{p}\mathfrak{s}((H_{[\ell]},G_{[\ell]})_{1})$.

  Similarly, we can check the corresponding properties for $\mathfrak{t}$ in Definition \ref{strict m cate def}.

$(b)\mathbf{(sources\ and\ targets\ of\ identities)}$

By the definition of the identity $\mathbf{1}$, for any $(H_{[\ell]},G_{[\ell]})\in Y(\ell)$, $\mathbf{1}_{(H_{[\ell]},G_{[\ell]})}(s^{\underline{\ell+1}},\cdots,t)=(H_{[\ell]},G_{[\ell]})(s^{\underline{\ell}},\cdots,t)$, $\forall s^{\underline{\ell+1}}\in[0,1]$. Take $s^{\underline{\ell+1}}=0,1$ respectively, then $$\mathfrak{s}(\mathbf{1}_{(H_{[\ell]},G_{[\ell]})})= \mathbf{1}_{(H_{[\ell]},G_{[\ell]})}(0,\cdots )= (H_{[\ell]},G_{[\ell]}),\,  \mathfrak{t}(\mathbf{1}_{(H_{[\ell]},G_{[\ell]})})=\mathbf{1}_{(H_{[\ell]},G_{[\ell]})}(1,\cdots) =(H_{[\ell]},G_{[\ell]}).$$

$(c)\mathbf{(associativity\ (up\ to\ homotopy))}$

Let $(H_{[\ell]},G_{[\ell]})_{1},(H_{[\ell]},G_{[\ell]})_{2},(H_{[\ell]},G_{[\ell]})_{3} \in Y(\ell)$, $((H_{[\ell]},G_{[\ell]})_{3},(H_{[\ell]},G_{[\ell]})_{2}),((H_{[\ell]},G_{[\ell]})_{2},(H_{[\ell]},G_{[\ell]})_{1})\in  Y(\ell)\times_{p} Y(\ell)$. We need to check that $((H_{[\ell]},G_{[\ell]})_{3}\circ_{p}(H_{[\ell]},G_{[\ell]})_{2})\circ_{p}(H_{[\ell]},G_{[\ell]})_{1}$ equals to  $(H_{[\ell]},G_{[\ell]})_{3}\circ_{p}((H_{[\ell]},G_{[\ell]})_{2}\circ_{p}(H_{[\ell]},G_{[\ell]})_{1})$ up to homotopy. More explictely,
 $$((H_{[\ell]},G_{[\ell]})_{3}\circ_{p}(H_{[\ell]},G_{[\ell]})_{2})\circ_{p}(H_{[\ell]},G_{[\ell]})_{1}:= \left\{
 \begin{split}
  &(H_{[\ell]},G_{[\ell]})_{1}(s^{\underline{\ell}},\cdots,2\cdot s^{\underline{p+1}},\cdots,s^{\underline{1}},t) &\indent& \textrm{if}\, s^{\underline{p+1}}\in [0, \frac{1}{2}]\\
  &(H_{[\ell]},G_{[\ell]})_{2}(s^{\underline{\ell}},\cdots,4\cdot s^{\underline{p+1}}-2,\cdots,s^{\underline{1}},t) &\indent& \textrm{if}\, s^{\underline{p+1}}\in [\frac{1}{2}, \frac{3}{4}]\\
  &(H_{[\ell]},G_{[\ell]})_{3}(s^{\underline{\ell}},\cdots,4\cdot s^{\underline{p+1}}-3,\cdots,s^{\underline{1}},t) &\indent& \textrm{if}\, s^{\underline{p+1}}\in [\frac{3}{4}, 1]
 \end{split}\right.$$
and
  $$(H_{[\ell]},G_{[\ell]})_{3}\circ_{p}((H_{[\ell]},G_{[\ell]})_{2}\circ_{p}(H_{[\ell]},G_{[\ell]})_{1}):= \left\{
 \begin{split}
  &(H_{[\ell]},G_{[\ell]})_{1}(s^{\underline{\ell}},\cdots,4\cdot s^{\underline{p+1}},\cdots,s^{\underline{1}},t) &\indent& \textrm{if}\, s^{\underline{p+1}}\in [0, \frac{1}{4}]\\
  &(H_{[\ell]},G_{[\ell]})_{2}(s^{\underline{\ell}},\cdots,4\cdot s^{\underline{p+1}}-1,\cdots,s^{\underline{1}},t) &\indent& \textrm{if}\, s^{\underline{p+1}}\in [\frac{1}{4}, \frac{1}{2}]\\
  &(H_{[\ell]},G_{[\ell]})_{3}(s^{\underline{\ell}},\cdots,2\cdot s^{\underline{p+1}}-1,\cdots,s^{\underline{1}},t) &\indent& \textrm{if}\, s^{\underline{p+1}}\in [\frac{1}{2}, 1]
 \end{split}\right.$$
 Let $T:[0,1]\times ([0,1]\times\cdots \times[0,1]\times I)\longrightarrow C^{\infty}(M)\times \mathcal{G}_{g_{0}}$ be
  $$T(r,s^{\underline{\ell}},\cdots,s^{\underline{1}},t):= \left\{
 \begin{split}
  &(H_{[\ell]},G_{[\ell]})_{1}(s^{\underline{\ell}},\cdots,\frac{4}{r+1}\cdot s^{\underline{p+1}},\cdots,s^{\underline{1}},t) &\indent& \textrm{if}\, s^{\underline{p+1}}\in [0, \frac{r+1}{4}]\\
  &(H_{[\ell]},G_{[\ell]})_{2}(s^{\underline{\ell}},\cdots,4\cdot s^{\underline{p+1}}-r-1,\cdots,s^{\underline{1}},t) &\indent& \textrm{if}\, s^{\underline{p+1}}\in [\frac{r+1}{4}, \frac{r+2}{4}]\\
  &(H_{[\ell]},G_{[\ell]})_{3}(s^{\underline{\ell}},\cdots,\frac{4}{2-r}\cdot s^{\underline{p+1}}+\frac{-r-2}{2-r},\cdots,s^{\underline{1}},t) &\indent& \textrm{if}\, s^{\underline{p+1}}\in [\frac{r+2}{4}, 1]
 \end{split}\right.$$
   with $r$ the homotopy parameter. Then $T(0,\cdots)=((H_{[\ell]},G_{[\ell]})_{3}\circ_{p}(H_{[\ell]},G_{[\ell]})_{2})\circ_{p}(H_{[\ell]},G_{[\ell]})_{1}$ and $T(1,\cdots)=(H_{[\ell]},G_{[\ell]})_{3}\circ_{p}((H_{[\ell]},G_{[\ell]})_{2}\circ_{p}(H_{[\ell]},G_{[\ell]})_{1})$. Therefore, the equality holds up to homotopy.

$(d)\mathbf{(identities(up\ to\ homotopy))}$

For the left hand side ($LHS$),
 \begin{align*}
   &\mathbf{1}^{\ell-p}(\mathfrak{t}^{\ell-p}((H_{[\ell]},G_{[\ell]})(s^{\underline{\ell}},\cdots,s^{\underline{1}},t)))
   = \mathbf{1}^{\ell-p}((H_{[\ell]},G_{[\ell]})(\underbrace{1,\cdots,1}\limits_{\ell-p},\underbrace{s^{\underline{p}},\cdots,s^{\underline{1}}}\limits_{p},t)) \\
 & =
       (\mathbf{1}^{\ell-p}(H_{[\ell]},G_{[\ell]}))(w^{\underline{\ell-p}},\cdots,w^{\underline{1}},1,\cdots,1,s^{\underline{p}},\cdots,s^{\underline{1}},t).
\end{align*}
 We rename the parameters $w^{\underline{\ell-p}},\cdots,w^{\underline{1}}$ by $s^{\underline{\ell}},\cdots,s^{\underline{p+1}}$,\\
 $$LHS=(\mathbf{1}^{\ell-p}(H_{[\ell]},G_{[\ell]}))(s^{\underline{\ell}},\cdots,s^{\underline{p+1}},1,\cdots,1,s^{\underline{p}},\cdots,s^{\underline{1}},t)\circ_{p} (H_{[\ell]},G_{[\ell]})(s^{\underline{\ell}},\cdots,s^{\underline{1}},t)$$
   $$= \left\{
 \begin{split}
  &(H_{[\ell]},G_{[\ell]})(s^{\underline{\ell}},\cdots,2\cdot s^{\underline{p+1}},\cdots,s^{\underline{1}},t)   &\indent& \textrm{if}\, s^{\underline{p+1}}\in [0, \frac{1}{2}]\\
  &(\mathbf{1}^{\ell-p}(H_{[\ell]},G_{[\ell]}))(s^{\underline{\ell}},\cdots,2\cdot s^{\underline{p+1}}-1,\cdots,1,\cdots ,1,s^{\underline{p}},\cdots,s^{\underline{1}},t)   &\indent& \textrm{if}\, s^{\underline{p+1}}\in [\frac{1}{2}, 1].
 \end{split}\right.$$
  Similarly, by the definitions of $\mathbf{1}$ and $\mathfrak{s}$,
  $$ RHS = \left\{
 \begin{split}
  &(\mathbf{1}^{\ell-p}(H_{[\ell]},G_{[\ell]}))(s^{\underline{\ell}},\cdots,2\cdot s^{\underline{p+1}},0,\cdots ,0,s^{\underline{p}},\cdots,s^{\underline{1}},t)   &\indent& \textrm{if}\, s^{\underline{p+1}}\in [0,\frac{1}{2}]\\
  &(H_{[\ell]},G_{[\ell]})(s^{\underline{\ell}},\cdots,2\cdot s^{\underline{p+1}}-1,\cdots,s^{\underline{1}},t)   &\indent& \textrm{if}\, s^{\underline{p+1}}\in [\frac{1}{2},1].
 \end{split}\right.$$
  Let $T_{1}:[0,1]\times ([0,1]\times\cdots \times[0,1]\times I)\longrightarrow C^{\infty}(M)\times \mathcal{G}_{g_{0}}$;
  $$ T_{1}(r,s^{\underline{\ell}},\cdots,s^{\underline{1}},t): = \left\{
 \begin{split}
  &LHS(s^{\underline{\ell}},\cdots,\frac{2}{r+1}\cdot s^{\underline{p+1}},\cdots,s^{\underline{1}},t)   &\indent& \textrm{if}\, s^{\underline{p+1}}\in [0,\frac{1+r}{2}]\\
  &(H_{[\ell]},G_{[\ell]})(s^{\underline{\ell}},\cdots,\frac{2}{1-r}\cdot s^{\underline{p+1}}+\frac{-r-1}{1-r},\cdots,s^{\underline{1}},t)   &\indent& \textrm{if}\, s^{\underline{p+1}}\in [\frac{1+r}{2},1].
 \end{split}\right.$$
   with $r$ the homotopy parameter. {Then $T_{1}(0,\cdots)=LHS$ and $T_{1}(1,\cdots)=(H_{[\ell]},G_{[\ell]})$. Similarly, we can construct a homotopy $T_{2}$ from $(H_{[\ell]},G_{[\ell]})$ to $RHS$. Moreover, there is a homotopy $T$ from $LHS$ to $RHS$.} Therefore the equality holds up to homotopy.

$(e)\mathbf{(binary\ interchange(up\ to\ homotopy))}$

$LHS= ((H_{[\ell]},G_{[\ell]})_{4}\circ_{p} (H_{[\ell]},G_{[\ell]})_{3}) \circ_{q}((H_{[\ell]},G_{[\ell]})_{2}\circ_{p} (H_{[\ell]},G_{[\ell]})_{1})$ and
$RHS= ((H_{[\ell]},G_{[\ell]})_{4}\circ_{q} (H_{[\ell]},G_{[\ell]})_{2}) \circ_{p}((H_{[\ell]},G_{[\ell]})_{3}\circ_{q} (H_{[\ell]},G_{[\ell]})_{1 })$.
Without loss of generality, let's suppose $\underline{p} > \underline{q}$,
  $$ LHS = \left\{
 \begin{split}
  &(H_{[\ell]},G_{[\ell]})_{1}(s^{\underline{\ell}},\cdots,2\cdot s^{\underline{p+1}},\cdots,2\cdot s^{\underline{q+1}},\cdots,s^{\underline{1}},t) &\indent& \textrm{if}\, s^{\underline{p+1}}\in [0,\frac{1}{2}] &\indent&  s^{\underline{q+1}}\in [0,\frac{1}{2}]\\
  &(H_{[\ell]},G_{[\ell]})_{2}(s^{\underline{\ell}},\cdots,2\cdot s^{\underline{p+1}}-1,\cdots,2\cdot s^{\underline{q+1}},\cdots,s^{\underline{1}},t) &\indent&\textrm{if}\, s^{\underline{p+1}}\in [\frac{1}{2},1] &\indent& s^{\underline{q+1}}\in [0,\frac{1}{2}]\\
  &(H_{[\ell]},G_{[\ell]})_{3}(s^{\underline{\ell}},\cdots,2\cdot s^{\underline{p+1}},\cdots,2\cdot s^{\underline{q+1}}-1,\cdots,s^{\underline{1}},t) &\indent& \textrm{if}\, s^{\underline{p+1}}\in [0,\frac{1}{2}] &\indent&  s^{\underline{q+1}}\in [\frac{1}{2},1]\\
  &(H_{[\ell]},G_{[\ell]})_{4}(s^{\underline{\ell}},\cdots,2\cdot s^{\underline{p+1}}-1,\cdots,2\cdot s^{\underline{q+1}}-1,\cdots,s^{\underline{1}},t) &\indent& \textrm{if}\, s^{\underline{p+1}}\in [\frac{1}{2},1] &\indent&  s^{\underline{q+1}}\in [\frac{1}{2},1]
 \end{split}\right.$$
and
  $$ RHS = \left\{
 \begin{split}
  &(H_{[\ell]},G_{[\ell]})_{1}(s^{\underline{\ell}},\cdots,2\cdot s^{\underline{p+1}},\cdots,2\cdot s^{\underline{q+1}},\cdots,s^{\underline{1}},t) &\indent& \textrm{if}\, s^{\underline{q+1}}\in [0,\frac{1}{2}] &\indent&  s^{\underline{p+1}}\in [0,\frac{1}{2}]\\
  &(H_{[\ell]},G_{[\ell]})_{3}(s^{\underline{\ell}},\cdots,2\cdot s^{\underline{p+1}},\cdots,2\cdot s^{\underline{q+1}}-1,\cdots,s^{\underline{1}},t) &\indent& \textrm{if}\, s^{\underline{q+1}}\in [\frac{1}{2},1] &\indent& s^{\underline{p+1}}\in [0,\frac{1}{2}]\\
  &(H_{[\ell]},G_{[\ell]})_{2}(s^{\underline{\ell}},\cdots,2\cdot s^{\underline{p+1}}-1,\cdots,2\cdot s^{\underline{q+1}},\cdots,s^{\underline{1}},t) &\indent& \textrm{if}\, s^{\underline{q+1}}\in [0,\frac{1}{2}] &\indent&  s^{\underline{p+1}}\in [\frac{1}{2},1]\\
  &(H_{[\ell]},G_{[\ell]})_{4}(s^{\underline{\ell}},\cdots,2\cdot s^{\underline{p+1}}-1,\cdots,2\cdot s^{\underline{q+1}}-1,\cdots,s^{\underline{1}},t) &\indent& \textrm{if}\, s^{\underline{q+1}}\in [\frac{1}{2},1] &\indent&  s^{\underline{p+1}}\in [\frac{1}{2},1]
 \end{split}\right.$$
{Now one can see that $LHS=RHS$.}

$(f)\mathbf{(nullary\ interchange)}$
  $$LHS=\mathbf{1}_{(H_{[\ell]},G_{[\ell]})_{2}}\circ \mathbf{1}_{(H_{[\ell]},G_{[\ell]})_{1}}:= \left\{
 \begin{split}
  &\mathbf{1}_{(H_{[\ell]},G_{[\ell]})_{1}}(s^{\underline{\ell+1}},s^{\underline{\ell}},\cdots,2\cdot s^{\underline{p+1}},\cdots,s^{\underline{1}},t) &\indent& \textrm{if}\, s^{\underline{p+1}}\in [0, \frac{1}{2}]\\
  &\mathbf{1}_{(H_{[\ell]},G_{[\ell]})_{2}}(s^{\underline{\ell+1}},s^{\underline{\ell}},\cdots,2\cdot s^{\underline{p+1}}-1,\cdots,s^{\underline{1}},t) &\indent& \textrm{if}\, s^{\underline{p+1}}\in [\frac{1}{2}, 1]
 \end{split}\right.$$
and
  $$RHS=\mathbf{1}_{(H_{[\ell]},G_{[\ell]})_{2}\circ (H_{[\ell]},G_{[\ell]})_{1}}
  = (H_{[\ell]},G_{[\ell]})_{2}\circ (H_{[\ell]},G_{[\ell]})_{1}(s^{\underline{\ell}},\cdots,s^{\underline{1}},t), \forall s^{\underline{\ell+1}}$$
  $$= \left\{
 \begin{split}
  &(H_{[\ell]},G_{[\ell]})_{1}(s^{\underline{\ell}},\cdots,2\cdot s^{\underline{p+1}},\cdots,s^{\underline{1}},t) &\indent& \textrm{if}\, s^{\underline{p+1}}\in [0, \frac{1}{2}] &\forall s^{\underline{\ell+1}}\\
  &(H_{[\ell]},G_{[\ell]})_{2}(s^{\underline{\ell}},\cdots,2\cdot s^{\underline{p+1}}-1,\cdots,s^{\underline{1}},t,) &\indent& \textrm{if}\, s^{\underline{p+1}}\in [\frac{1}{2}, 1]  &\forall s^{\underline{\ell+1}}
 \end{split}\right.$$
   $$ = \left\{
 \begin{split}
  &\mathbf{1}_{(H_{[\ell]},G_{[\ell]})_{1}}(s^{\underline{\ell+1}},s^{\underline{\ell}},\cdots,2\cdot s^{\underline{p+1}},\cdots,s^{\underline{1}},t) &\indent& \textrm{if}\, s^{\underline{p+1}}\in [0, \frac{1}{2}] \\
  &\mathbf{1}_{(H_{[\ell]},G_{[\ell]})_{2}}(s^{\underline{\ell+1}},s^{\underline{\ell}},\cdots,2\cdot s^{\underline{p+1}}-1,\cdots,s^{\underline{1}},t) &\indent& \textrm{if}\, s^{\underline{p+1}}\in [\frac{1}{2}, 1]
 \end{split}\right.$$
Consequently, $LHS=RHS$.

\begin{remark}

(1) Concerning the requirement $(\star)$, for the composite $(H_{[\ell]},G_{[\ell]})_{2}\circ_{p} (H_{[\ell]},G_{[\ell]})_{1}$, its codimension one sets $\widehat{S}(p,c')$ and $\widehat{S}(c,q)$ satisfy
$\widehat{S}(p,c')=\widehat{S}_{1}(p,c')\cup \widehat{S}_{2}(p,c')$ and $\widehat{S}(c,q)=\widehat{S}_{1}(c,q)\cup \widehat{S}_{2}(c,q)$ with $\widehat{S}_{1}(p,c'),\  \widehat{S}_{1}(c,q)$ and  $\widehat{S}_{2}(p,c'),\  \widehat{S}_{2}(c,q)$ corresponding to $(H_{[\ell]1},G_{[\ell]1})$, $(H_{[\ell]2},G_{[\ell]2})$ respectively.

{(2) The above constructions for homotopies assure that their corresponding cells $(H_{[\ell+1]},G_{[\ell+1]})$ satisfy the codimension one condition. In addition, the cell and homotopy preserve generic and regular conditions once their ends are generic and regular.}
\end{remark}

\begin{subsection}
  {Strict $n$-category $\mathbf{MSWCx}$}
  \label{strict n cate CHCx}
\end{subsection}

From two fixed Morse-Smale pairs $(f^{\alpha},g^{\alpha})$ and $(f^{\beta},g^{\beta})$ on the smooth closed manifold $M$ of dimension $n$, we obtain two MSW complexes $C:=(C(f^{\alpha},g^{\alpha}),\partial)$ and $C':=(C(f^{\beta},g^{\beta}),\partial')$ well known in Morse theory. We will construct an $n$-globular set together with composite and identity which finally turns out to be a strict $n$-category, then we extend it trivially to a strict $\infty$-category by adding $Y'(k)=\{0\},\, k>n$.
Note that the category we construct is different from the example of Leinster \cite{leinster} (p. 12).
As usual the shifting of the complex $C$ is $$(C[i])_{r}:=C_{r+i}, \, i,r\in \Z.$$

We start with the strict $n$-category.

$\mathbf{Step\,1:\, source\, and\, target}$

We define the following sets:

$$Y(0)'=\Big\{\phi:C\longrightarrow C' \Big| \partial'\phi=\phi\partial\Big\}.$$
This means that $Y(0)'$ consists of all chain maps from $C$ to $C'$. And as usual $\partial'\phi=\phi\partial$ means $0=-\phi_{r-1}\partial_{r}+\partial'_{r}\phi_{r},\, \forall r=1,\cdots,n.$

\begin{tikzcd}
 C_{n} \arrow[r, "\partial_{n}"]  \arrow[d, "\phi_{n}"]
& \cdots  \arrow[r]
& C_{r+1} \arrow[r, "\partial_{r+1}"]  \arrow[d, "\phi_{r+1}"]
& C_{r} \arrow[r, "\partial_{r}"]  \arrow[d, "\phi_{r}"]
& C_{r-1} \arrow[r, "\partial_{r-1}"]  \arrow[d, "\phi_{r-1}"]
& \cdots \arrow[r]
& C_{0} \arrow[r]  \arrow[d, "\phi_{0}"]
& 0
\\
 C'_{n} \arrow[r, "\partial'_{n}"]
& \cdots  \arrow[r]
& C'_{r+1} \arrow[r, "\partial'_{r+1}"]
& C'_{r} \arrow[r, "\partial'_{r}"]
& C'_{r-1} \arrow[r, "\partial'_{r-1}"]
& \cdots\arrow[r]
& C_{0} \arrow[r]
& 0.
\end{tikzcd}

$$Y(1)'=\Big\{\phi_{[1]}:C\longrightarrow C'[-1] \Big| \phi_{1}-\phi_{0}=\partial'\phi_{[1]}+\phi_{[1]}\partial,\ \forall \phi_{0},\phi_{1}\in Y(0)' \Big\}.$$
So $Y(1)'$ consists of all chain homotopies from arbitrary chain maps $\phi_{0}$ to $\phi_{1}$. And $\phi_{1}-\phi_{0}=\partial'\phi_{[1]}+\phi_{[1]}\partial$ means that $\phi_{1,r}-\phi_{0,r}=\partial'_{r+1}\phi_{[1]r}+\phi_{[1]r-1}\partial_{r},\, \forall r=1,\cdots,n-1.$

\begin{tikzcd}
 C_{n} \arrow[r, "\partial_{n}"]
& \cdots  \arrow[r]
& C_{r+1} \arrow[r, "\partial_{r+1}"]
& C_{r} \arrow[r, "\partial_{r}"]  \arrow[d, shift left, "\phi_{0,r}"swap]  \arrow[d, shift right, "\phi_{1,r}"]  \arrow[ld, "\phi_{[1]r}"]
& C_{r-1} \arrow[r, "\partial_{r-1}"]  \arrow[ld, "\phi_{[1]r-1}"]
& \cdots\arrow[r]
& C_{0} \arrow[r]
& 0 \\
 C_{n} \arrow[r, "\partial_{n}"]
& \cdots  \arrow[r]
& C'_{r+1} \arrow[r, "\partial'_{r+1}"]
& C'_{r} \arrow[r, "\partial'_{r}"]
& C'_{r-1} \arrow[r, "\partial'_{r-1}"]
& \cdots\arrow[r]
& C'_{0} \arrow[r]
& 0.
\end{tikzcd}

$$Y(2)'=\Big\{\phi_{[2]}:C\longrightarrow C'[-2] \Big| \phi_{[1]1}-\phi_{[1]0}=\partial'\phi_{[2]}-\phi_{[2]}\partial, \,\,\forall \phi_{[1]1},\ \phi_{[1]0} \in Y(1)' 
 \Big\}.$$
This means that $Y(2)'$ consists of all homotopies of chain homotopies from chain homotopies $\phi_{[1]0}$ to $\phi_{[1]1}$ with the same sources and the same targets as stated in the atavistic property. In addition, $\phi_{[1]1}-\phi_{[1]0}=\partial'\phi_{[2]}-\phi_{[2]}\partial$ means that $\phi_{[1]1,r}-\phi_{[1]0,r}=\partial'_{r+2}\phi_{[2]r}-\phi_{[2]r-1}\partial_{r}\, \forall r=1,\cdots,n-2.$

In general, for $0\leqslant \ell \leqslant n$
$$Y(\ell)'=\Big\{\phi_{[\ell]}:C\longrightarrow C'[-\ell] \Big| \phi_{[\ell-1]1}-\phi_{[\ell-1]0}=\partial'\phi_{[\ell]}+(-1)^{\ell+1}\phi_{[\ell]}\partial, \forall \phi_{[\ell-1]1},\ \phi_{[\ell-1]0}\in Y(\ell-1)'
\Big\}.$$
This means that $Y(\ell)'$ consists of all higher homotopies from $\phi_{[\ell-1]0}$ to $\phi_{[\ell-1]1}\in Y(\ell-1)'$ with the same sources and targets in $Y(\ell-2)'$ as required in the atavistic property. Moreover, $$\phi_{[\ell-1]1}-\phi_{[\ell-1]0}=\partial'\phi_{[\ell]}+(-1)^{\ell+1}\phi_{[\ell]}\partial$$ means that $$\phi_{[\ell-1]1,r}-\phi_{[\ell-1]0,r}=\partial'_{r+l}\phi_{[\ell]r}+(-1)^{\ell+1}\phi_{[\ell]r-1}\partial_{r}\, \forall r=1,\cdots,n-\ell.$$

Let $Y'$ be the collection of the above sets, i.e.,
$Y'=\big\{Y'(\ell) \big| 0\leqslant \ell \leqslant n \big\}.$
The construction terminates at $Y(n)'$ because the dimension of $M$ is $n$.

Now we define source and target maps:

For $\phi_{[\ell]}\in Y(\ell)'$ from $\phi_{[\ell-1]0}$ to $\phi_{[\ell-1]1}$, i.e.,  $\phi_{[\ell-1]1}-\phi_{[\ell-1]0}=\partial'\phi_{[\ell]}+(-1)^{\ell+1}\phi_{[\ell]}\partial$, we define
$\mathfrak{s}':Y(\ell)'\longrightarrow Y(\ell-1)';\  \phi_{[\ell]}\mapsto \phi_{[\ell-1]0};$
and
$\mathfrak{t}':Y(\ell)'\longrightarrow Y(\ell-1)';\  \phi_{[\ell]}\mapsto \phi_{[\ell-1]1}.$

\begin{claim}
$\mathfrak{s}'\circ \mathfrak{s}'= \mathfrak{s}'\circ \mathfrak{t}'$; $\mathfrak{t}'\circ \mathfrak{s}'= \mathfrak{t}'\circ \mathfrak{t}'$.
\end{claim}
So $Y'$ is an $n$-globular set.

$\mathbf{Step\,2:\, composite\, and \, identity}$


$1.\mathbf{(composite)}$
Now we define the composition of two $\ell$-cells along a $p$-cell as:\\
$\circ _{p}':Y(\ell)'\times_{p} Y(\ell)'\longrightarrow Y(\ell)'$, for $0\leqslant p < \ell\leqslant n$, $\circ _{p}'(\phi_{[\ell]2},\phi_{[\ell]1}):=\phi_{[\ell]2}\circ _{p}'\phi_{[\ell]1}:=\phi_{[\ell]2}+\phi_{[\ell]1}.$

Notice $\mathfrak{s}'^{\ell-p}(\phi_{[\ell]2})=\mathfrak{t}'^{\ell-p}(\phi_{[\ell]1})$ because $(\phi_{[\ell]2},\phi_{[\ell]1})\in Y(\ell)'\times_{p} Y(\ell)'$.

$2.\mathbf{(identity)}$ We define identitiy map:\\
$\mathbf{1'}:Y(\ell)'\longrightarrow Y(\ell+1)'$, $0\leqslant \ell\leqslant n$; $\mathbf{1}'(\phi_{[\ell]})= \mathbf{1}'_{\phi_{[\ell]}}:=0$, which is from $\phi_{[\ell]}$ to $\phi_{[\ell]}$.

$\mathbf{Step\,3:\, compatible\, conditions}$

We check that the six properties in Definition \ref{strict m cate def} are fulfilled and the globular set $Y'$ is actually a strict $n$-category. To lighten the notions, we denote by $\phi_{[\ell]}:\phi_{[\ell-1]0}\rightsquigarrow \phi_{[\ell-1]1}$ the homotopy $\phi_{[\ell]}$ from $\phi_{[\ell-1]0}$ to $\phi_{[\ell-1]1}$.

$(a)\mathbf{(sources\ and\ targets\ of\ composites)}$

Assume $\phi_{[\ell]}:\phi_{[\ell-1]0}\rightsquigarrow \phi_{[\ell-1]1}\in Y(\ell)'$ and $\phi_{[\ell]}':\phi_{[\ell-1]0}'\rightsquigarrow \phi_{[\ell-1]1}'\in Y(\ell)'$.
For $p=l-1$, let $(\phi_{[\ell]}',\phi_{[\ell]})\in Y(\ell)'\times_{\ell-1} Y(\ell)'$, so $\phi_{[\ell-1]1}=\phi_{[\ell-1]0}'$.
Then we have $\phi_{[\ell]}'\circ_{\phi_{[\ell-1]1}}\phi_{[\ell]}=\phi_{[\ell]}+\phi_{[\ell]}:\phi_{[\ell-1]0}\rightsquigarrow \phi_{[\ell-1]1}'$.
Hence $$\mathfrak{s}'(\phi_{[\ell]}'\circ_{\phi_{[\ell-1]1}}\phi_{[\ell]})=\phi_{[\ell-1]0}=\mathfrak{s}'(\phi_{[\ell]});\,\,\mathfrak{t}'(\phi_{[\ell]}\circ_{\phi_{[\ell-1]1}}\phi_{[\ell]})=\phi_{[\ell-1]1}'=\mathfrak{t}'(\phi_{[\ell]}').$$
For $p\leqslant \ell-2$, let $(\phi_{[\ell]}',\phi_{[\ell]})\in Y(\ell)'\times_{p} Y(\ell)'$, which means $\mathfrak{t}'_{\ell-p}(\phi_{[\ell]})=\mathfrak{s}'_{\ell-p}(\phi_{[\ell]}')$.
We have $\phi_{[\ell]}'\circ_{\phi_{[p]}}\phi_{[\ell]}:\phi_{[\ell-1]0}'\circ_{\phi_{[p]}}\phi_{[\ell-1]0}\rightsquigarrow\phi_{[\ell-1]1}'\circ_{\phi_{[p]}}\phi_{[\ell-1]1}$. So $\mathfrak{s}'(\phi_{[\ell]}'\circ_{\phi_{[p]}}\phi_{[\ell]})=\phi_{[\ell-1]0}'\circ_{\phi_{[p]}}\phi_{[\ell-1]0}=\mathfrak{s}'(\phi_{[\ell]}')\circ_{\phi_{[p]}}\mathfrak{s}'(\phi_{[\ell]}).$ 
Similarly, the properties hold for the target map $\mathfrak{t}'$.

$(b)\mathbf{(sources\ and\ targets\ of\ identities)}$

For $\phi_{[\ell]}\in Y(\ell)'$, $0\leqslant \ell \leqslant n$, by definition, $\mathbf{1}'_{\phi_{[\ell]}}=0:\phi_{[\ell]}\rightsquigarrow \phi_{[\ell]}$, we have $\mathfrak{s}'(\mathbf{1}'_{\phi_{[\ell]}})=\phi_{[\ell]},\mathfrak{t}'(\mathbf{1}'_{\phi_{[\ell]}})=\phi_{[\ell]}.$

$(c)\mathbf{(associativity)}$

Let $\phi_{[\ell]i}:\phi_{[\ell-1]i,0}\rightsquigarrow \phi_{[\ell-1]i,1}$ ($i=1,2,3$).
We have $$\phi_{[\ell]3}\circ_{p}(\phi_{[\ell]2}\circ_{p}\phi_{[\ell]1})=\phi_{[\ell]3}+(\phi_{[\ell]2}+\phi_{[\ell]1})=(\phi_{[\ell]3}+\phi_{[\ell]2})+\phi_{[\ell]1}=(\phi_{[\ell]3}\circ_{p}\phi_{[\ell]2})\circ_{p}\phi_{[\ell]1}.$$
The point is that the operation ``$+$'' has the associativity.

$(d)\mathbf{(identities)}$

Let $0 \leqslant p < \ell \leqslant n$ and $\phi_{[\ell]}\in Y(\ell)'$, we have
$\mathbf{1'}^{\ell-p}(\mathfrak{t}^{\ell-p}(\phi_{[\ell]}))\circ_{p}\phi_{[\ell]}=0+\phi_{[\ell]}=\phi_{[\ell]}+0=\phi_{[\ell]}\circ_{p}\mathbf{1'}^{\ell-p}(\mathfrak{s}^{\ell-p}(\phi_{[\ell]})).$

$(e)\mathbf{(binary\ interchange)}$

Let $0 \leqslant q < p < \ell \leqslant n$ and $\phi_{[\ell]i}\in Y(\ell)' (i=1,2,3,4)$ satisfy $(\phi_{[\ell]4},\phi_{[\ell]3}) \in Y(\ell)'\times_{p}Y(\ell)'$, $(\phi_{[\ell]2},\phi_{[\ell]1}) \in Y(\ell)'\times_{p}Y(\ell)'$, $(\phi_{[\ell]4},\phi_{[\ell]2}) \in Y(\ell)'\times_{q}Y(\ell)'$, $(\phi_{[\ell]3},\phi_{[\ell]1}) \in Y(\ell)'\times_{q}Y(\ell)'$.
We have $$(\phi_{[\ell]4}\circ_{p}\phi_{[\ell]3})\circ_{q}(\phi_{[\ell]2}\circ_{p}\phi_{[\ell]1})=(\phi_{[\ell]4}+\phi_{[\ell]3})+(\phi_{[\ell]2}+\phi_{[\ell]1})
=(\phi_{[\ell]4}+\phi_{[\ell]2})+(\phi_{[\ell]3}+\phi_{[\ell]1})=(\phi_{[\ell]4}\circ_{q}\phi_{[\ell]2})\circ_{p}(\phi_{[\ell]3}\circ_{q}\phi_{[\ell]1}).$$

$(f)\mathbf{(nullary\ interchange)}$

For $0\leqslant p< \ell<n$ and $(\phi_{[\ell]}',\phi_{[\ell]})\in Y(\ell)'\times_{q} Y(\ell)'$,
let $\phi_{[\ell]}:\phi_{[\ell-1]0}\rightsquigarrow\phi_{[\ell-1]1}$, $\mathbf{1}'_{\phi_{[\ell]}}=0:\phi_{[\ell]}\rightsquigarrow\phi_{[\ell]}$;
$\phi_{[\ell]}':\phi_{[\ell-1]0}'\rightsquigarrow\phi_{[\ell-1]1}'$, $\mathbf{1}'_{\phi_{[\ell]}'}=0:\phi_{[\ell]}'\rightsquigarrow\phi_{[\ell]}'$, then we have

$$\mathbf{1}'_{\phi_{[\ell]}'}\circ_{p}\mathbf{1}'_{\phi_{[\ell]}}=0+0=0:\phi_{[\ell]}'\circ_{p}\phi_{[\ell]}\rightsquigarrow\phi_{[\ell]}'\circ_{p}\phi_{[\ell]};
\mathbf{1}'_{\phi_{[\ell]}'\circ_{p}\phi_{[\ell]}}=0:\phi_{[\ell]}'\circ_{p}\phi_{[\ell]}\rightsquigarrow\phi_{[\ell]}'\circ_{p}\phi_{[\ell]}.$$

Here are some examples to illustrate the composites.

(1) two $1$-cells composites along $0$-cell.

Let $\phi_{[1]1}:C\longrightarrow C'[1]$ be a chain homotopy from $\phi_{1}$ to $\phi_{2}$, i.e., $\phi_{2}-\phi_{1}= \phi_{[1]1}\circ \partial+\partial'\circ \phi_{[1]1}$. Let $\phi_{[1]2}:C\longrightarrow C'[1]$ be another chain homotopy from $\phi_{2}$ to $\phi_{3}$, that is, $\phi_{3}-\phi_{2}= \phi_{[1]2}\circ \partial+\partial'\circ \phi_{[1]2}$. Adding them together, we obtain $\phi_{3}-\phi_{1}= (\phi_{[1]1}+\phi_{[1]2})\circ \partial+\partial'\circ (\phi_{[1]1}+\phi_{[1]2}).$
Then the composite of $\phi_{[1]1}$ and $\phi_{[1]2}$ along $\phi_{2}$ is
$\phi_{[1]2}\circ_{\phi_{2}} \phi_{[1]1}:= \phi_{[1]1}+\phi_{[1]2}.$

(2) two $2$-cells composite along $1$-cell and $0$-cell.

 (a) Composite along a $1$-cell.

Let $S_{1},S_{2},S_{3}:C\longrightarrow C'[1]$ be three chain homotopies from chain map $\phi_{0}:C\longrightarrow C'$ to chain map $\phi_{1}:C\longrightarrow C'$. Let $\phi_{[2]1}=T:C\longrightarrow C'[2]$ be a homotopy of chain homotopies from $S_{1}$ to $S_{2}$, that is, $S_{2}-S_{1}= -T\circ \partial+\partial'\circ T$; Let $\phi_{[2]2}=T':C\longrightarrow C'[2]$ be the chain homotopy from $S_{2}$ to $S_{3}$, that is, $S_{3}-S_{2}= -T\circ \partial+\partial'\circ T$. So $S_{3}-S_{1}= -(T+T')\circ \partial+\partial'\circ (T+T')$.

Then the composition of $T$ and $T'$ along $S_{2}$ is
$T'\circ_{S_{2}} T:= T+T'.$

(b) Composite along a $0$-cell.

Let $S_{0},S_{1}:C\longrightarrow C'[1]$ be two chain homotopies from chain map $\phi_{0}:C\longrightarrow C'$ to chain map $\phi_{1}:C\longrightarrow C'$, and $S_{0}',S_{1}':C\longrightarrow C'[1]$ be two chain homotopies from chain map $\phi_{1}:C\longrightarrow C'$ to chain map $\phi_{2}:C\longrightarrow C'$. Let $\phi_{[2]1}=T:C\longrightarrow C'[2]$ be a homotopy of chain homotopies from $S_{0}$ to $S_{1}$, that is, $S_{1}-S_{0}= -T\circ \partial+\partial'\circ T$, and $\phi_{[2]2}=T':C\longrightarrow C'[2]$ be the chain homotopy from $S_{0}'$ to $S_{1}'$, that is, $S_{1}'-S_{0}'= -T'\circ \partial+\partial'\circ T'$. By the composite of $2$-cells along a $1$-cell above, we have $S_{0}'\circ_{\phi_{1}}S_{0}=S_{0}'+S_{0}$ and $S_{1}'\circ_{\phi_{1}}S_{1}=S_{1}'+S_{1}$. Then we get
\begin{equation}
\begin{aligned}
 S_{1}'\circ_{\phi_{1}}S_{1}-S_{0}'\circ_{\phi_{1}}S_{0}
  = &S_{1}'+S_{1}-(S_{0}'+S_{0}) \\
  = &-T'\circ \partial+\partial'\circ T'+(-T\circ \partial+\partial'\circ T) \\
  = &-(T'+T)\circ \partial+ \partial'\circ (T'+T).\nonumber
\end{aligned}
\end{equation}
Then the composition of $T$ and $T'$ along $\phi_{1}$ is
 $T'\circ_{\phi_{1}} T:= T+T'.$

Finally we extend the above strict $n$-category trivially to strict $\infty$-category by adding $Y'(k)=\{0\},\, k>n$ and the morphisms are also extended trivially. We call it the strict $\infty$-category $\mathcal{B}$.

\begin{subsection}
  {Weak $\infty$-functor}
  \label{weak infty functor}
\end{subsection}

For the weak Morse $\infty$-category $\mathcal{A}$ and the strict $\infty$-category$\mathcal{B}$, the construction of the functor $\mathcal{F}: \mathcal{A}\rightarrow \mathcal{B}$ is based on the boundary behaviour of the compactified moduli spaces of dimension one which have been studied thoroughly at the end of \S \ref{Moduli Space and its Compactification sect3}.

We firstly define a linear map between MSW complexes {corresponding to $(f^{\alpha},g^{\alpha})$ and $(f^{\beta},g^{\beta})$ respectively}
\begin{align}
\phi_{[\ell]}:=\Big\{\phi_{[\ell]r}: C_{r}\longrightarrow C'_{r+\ell} \Big| \phi_{[\ell]r}(p)
&=\sum\limits_{c'\in Crf^{\beta}\atop ind c'= ind p+\ell  } \sharp_{2}\mathcal{M}_{p,c'}^{H_{[\ell]}(\widehat{S}(p,c'),\cdot,\cdot)}\cdot c'   \nonumber \\
&=\sum\limits^k\limits_{i=1}\sum\limits_{c'\in Crf^{\beta}\atop ind c'= ind p+\ell } \sharp_{2}\mathcal{M}_{p,c'}^{H_{[\ell]}(\widehat{S}_{i}(p,c'),\cdot,\cdot)}\cdot c'   \nonumber \\
&=\sum\limits^k\limits_{i=1}\sum\limits_{c'\in Crf^{\beta};\ ind c'= ind p+\ell \atop \mathbf{\widehat{s}_{0}}\in \widehat{S}_{i}(p,c')} \sharp_{2}\mathcal{M}_{p,c'}^{H_{[\ell]}(\mathbf{\widehat{s}_{0}},\cdot,\cdot)}\cdot c',\  {r=0,\cdots, {dimM-\ell}  }\Big\}, \nonumber
\end{align}
where $\widehat{S}(p,c')=\bigcup \limits^k\limits_{i=1}\widehat{S}_{i}(p,c')$ is the codimension one non-generic set for some $k\in \Z_{>0}$ finite, and each $\widehat{S}_{i}(p,c')$ is a codimension one manifold. Here by $\sharp_{2}$ we mean taking the $\mathbb{Z}_{2}$ coefficient as usual.

{ For any fixed $p\in C_{r}$ and $\forall r=0,\cdots, dimM-\ell=n-\ell $, we have\\
\begin{equation}
\begin{split}
&(\phi_{[\ell-1]0,r}+\phi_{[\ell-1]1,r}+\phi_{[\ell]r-1}\partial_{r}+\partial_{r+\ell}'\phi_{[\ell]r})\cdot p\\
=&\big(\sum\limits_{q\in C'_{r+\ell-1}}\sharp_{2}\mathcal{M}^{H_{[\ell]}(0,s^{\underline{\ell-1}},\cdots,s^{\underline{1}},\cdot,\cdot)}_{p,q}
+ \sum\limits_{q\in C'_{r+\ell-1}}\sharp_{2}\mathcal{M}^{H_{[\ell]}(1,s^{\underline{\ell-1}},\cdots,s^{\underline{1}},\cdot,\cdot)}_{p,q}\\
 &+ \sum\limits_{c\in C_{r-1}} \sharp_{2}\mathcal{M}^{f^{\alpha}}_{p,c} \sum\limits_{q\in C'_{r+\ell-1}} \sharp_{2}\mathcal{M}^{H_{[\ell]}(\mathbf{\widehat{s}},\cdot,\cdot)}_{c,q}
+ \sum\limits_{c'\in C'_{r+\ell}} \sharp_{2}\mathcal{M}^{H_{[\ell]}(\mathbf{\widehat{s}},\cdot,\cdot)}_{p,c'}\sum\limits_{q\in C'_{r+\ell-1}} \sharp_{2}\mathcal{M}^{f^{\beta}}_{c',q})\cdot q\\
=&0.\nonumber
\end{split}
\end{equation} }
{For fixed $q$, the geometric meaning for
\begin{equation}
\begin{split}
&\sum\limits_{q\in C'_{r+\ell-1}}\sharp_{2}\mathcal{M}^{H_{[\ell]}(0,s^{\underline{\ell-1}},\cdots,s^{\underline{1}},\cdot,\cdot)}_{p,q}
+ \sum\limits_{q\in C'_{r+\ell-1}}\sharp_{2}\mathcal{M}^{H_{[\ell]}(1,s^{\underline{\ell-1}},\cdots,s^{\underline{1}},\cdot,\cdot)}_{p,q}\\
&+ \sum\limits_{c\in C_{r-1}} \sharp_{2}\mathcal{M}^{f^{\alpha}}_{p,c} \sum\limits_{q\in C'_{r+\ell-1}} \sharp_{2}\mathcal{M}^{H_{[\ell]}(\mathbf{\widehat{s}},\cdot,\cdot)}_{c,q}
+ \sum\limits_{c'\in C'_{r+\ell}} \sharp_{2}\mathcal{M}^{H_{[\ell]}(\mathbf{\widehat{s}},\cdot,\cdot)}_{p,c'}\sum\limits_{q\in C'_{r+\ell-1}} \sharp_{2}\mathcal{M}^{f^{\beta}}_{c',q}\nonumber
\end{split}
\end{equation}
is to count (in the sense of$\mod 2$) the boundary points $$\partial \overline{\mathcal{M}}_{p,q}^{H_{[\ell]}} \stackrel{(\ref{E ell})}{=}\mathcal{M}_{p,q}^{H_{[\ell]}(0,\cdots,\cdot,\cdot)}\cup \mathcal{M}_{p,q}^{H_{[\ell]}(1,\cdots,\cdot,\cdot)}
\cup\bigcup\limits_{c'\in Crf^{\beta} \atop ind c'= ind p+\ell}\mathcal{M}_{p,c'}^{H_{[\ell]}(\widehat{S}(p,c'),\cdot,\cdot)}\times \mathcal{M}_{c',q}^{f^{\beta}}
\cup\bigcup\limits_{c\in Crf^{\alpha} \atop  ind c= ind q-\ell }\mathcal{M}_{p,c}^{f^{\alpha}}\times \mathcal{M}_{c,q}^{H_{[\ell]}(\widehat{S}(c,q),\cdot,\cdot)}.$$
The last equality follows. }

The formula
$$\phi_{[\ell-1]0,r}+\phi_{[\ell-1]1,r}+\phi_{[\ell]r-1}\partial_{r}+\partial'_{r+\ell}\phi_{[\ell]r}=0,\ \forall r=0,\cdots, dimM-\ell=n-\ell, $$
can also be written as
$$\phi_{[\ell-1]1,r}-\phi_{[\ell-1]0,r}=(-1)^{\ell+1}\phi_{[\ell]r-1}\partial_{r}+\partial'_{r+\ell}\phi_{[\ell]r},\ \forall r=0,\cdots, dimM-\ell=n-\ell.$$
{The map $\phi_{[\ell]}$ is exactly the higher homotopy defined in \S \ref{strict n cate CHCx}!}



The construction of the functor $\mathcal{F}:\mathcal{A}\longrightarrow \mathcal{B}$ is as follows.    For $(H_{[\ell]},G_{[\ell]})\in \mathcal{A}$, we define the higher $\ell$-homotopy between $C:=(C(f^{\alpha},g^{\alpha}),\partial)$ and $C':=(C(f^{\beta},g^{\beta}),\partial')$ to be
$$\phi_{[\ell]}:=\mathcal{F}(H_{[\ell]},G_{[\ell]}),$$
such that, for any $r\in \Z_{\geqslant 0}$,
\begin{equation}
\begin{split}
\phi_{[\ell]r}:=\mathcal{F}(H_{[\ell]},G_{[\ell]}): C_{r} &\longrightarrow C'_{r+\ell} \\
 c_r &\mapsto \phi_{[\ell]r}(c_{r})= \sum\limits_{c'_{r+\ell}\in Crf^{\beta}\atop indc_{r}-indc'_{r+\ell}=-\ell}\sharp_{2}\Big(\mathcal{M}_{c_{r}\, c'_{r+\ell}}^{H_{[\ell]}(\widehat{S},\cdot)} \Big)\, c'_{r+\ell} \nonumber
\end{split}
\end{equation}
with
$\widehat{S}$ the non-generic codimension one set with respect to $c_{r}$, $c'_{r+\ell}$ and $H_{[\ell]}$.


Now we prove that the map $\mathcal{F}$ is indeed a weak $\infty$-functor, that is, it commutes with the composites and preserves the identity.

\begin{theorem}\label{inftyfunctor}
$\mathcal{F}:\mathcal{A}\longrightarrow \mathcal{B}$ is a weak $\infty$-functor.
\end{theorem}
\begin{proof}

\textbf{(1) Commuting with composites}

For any $\ell>1$, let $(H_{[\ell]},G_{[\ell]})_{1},(H_{[\ell]},G_{[\ell]})_{2}\in Y(\ell)$ with the associated codimension one sets $\widehat{S}_{1}$ ($\widehat{S}_{1}(p,c')$, $\widehat{S}_{1}(c,q)$) and $\widehat{S}_{2}$ ($\widehat{S}_{2}(p,c')$, $\widehat{S}_{2}(c,q)$) respectively. Besides, let  $\phi_{[\ell]1}=\mathcal{F}((H_{[\ell]},G_{[\ell]})_{1})$ and $\phi_{[\ell]2}=\mathcal{F}((H_{[\ell]},G_{[\ell]})_{2})$.

We need to check that $$\mathcal{F}((H_{[\ell]},G_{[\ell]})_{2}\circ (H_{[\ell]},G_{[\ell]})_{1})=\phi_{[\ell]1}+\phi_{[\ell]2}.$$
For fixed $H_{[\ell]}$, and any given critical points $c_{r}\in Crf^{\alpha}$ with  $indc_{r}=r$ and $c'_{r+\ell}\in Crf^{\beta}$ with  $indc'_{r+\ell}=r+\ell$, the key factor to obtain $\phi_{[\ell]}=\{\phi_{[\ell]r}\}_{r=0}^{n}$ is the codimension one set $\widehat{S}(c_{r},c'_{r+\ell})$.
By composing $(H_{[\ell]},G_{[\ell]})_{1}$ and $(H_{[\ell]},G_{[\ell]})_{2}$ with corresponding non-generic sets $\widehat{S}_{1}(c_{r},c'_{r+\ell})$ and $\widehat{S}_{2}(c_{r},c'_{r+\ell})$ respectively ($\forall r=0,1,\cdots,n$) along a $k$-cell,
we obtain the new codimension one set $$\widehat{S}(c_{r},c'_{r+\ell})=\widetilde{\widehat{S}}_{1}(c_{r},c'_{r+\ell})\cup \widetilde{\widehat{S}}_{2}(c_{r},c'_{r+\ell})\subset [0,1]^{\ell}$$ for $(H_{[\ell]},G_{[\ell]})_{2}\circ (H_{[\ell]},G_{[\ell]})_{1}$ {(Compare the schematic Figure \ref{gluealongdifferentcells})}, where
$$\widetilde{\widehat{S}}_{1}(c_{r},c'_{r+\ell})=\Big\{ \widetilde{\widehat{s}}_{1}=(\widetilde{s}_{[\ell]},\cdots,\widetilde{s}_{[1]})=(s_{[\ell]},\cdots,\frac{s_{[k+1]}}{2},\cdots,s_{[1]}) \Big|\widehat{s}_{1}=(s_{[\ell]},\cdots,s_{[1]})\in \widehat{S}_{1}\subset[0,1]^{\ell}  \Big\},$$
$$\widetilde{\widehat{S}}_{2}(c_{r},c'_{r+\ell})=\Big\{ \widetilde{\widehat{s}}_{2}=(\widetilde{s'}_{[\ell]},\cdots,\widetilde{s'}_{[1]})=(s'_{[\ell]},\cdots,\frac{s'_{[k+1]}+1}{2},\cdots,s'_{[1]}) \Big|\widehat{s'}_{1}=(s'_{[\ell]},\cdots,s'_{[1]})\in \widehat{S}_{2}\subset[0,1]^{\ell}  \Big\}.$$

\begin{figure}
  \centering
  \includegraphics[scale=0.5]{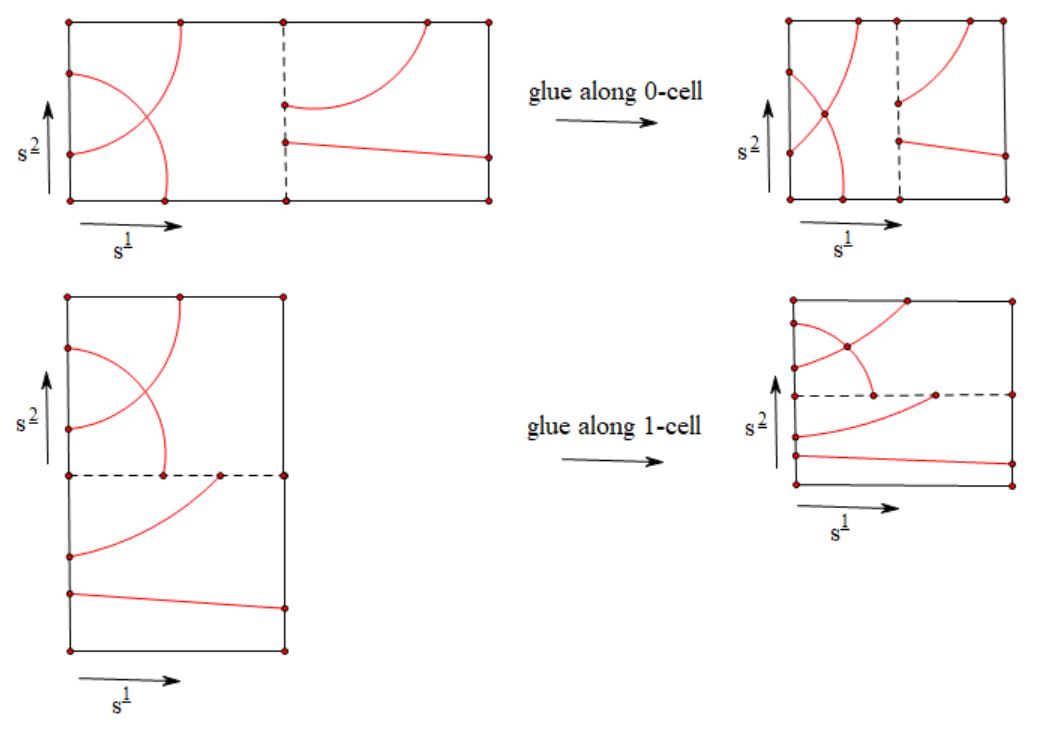}
  \caption{Gluing $(H_{[2]},G_{[2]})_{1}$ and $(H_{[2]},G_{[2]})_{2}$ along different cells}
  \label{gluealongdifferentcells}
\end{figure}
Consequently,
for any $p\in Crf^{\alpha}$,
\begin{equation}
\begin{split}
\mathcal{F}((H_{[\ell]},G_{[\ell]})_{2}\circ (H_{[\ell]},G_{[\ell]})_{1})p=:\phi_{[\ell]}p
&=\sum\limits_{i,j}\sum\limits_{c'\in Crf^{\beta}\atop indp-indc'=-l}\sharp_{2} (\mathcal{M}_{p,c'}^{H_{[\ell]2}\circ H_{[\ell]1}(\widetilde{\widehat{S}}_{1i}(p,c')\cup \widetilde{\widehat{S}}_{2j}(p,c'),\cdot)})\cdot c'\\
&= \sum\limits_{i}\sum\limits_{c'\in Crf^{\beta}\atop indp-indc'=-l}\sharp_{2} (\mathcal{M}_{p,c'}^{H_{[\ell]1}(\widehat{S}_{1i}(p,c'),\cdot)})\cdot c'+ \sum\limits_{j}\sum\limits_{c'\in Crf^{\beta}\atop indp-indc'=-l}\sharp_{2}(\mathcal{M}_{p,c'}^{H_{[\ell]2}(\widehat{S}_{2j}(p,c'),\cdot)})\cdot c'\\
&=(\phi_{[1]1}+\phi_{[1]2})p
=\mathcal{F}((H_{[\ell]},G_{[\ell]})_{2})p\circ' \mathcal{F}((H_{[\ell]},G_{[\ell]})_{1})p,\nonumber
\end{split}
\end{equation}
where  $\widehat{S}_{1}(p,c')=\bigcup\limits_{finite\ i} \widehat{S}_{1i}(p,c')$, $\widehat{S}_{2}(p,c')=\bigcup\limits_{finite \ j} \widehat{S}_{2j}(p,c')$.

\textbf{(2) Preserving identities}

For any $1<\ell\leqslant n=dimM$, given $(H_{[\ell]},G_{[\ell]})\in Y(\ell)$ with the associated non-generic codimension one set $\widehat{S}=\bigcup \widehat{S}_{i}$,
and the identity $(H_{[\ell+1]},G_{[\ell+1]})=\mathbf{1}_{(H_{[\ell]},G_{[\ell]})}$ with the non-generic codimension one set $\widehat{S}'=\widehat{S}\times[0,1]$. We check that $\mathcal{F}(\mathbf{1}_{(H_{[\ell]},G_{[\ell]})})=0.$
For brevity, we take the case $\ell=1$ as an example, and the conclusion holds for $(H_{[\ell]},G_{[\ell]})$ similarly.

{For critical points $p\in Crf^{\alpha},\  c',q\in Crf^{\beta}$ such that  $indp-indq=-1,\  indp-indc'=-2$, for $(H_{[1]},G_{[1]})$ and $(H_{[2]},G_{[2]}):=\mathbf{1}_{(H_{[1]},G_{[1]})}$,
\begin{equation}
\begin{split}
\mathcal{F}(H_{[2]},G_{[2]})p(=:\phi_{[2]}p)
&= \sum\limits_{c'\in Crf^{\beta}\atop indp-indc'=-2}\sum\limits_{(s_{00},s_{01})\in \widehat{S}_{[2]}(p,c')}\sharp_{2} (\mathcal{M}_{p,c'}^{H_{[2]}(s_{00},s_{01},\cdot)})\cdot c' \\
&= \sum\limits_{c'\in Crf^{\beta}\atop indp-indc'=-2}\sum\limits_{(s_{00},s_{01})\in \widehat{S}_{[2]}(p,c')}\sharp_{2} (\mathcal{M}_{p,c'}^{H_{[1]}(s_{01},\cdot)})\cdot c'
 =0,\nonumber
\end{split}
\end{equation}
where $\widehat{S}_{[2]}(p,c')$ is the non-generic codimension one set with respect to $p,c'$ and $H_{[2]}$. The second equality holds because of the definition of the identity $\mathbf{1}$: $H_{[2]}(s^{\underline{2}},s^{\underline{1}},\cdot)=\mathbf{1}_{H_{[1]}(s^{\underline{1}},\cdot)}$, which means that $\forall s_{2},\, H_{[2]}(s_{2},s^{\underline{1}},\cdot)= H_{[1]}(s^{\underline{1}},\cdot)$, therefore $H_{[2]}(s_{00},s_{01},\cdot)=H_{[2]}(\cdot,s_{01},\cdot)=H_{[1]}(s_{01},\cdot)$. 
The third one holds because $\mathcal{M}_{p,c'}^{H_{[1]}(s_{01},\cdot)}=\emptyset,$ for $dim(\mathcal{M}_{p,c'}^{H_{[1]}(s_{01},\cdot)})=indp-indc'+(1-1)+1=-2+0+1=-1$.  } 
\end{proof}

\begin{section}
  {Further Directions}
  \label{further direction sec5}
\end{section}

\textbf{Orientation of moduli space of gradient flow lines}

To drop out the $\Z_2$ counting and get the $\Z$ coefficients, the orientation of the moduli space of gradient flow lines with parameters is needed. One anticipates the coherent orientation arguments as developed in \cite{schwarz}  could be adapted to our case.

\textbf{Higher dimensional moduli space}

In the current paper, our main concern is the one dimensional compactified moduli space whose boundary structures provide us a higher algebraic structure generalizing the usual Morse-Smale-Witten complex. One may attempt to explore further the boundaries and corners of {\it higher} dimensional compactified moduli space.
Here "higher dimensional'' means the dimension of the moduli spcace $\geqslant 2$.  One can anticipate that the resulting structures on the boundaries and corners of compactified higher dimensional moduli space could reveal even richer algebraic structures. Here we are content with describing some boundaries (corners) of  {\it two} dimensional compactified moduli space, and leave the exploration about the corresponding algebraic structures in the future.

\textbf{$\mathbf{1}$-parameter case}\\
\indent Let $dim\mathcal{M}^{H_{[1]}}_{p,q}=ind p- ind q +1$ be equal to $2$ which implies $indp-indq=1$.
As shown in Figure \ref{a1b1} and Figure \ref{c1}, in the parameter space (a$1$) there are two types of special points which we distinguish by different colors: the two green points are the boundaries of the parameters space; and the red loci are non-generic points (i.e., the corresponding higher metrics are not generic); except  for these special points, the left are generic points. (b$1$) shows all possible types of boundaries (and corners).
The left picture in Figure \ref{c1} (c$1$) shows the natural boundaries which contains both boundaries and corners, and the right one displays the non-generic boundaries and corners.
\begin{figure}
  \centering
  \includegraphics[scale=0.5]{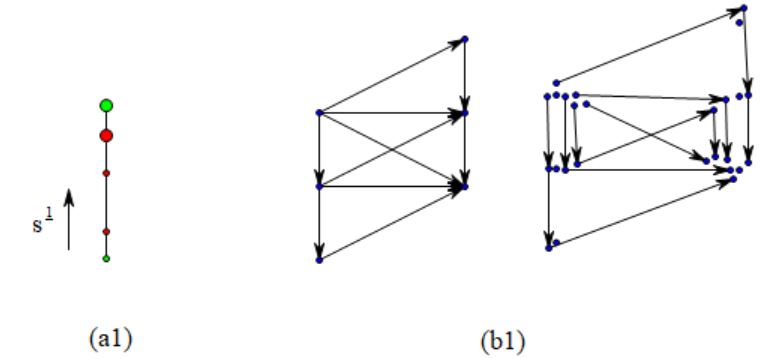}
  \caption{Parameter space (a$1$) and schematic boundary types (b$1$)}
  \label{a1b1}
\end{figure}
\begin{figure}
  \centering
  \includegraphics[scale=0.6]{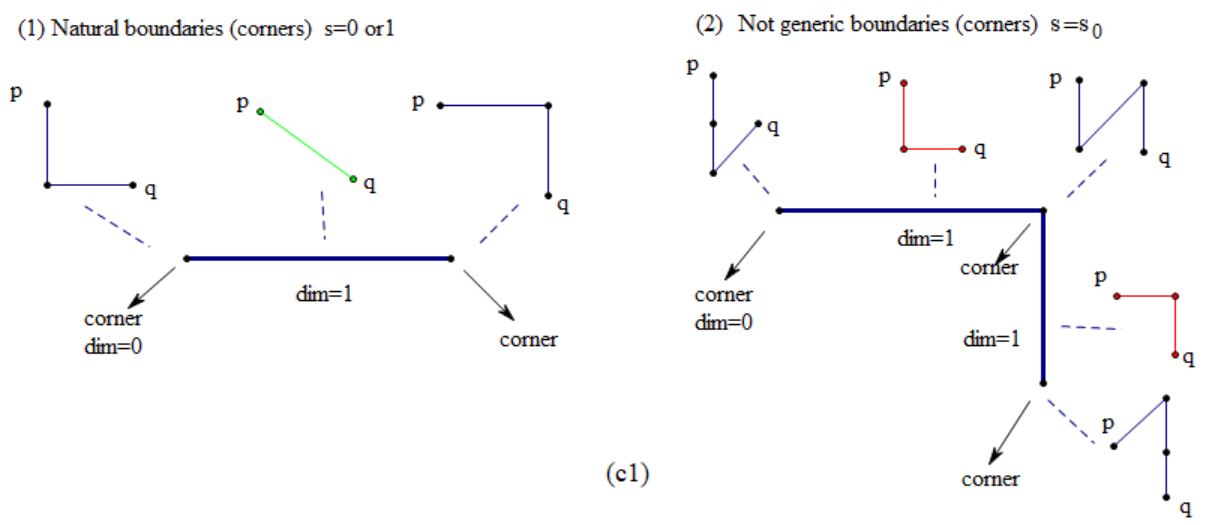}
  \caption{(c$1$) {Different types of the boundary with $indp-indq=1$ and the dimension of $\partial \overline{\mathcal{M}}_{p,q}^{H_{[1]}^{\alpha,\beta}}$ is one}}
  \label{c1}
\end{figure}

\textbf{$\mathbf{2}$-parameter case}\\
\indent Let $dim\mathcal{M}^{H_{[2]}}_{p,q}=ind p- ind q +2 $ be equal to $2$ which implies $indp-indq=0$.
As shown in Figure \ref{a2b2} and Figure \ref{c2}, in the parameter space (a$2$) there are several types of points: $(1)$ the generic points in the interior (blue points); $(2)$ generic points on the boundary (green points); $(3)$ non-generic points on the boundary (purple points); $(4)$ non-generic points of codimension $1$ (orange points); $(5)$ non-generic points of codimension $2$ (red points). (b$2$) reveals schematically all boundaries and corners, while (c$2$) shows some specific boundaries (corners) and their coalescence. One can see the clear correspondence between (a$2$) and (c$2$).
\begin{figure}
  \centering
  \includegraphics[scale=0.5]{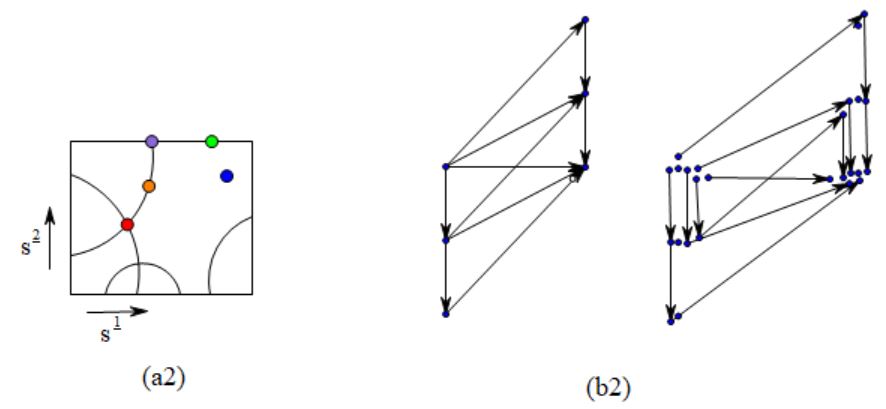}
  \caption{Parameter space (a$2$) and Schematic boundary types (b$2$)}
\label{a2b2}
\end{figure}
\begin{figure}
  \centering
  \includegraphics[scale=0.6]{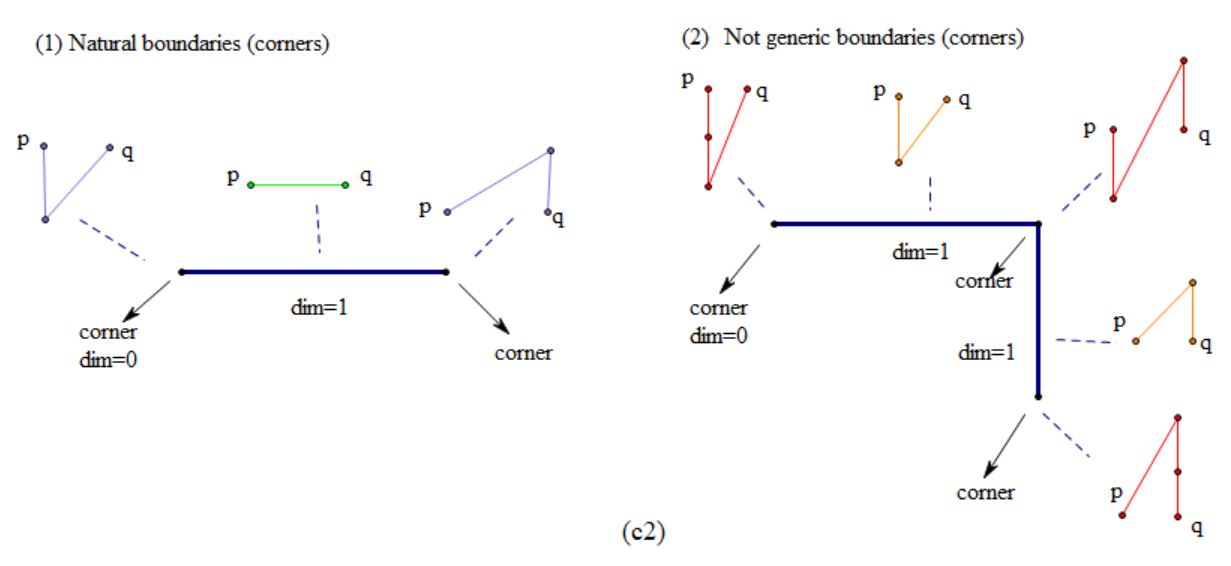}
  \caption{(c$2$) {Different types of the boundary with $indp-indq=0$ and the dimension of $\partial \overline{\mathcal{M}}_{p,q}^{H_{[2]}^{\alpha,\beta}}$ is one}}
\label{c2}\end{figure}

In this case, we are able to write down the boundary of the compactified moduli space:\\
$\partial \overline{\mathcal{M}}_{p,q}^{H_{[2]}}=\bigcup\limits_{(s_{1},s_{2})\in \partial [0,1]^{2}}\mathcal{M}_{p,q}^{H_{[2]}(s_{1},s_{2})}
\cup\bigcup\limits_{(\widehat{s}_{1},\widehat{s}_{2})\in \widehat{S}_{1}\ codimension\ 1\ of\ \partial [0,1]^{2} \atop c\in Cr f^{\alpha}\ indc=indp-1}\mathcal{M}_{p,c}^{f^{\alpha}}\times \mathcal{M}_{c,q}^{H_{[2]}(\widehat{s}_{1},\widehat{s}_{2})}\\
\cup\bigcup\limits_{(\widehat{s}_{1},\widehat{s}_{2})\in \widehat{S}'_{1}\ codimension\ 1\ of\ \partial [0,1]^{2} \atop c'\in Cr f^{\beta}\ indc'=indq+1}\mathcal{M}_{p,c'}^{H_{[2]}(\widehat{s}_{1},\widehat{s}_{2})}\times \mathcal{M}_{c',q}^{f^{\beta}}
\cup\bigcup\limits_{(\widehat{s}_{1},\widehat{s}_{2})\in \widehat{S}_{2}\ codimension\ 1\ of\ (0,1)^{2} \atop c\in Cr f^{\alpha}\ indc=indp-1}\mathcal{M}_{p,c}^{f^{\alpha}}\times \mathcal{M}_{c,q}^{H_{[2]}(\widehat{s}_{1},\widehat{s}_{2})}\\
\cup\bigcup\limits_{(\widehat{s}_{1},\widehat{s}_{2})\in \widehat{S}'_{2}\ codimension\ 1\ of\ (0,1)^{2} \atop c'\in Cr f^{\beta}\ indc'=indq+1}\mathcal{M}_{p,c'}^{H_{[2]}(\widehat{s}_{1},\widehat{s}_{2})}\times \mathcal{M}_{c',q}^{f^{\beta}}
\cup\bigcup\limits_{(\widehat{s}_{1},\widehat{s}_{2})\in \widehat{S}_{3}\ codimension\ 2\ of\ (0,1)^{2} \atop c_{1},c_{2}\in Cr f^{\alpha}\ indc_{1}=indp-1,\ indc_{2}=indc_{1}-1}\mathcal{M}_{p,c_{1}}^{f^{\alpha}}\times \mathcal{M}_{c_{1},c_{2}}^{f^{\alpha}}\times \mathcal{M}_{c_{2},q}^{H_{[2]}(\widehat{s}_{1},\widehat{s}_{2})}\\
\cup\bigcup\limits_{(\widehat{s}_{1},\widehat{s}_{2})\in \widehat{S}'_{3}\ codimension\ 2\ of\ (0,1)^{2} \atop c_{1}\in Cr f^{\alpha},\ c_{2}\in Cr f^{\beta}\ indc_{1}=indp-1,\ indc_{2}=indq+1}\mathcal{M}_{p,c_{1}}^{f^{\alpha}}\times \mathcal{M}_{c_{1},c_{2}}^{H_{[2]}(\widehat{s}_{1},\widehat{s}_{2})}\times \mathcal{M}_{c_{2},q}^{f^{\beta}}
\cup\bigcup\limits_{(\widehat{s}_{1},\widehat{s}_{2})\in \widehat{S}'_{3}\ codimension\ 2\ of\ (0,1)^{2} \atop c_{1},\ c_{2}\in Cr f^{\beta}\ indc_{1}=indq+2,\ indc_{2}=indq+1}\mathcal{M}_{p,c_{1}}^{H_{[2]}(\widehat{s}_{1},\widehat{s}_{2})}\times \mathcal{M}_{c_{1},c_{2}}^{f^{\beta}}\times \mathcal{M}_{c_{2},q}^{f^{\beta}}$.\\

\textbf{$\mathbf{3}$-parameter case}\\
\indent Let $dim\mathcal{M}^{H_{[3]}}_{p,q}=ind p- ind q +3 $ be equal to two, which implies $indp-indq=-1$.
As shown in Figure \ref{a3b3} and Figure \ref{c3}, in the parameter space (a$3$), there are several types of points: $(1)$ the generic points in the interior (blue points); $(2)$ generic points on the boundary (green points); $(3)$ non-generic points on the boundary (purple points); $(4)$ non-generic points of codimension $1$ (orange points); $(5)$ non-generic points of codimension $2$ (red points) and they are the intersections of two $2$-dimensional submanifolds; $(6)$ non-generic points of codimension $3$ (pink points) and they are the intersections of three $2$-dimensional submanifolds.
(b$3$) reveals all boundaries and corners, while (c$3$) shows the specific boundaries (corners) and the interrelations.
Note that there is no corresponding boundaries in (c$3$) to the pink point in the parameter space, which means that the parameter space contains more information that will never show up in algebraic structures.
\begin{figure}
  \centering
  \includegraphics[scale=0.5]{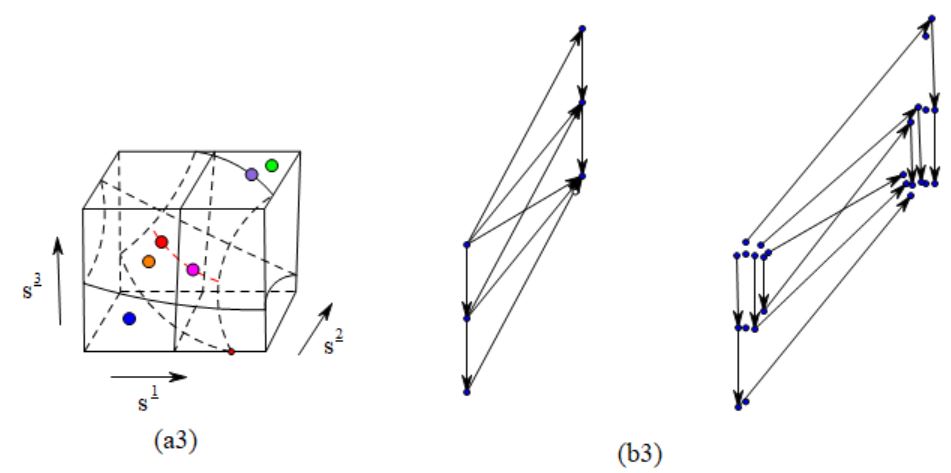}
  \caption{Parameter space (a$3$) and Scematic boundary types (b$3$)}
  \label{a3b3}
\end{figure}
\begin{figure}
  \centering
  \includegraphics[scale=0.5]{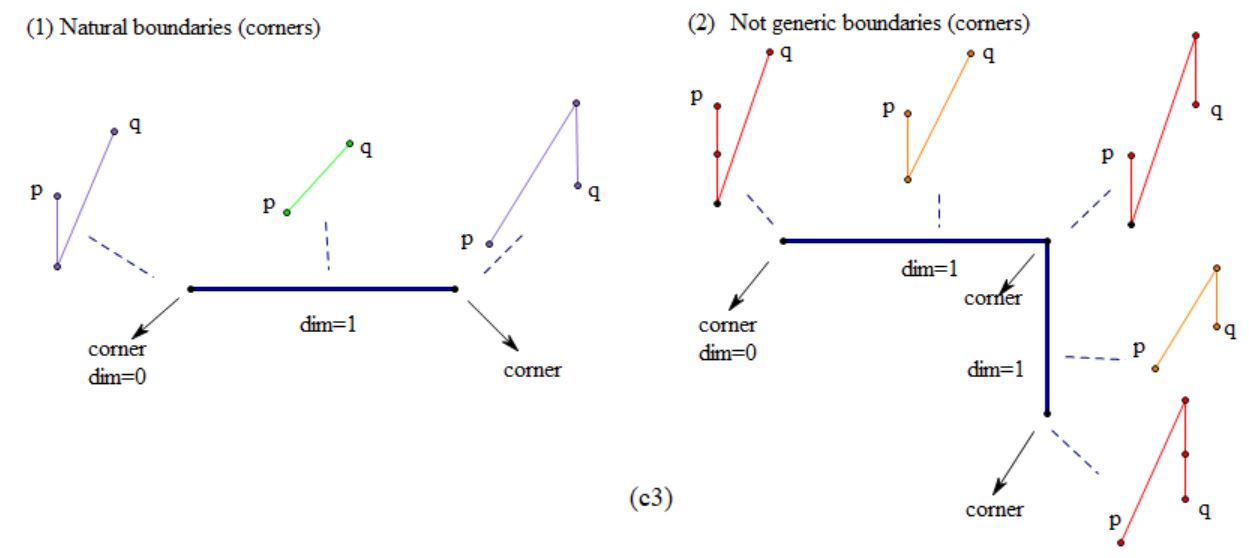}
  \caption{(c$3$) {Different types of the boundary with $indp-indq=-1$ and the $1$-dimensional boundary $\partial \overline{\mathcal{M}}_{p,q}^{H_{[3]}^{\alpha,\beta}}$}}
  \label{c3}
\end{figure}

{In fact, we can stop where the number of parameters is equal to the dimension of the moduli spaces. In other words, the space of parameters $(s^{\underline{\ell}},\cdots,s^{\underline{1}})$ has redundant information if $\ell$ is equal or greater than the dimension of the moduli space.}

\textbf{Generalization of the functor}

The construction of $\mathcal{A}$ and $\mathcal{B}$ in Section \ref{Weak infty category sect4} has a nice category structure, and the functor constructed via the boundary of the compactified moduli space between them gives us algebraic structure from geometry. All these constructions are based on the {\it fixed} MS pairs $(f^{\alpha},g^{\alpha})$, $(f^{\beta}, g^{\beta})$ and the {\it fixed} MSW complexes $C(f^{\alpha})$, $C(f^{\beta})$. One may wondering what kind of algebraic structures we could expect if we {\it vary} the MS pairs and regard them as some $0$-cells in a new category $\widetilde{\mathcal{A}}$. Accordingly, all MSW complexes should be taken into considerations, and form the objects in another new category $\widetilde{\mathcal{B}}$. We consider this as an interesting problem and still want to apply ideas used in the current paper.  Nevertheless there will be further issue to be addressed when we glue cells along a $0$-cell and extra data seems to come into play. This is the partial reason that we consider the categories $\mathcal{A}$ and $\mathcal{B}$. The role of codimension one sets as defects truly shows up.

{Here is an example to be pursued in the future.}

\indent Let $\phi_{[1]}=S:C\longrightarrow C'[-1]$ be a chain homotopy from $\phi_{0}$ to $\phi_{1}$, i.e., $\phi_{1}-\phi_{0}= S\circ \partial+\partial'\circ S$, and $\phi'_{[1]}=S':C'\longrightarrow C''[-1]$ be another chain homotopy from $\phi_{0}'$ to $\phi_{1}'$, that is, $\phi_{1}'-\phi_{0}'= S'\circ \partial'+\partial''\circ S'$, then $S'\circ S=S'\circ_{0} S:C\longrightarrow C''[-1]$ should also be a chain homotopy, i.e., a $2$-cell.
{When the codimension one set $\widehat{S}$ (resp. $\widehat{S}'$) has a unique point, the result coincides with the algebraic computations, namely the resulting homotopy has two possibilities: $\phi_{1}'\circ S+S'\circ \phi_{0}$ or $S'\circ \phi_{1}+\phi_{0}'\circ S$, just as Lienster (\cite{leinster} , p. 13) has shown.} When the codimension one set $\widehat{S}$ (resp. $\widehat{S}'$) has more than one components, the above description is not enough.
We take Figure \ref{glue0celltgeo} as an example to illustrate the general situation, and here the codimension one sets $\widehat{S}=\{\widehat{s}_{1},\widehat{s}_{2},\widehat{s}_{3},\widehat{s}_{4} \}$ and $\widehat{S}'=\{\widehat{s}'_{1},\widehat{s}'_{2}\}$.
\begin{figure}
\centering
\def\svgscale{0.6}
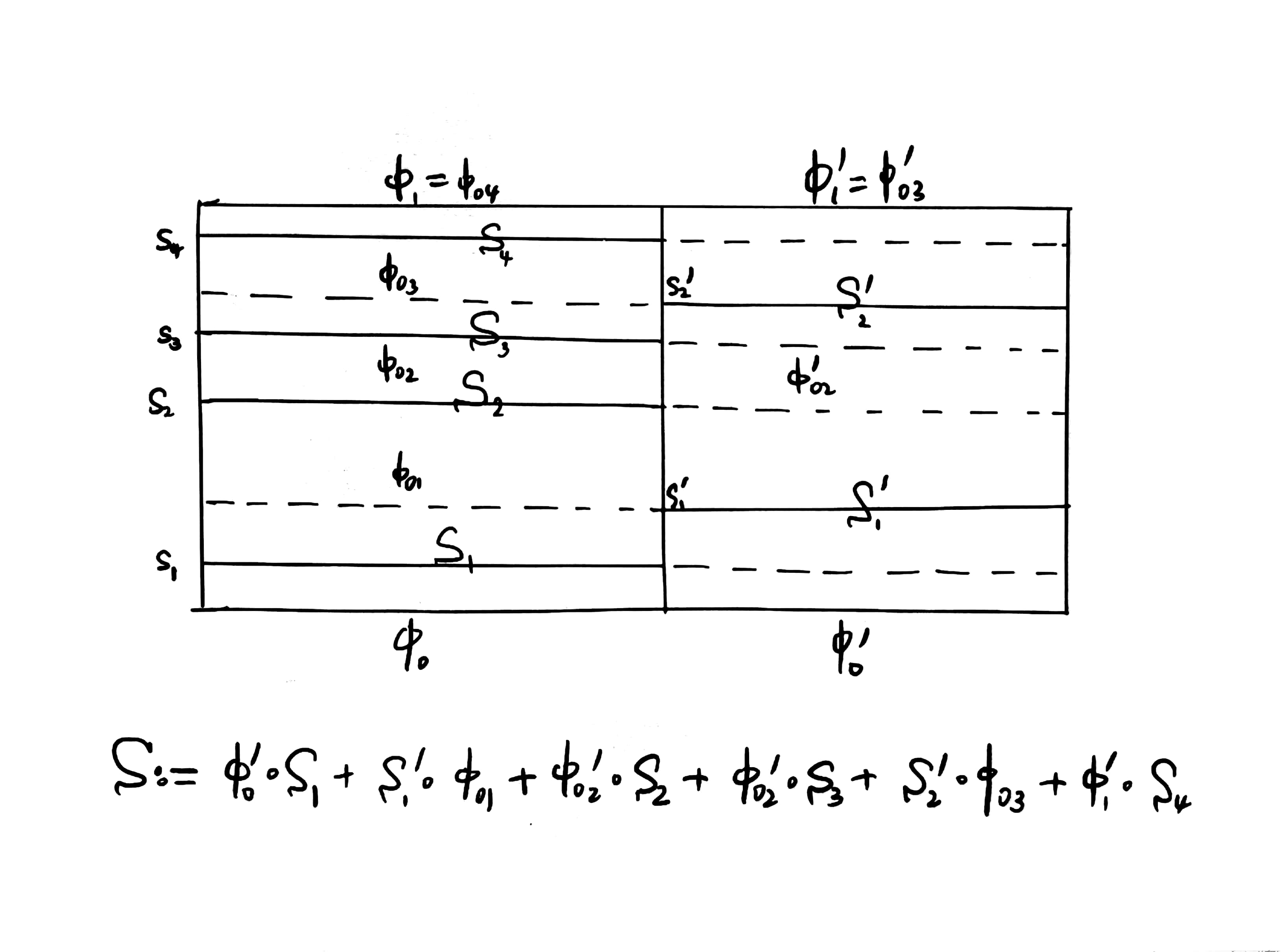
\caption{One possibility for the gluing in the presence of more codimension one sets}
\label{glue0celltgeo}
\end{figure}
One can easily see that the relative positions of the components in the codimension one sets $\widehat{S}$ and $\widehat{S}'$ are necessary data to define the homotopy from $\phi_{0}'\circ \phi_{0}$ to $\phi_{1}'\circ \phi_{1}$.
{When we try to work out the higher homotopies,  it becomes rather complicated since the codimension one set will be submanifolds instead of finitely many points. We will continue the projects in the sequel.}

{\bf Acknowledgements}

Our research is partially supported by National Key R\&D Program
of China (2020YFA0713300), NSFC (No.s 11771303, 12171327, 11911530092, 11871045). {Both authors are grateful to professor Urs FRAUENFELDER for valuable discussions and suggestions.} The second named author would like to thank the valuable discussions with Yong LI and Zhengyu ZONG.



\begin{thebibliography}{100}
\addtolength{\itemsep}{-1.2ex}


\bibitem{abouzaid11} M. Abouzaid, \emph{A topological model for the Fukaya categories of plumbings}, J. Differential Geom., {\bf 87} (2011), no.1, 1-80.

 \bibitem{Arnold-Gusein-Zade-Varchenko}  V. I. Arnold, S. M. Gusein-Zade and A. N. Varchenko, \emph{Singularities of differentiable maps}, I, II, Monographs in Mathematics, Birkhauser, Vols 82, 83, (1985).

 \bibitem{audin-damian} M. Audin and M. Damian, \emph{Morse theory and Floer homology}, Translated from the 2010 French original by Reinie Ern\'e, Universitext, Springer, London, EDP, Sciences, Les Ulis (2014).




 \bibitem{betz-cohen}M. Betz and R. L. Cohen, \emph{Graph moduli spaces and cohomology operations}, Turkish J. Math. {\bf 18} (1994), no. 1, 23-41.




 \bibitem{carqueville} N. Carqueville, \emph{Lecture notes on two-dimensional defect TQFT}, Advanced school on topological quantum field theory, 49-84, Banach Center Publ., {\bf 114}, Polish Acad. Sci. Inst. Math., Warsaw, 2018.(2016).






 \bibitem{eliasson} H. I. Eliasson, \emph{Geometry of manifolds of maps}, J. Differential Geom. {\bf 1} (1967), 169-194.

 \bibitem{frauenfelder-nicholls} U.\,Frauenfelder, R.\,Nicholls, \emph{The moduli space of gradient flow lines and Morse homology}, (2020), [arXiv:2005.10799].






 \bibitem{fukaya93} K. Fukaya, \emph{Morse homotopy, $A^\infty$-category and Floer homologies}, Proceedings of GARC Workshop on Geometry and Topology '93 (Seoul, 1993), 1-102, Lecture Notes Ser., {\bf 18}, Seoul, 1993.


 \bibitem{gaiotto-moore-witten} D.\, Gaiotto, G.\, W. Moore and E.\, Witten, \emph{Algebra of the Infrared: String Field Theoretic Structures in Massive ${\cal N}=(2,2)$ Field Theory In Two Dimensions}, (2015), [arXiv:1506.04087].



 \bibitem{Guillemin- Sternberg} V. Guillemin and S. Sternberg, \emph{Geometric asymptotics, Mathematical survey and monographs}, No. 14, AMS, (1977).



 \bibitem{hohloch} S. Hohloch, \emph{Higher Morse moduli spaces and n-categories}. Homology Homotopy and Applications, {\bf 16} (2017), no. 2, 1-32.



 \bibitem{hutchings} M.\,Hutchings, \emph{Lecture notes on Morse homology (with an eye towards Floer theory and pseudoholomorphic curves)}, (2002).


 \bibitem{kapustin} A. Kapustin, \emph{Topological Field Theory, Higher Categories, and Their Applications}, Proceedings of the International Congress of Mathematicians. Volume III, 2021-2043, Hindustan Book Agency, New Delhi, 2010.







 \bibitem{leinster} T. Leinster, \emph{Higher operads, higher categories}, London Mathematical Society Lecture Note Series 298, Cambridge University Press, (2004).









\bibitem{lurie18} J. Lurie and H. L. Tanaka, \emph{Associative algebras and broken lines}, arXiv: 1805.09587, (2018).

 \bibitem{lurie08} J. Lurie, \emph{On the classification of topological field theories}, Current developments in mathematics, 2008, 129-180, Int. Press, Somerville, MA, (2009).

 \bibitem{mazuir1} T. Mazuir, \emph{Higher algebra of $A_{\infty}$ and $\Omega B As$-algebras in Morse theory} I, arXiv: 2102.06654; II, arXiv: 2102.08996, (2021).


 \bibitem{mcduff-salamon} D. McDuff and D. Salamon, \emph{J-holomorphic curves and symplectic topology} , 2nd edition, Amer. Math. Soc., Providence, RI (2012).

 \bibitem{mescher18} S. Mescher, \emph{Perturbed gradient flow trees and $A_\infty$-algebra structures in Morse cohomology}, Vol. {\bf 6} of Atlantis Studies in Dynamical Systems. Atlantis Press, [Paris]; Springer, Cham, 2018.

 \bibitem{milnor} J.\,Milnor, \emph{Morse theory. Based on lecture notes by M.\,Spivak and R.\,Wells}, Annals of Mathematics Studies \textbf{51}, Princeton University Press, Princeton, N.J. (1963).



 \bibitem{nicolaescu} L.\,Nicolaescu, \emph{An invitation to Morse theory}, Second edition. Universitext. Springer, New York(2011).



 \bibitem{schwarz} M.\,Schwarz, \emph{Morse homology}, Progress in Mathematics \textbf{111}, Birkh\"auser Verlag, Basel,(1993).



 \bibitem{witten82} E. Witten, \emph{Supersymmetry and Morse theory},  J. Differential Geom., {\bf 17}(1982) no. 4, 661-692 (1983).














\end{thebibliography}
\end{document}